%% file: sep-FT-2.tex
\documentclass[a4paper,11pt]{article}

\usepackage[T1]{fontenc}
\usepackage[utf8]{inputenc}
\usepackage{type1ec}

\usepackage[colorlinks]{hyperref}
\usepackage[french]{babel}
\usepackage{xcolor}
\usepackage{fullpage}
\usepackage{imakeidx}

\hypersetup{
	linkcolor=blue,
	citecolor=red,
	urlcolor=teal,
	pdftitle = {La formule des traces pour les revêtements de groupes réductifs connexes. II. Analyse harmonique locale},
	pdfauthor = {Wen-Wei Li}
}

\input{Meta-traces}

\title{La formule des traces pour les revêtements de groupes réductifs connexes. II. \\ Analyse harmonique locale}
\author{Wen-Wei Li}
\date{}

\makeindex[name=iFT2,title=Index,columns=3]
\begin{document}

\maketitle


\begin{abstract}
  On établit des résultats de l'analyse harmonique locale nécessaires pour la formule des traces invariante d'Arthur pour les revêtements de groupes réductifs connexes. Plus précisément, on démontre pour les revêtements locaux (1) la formule de Plancherel et des préparatifs reliés, (2) la normalisation des opérateurs d'entrelacement soumise aux conditions d'Arthur, (3) le comportement local de caractères de représentations admissibles dans le cas non archimédien, et (4) la partie spécifique de la formule des traces locale invariante. Comme un sous-produit de la démonstration de la formule des traces locale invariante, on obtient aussi la densité de caractères tempérés pour les revêtements.
\end{abstract}

\begin{flushleft}
  \small MSC classification (2010): \textbf{11F72}, 11F70.
\end{flushleft}

\tableofcontents

\input{FT-2}

\bibliographystyle{abbrv-fr}
\bibliography{metaplectic}

\printindex[iFT2]

\begin{flushleft}
  Wen-Wei Li \\
  Morningside Center of Mathematics, \\
  Academy of Mathematics and Systems Science, Chinese Academy of Sciences, \\
  55, Zhongguancun East Road, \\
  100190 Beijing, China. \\
  Adresse électronique: \texttt{wwli@math.ac.cn}
\end{flushleft}

\end{document}

%% file: Meta-traces.tex
\usepackage{bbm}
\usepackage[bbgreekl]{mathbbol}
\usepackage{nccrules}
\usepackage{tocbibind}
\usepackage[intlimits,leqno]{amsmath}
\usepackage{amsthm}
\usepackage{amssymb}
\usepackage{mathrsfs}
\usepackage{stmaryrd}
\usepackage{xspace}
\usepackage[all]{xy}
\newdir{ >}{{}*!/-10pt/\dir{>}}
\newdir{|>}{!/4.5pt/@{|}*:(1,-.2)@^{>}*:(1,+.2)@_{>}}

\newcommand{\Z}{\ensuremath{\mathbb{Z}}}
\newcommand{\Q}{\ensuremath{\mathbb{Q}}}
\newcommand{\R}{\ensuremath{\mathbb{R}}}
\newcommand{\C}{\ensuremath{\mathbb{C}}}
\newcommand{\F}{\ensuremath{\mathbb{F}}}
\newcommand{\A}{\ensuremath{\mathbb{A}}}

\newcommand{\Aut}{\ensuremath{\mathrm{Aut}}}
\newcommand{\Tr}{\ensuremath{\mathrm{tr}\,}}

\newcommand{\dd}{\ensuremath{\,\mathrm{d}}}

\newcommand{\angles}[1]{\ensuremath{\langle #1 \rangle}}
\newcommand{\mes}{\ensuremath{\mathrm{mes}}}

\newcommand{\identity}{\ensuremath{\mathrm{id}}}

\newcommand{\Hom}{\ensuremath{\mathrm{Hom}}}
\newcommand{\End}{\ensuremath{\mathrm{End}}}
\newcommand{\rightiso}{\ensuremath{\stackrel{\sim}{\rightarrow}}}

\newcommand{\Ker}{\ensuremath{\mathrm{Ker}\,}}
\newcommand{\Coker}{\ensuremath{\mathrm{Coker}\,}}

\newcommand{\Ad}{\ensuremath{\mathrm{Ad}\,}}
\newcommand{\ad}{\ensuremath{\mathrm{ad}\,}}

\newcommand{\Gm}{\ensuremath{\mathbb{G}_\mathrm{m}}}

\newcommand{\Supp}{\ensuremath{\mathrm{Supp}}}

\newcommand{\SL}{\ensuremath{\mathrm{SL}}}

\newcommand{\Exp}{\ensuremath{\mathcal{E}\mathrm{xp}}}

\theoremstyle{plain}
\newtheorem{proposition}{Proposition}[subsection]
\newtheorem{lemma}[proposition]{Lemme}
\newtheorem{theorem}[proposition]{Théorème}
\newtheorem{corollary}[proposition]{Corollaire}

\theoremstyle{definition}
\newtheorem{definition}[proposition]{Définition}
\newtheorem{definition-theorem}[proposition]{Définition-Théorème}
\newtheorem{definition-proposition}[proposition]{Définition-Proposition}
\newtheorem{hypothesis}[proposition]{Hypothèse}

\newtheorem{remark}[proposition]{Remarque}

\newcommand{\bmu}{\ensuremath{\bbmu}}
\newcommand{\bomega}{{\ensuremath{\boldsymbol{\omega}}}}
\newcommand{\noyau}{\ensuremath{\boldsymbol{\varepsilon}}} 
\newcommand{\rev}{\ensuremath{\mathbf{p}}} 
\newcommand{\asp}{\ensuremath{\dashrule[.7ex]{2 2 2 2}{.4}}} 

\renewcommand{\Re}{\ensuremath{\mathrm{Re}\xspace}}
\renewcommand{\Im}{\ensuremath{\mathrm{Im}\xspace}}



%% file: FT-2.tex
\section{Introduction}
Cet article fait suite à \cite{Li10a} qui vise à établir une formule des traces invariante à la Arthur \cite{Ar88-TF2} pour une grande classe de revêtements des groupes réductifs connexes. Afin d'arriver à la formule des traces invariante, Arthur a besoin des résultats profonds d'analyse harmonique locale. Tandis qu'ils sont admis par certains mathématiciens, il s'avère que les modifications nécessaires ne sont pas toujours triviales. Notre objet est donc d'établir, ou plutôt de justifier, de tels résultats. Plus précisément, nous visons à établir
\begin{itemize}
  \item la formule de Plancherel \cite{Wa03};
  \item la normalisation des opérateurs d'entrelacement \cite{Ar89-IOR1}, avec formules explicites dans le cas archimédien;
  \item la régularité et le développement local de caractères dans le cas non archimédien \cite{HC99};
  \item le Théorème de Paley-Wiener invariant pour les fonctions de Schwartz-Harish-Chandra \cite{Ar94-PW}, qui est un ingrédient crucial pour mettre la formule des trace locale en une forme invariante;
  \item la formule des traces locale invariante \cite{Ar91}.
  \item le théorème à la Kazhdan de la densité de caractères tempérés dans le cas non archimédien \cite{Ka86}.
\end{itemize}

Le Graal de cet article est la formule des traces locale invariante sous la forme présentée dans \cite[Proposition 6.1]{Ar03}:
$$ I_\text{disc}(f) = \sum_{M \in \mathcal{L}(M_0)} |W^M_0| |W^G_0|^{-1} (-1)^{\dim A_M/A_G} \int_{\Gamma_{G-\text{reg,ell}}(M(F))^\text{bon}} I_{\tilde{M}}(\gamma, f) \dd\gamma $$
pour tout $f \in \mathcal{C}_{\asp}(\tilde{G} \times \tilde{G})$, où
\begin{itemize}
  \item $\mathcal{C}_{\asp}(\tilde{G} \times \tilde{G})$ est la ``composante anti-spécifique'' de l'espace des fonctions de Schwartz-Harish-Chandra sur $\tilde{G} \times \tilde{G}$;
  \item $I_\text{disc}(\cdot)$ est la ``partie discrète'' du côté spectral spécifique de la formule des traces locale non invariante;
  \item $I_{\tilde{M}}(\gamma, \cdot)$ est la distribution invariante fabriquée des intégrales orbitales pondérées et des caractères pondérés.  C'est l'intégrale orbitale usuelle lorsque $M=G$. Cette distribution est reliée aux distributions apparaissant dans la formule des traces globale par une formule de scindage (le Lemme \ref{prop:I-scindage}).
\end{itemize}
Voir \S\ref{sec:formule-traces} pour les détails. Il peut paraître curieux car la formule des traces locale n'est pas un ingrédient dans la preuve originelle d'Arthur de la formule des traces globale. Néanmoins, Arthur fait usage d'un argument global pour compléter son argument de récurrence (voir \cite[\S 5]{Ar88-TF2}), ce qui est problématique sur les revêtements car il faudrait une propriété d'approximation faible des bons éléments. Un cas particulier du résultat d'Arthur est la densité de caractères tempérés due à Kazhdan \cite{Ka86} qui est indispensable dans toute application de la formule des traces. Donc il faut les établir à tout prix. Ici la formule des traces locale en fournit une approche contournée mais purement locale, cf. \cite[Corollary 5.3]{Ar93}. D'ailleurs, la formule des traces locale interviendra dans toute étude sérieuse d'analyse harmonique locale, e.g. la théorie de représentations tempérées elliptiques; elle est également utilisée dans le travail de Mezo \cite[\S 3]{MZ01} sur la 
correspondance métaplectique.

Vu les travaux de Harish-Chandra, il est tentant de penser que les théories archimédiennes s'adaptent aux revêtements sans peine. Bien au contraire! Par exemple, le théorème de Paley-Wiener invariant pour les fonctions de Schwartz-Harish-Chandra fait l'usage de la théorie de $K$-types minimaux de Vogan, notamment la multiplicité $1$, ce qui dépend des hypothèses d'algébricité ou de connexité du groupe de Lie en question.

C'est inévitable d'admettre les généralisations aux revêtements de certains résultats standards d'analyse harmonique dans cet article, à savoir:
\begin{itemize}
  \item la théorie de décomposition de Bernstein, notamment le lemme géométrique \cite{BZ77};
  \item la classification des représentations tempérées en termes des représentations de carré intégrable modulo le centre dans le cas non archimédien \cite{Wa03};
  \item la classification de Langlands des représentations admissibles dans le cas non archimédien \cite{Kon03};
  \item la théorie de $R$-groupes, cf. \cite{Si78}.
\end{itemize}
On donnera des justifications pour le lemme géométrique dans \S\ref{sec:representations}.

\paragraph{Organisation de cet article}
Dans le \S 2, nous recueillons des définitions de base pour l'analyse harmonique locale non archimédienne pour les revêtements et nous établissons la théorie de Harish-Chandra pour les revêtements: fonctions de Schwartz-Harish-Chandra, représentations tempérées, opérateurs d'entrelacement, fonction $c$ et $\mu$, la formule de Plancherel. Les preuves sont plus ou moins identiques au cas des groupes réductifs connexes et nous adoptons les notations de \cite{Wa03}; le but de cette section est plutôt de fixer les notations et les choix de mesures. Une exception: il faut choisir un sous-groupe ouvert d'indice fini $A_M(F)^\dagger$ de $A_M(F)$, où $M$ est un sous-groupe de Lévi semi-standard de $G$ (voir la Proposition \ref{prop:dagger}). Afin de compenser cette ambiguïté, les intégrales concernant $A_M(F)^\dagger$ sont multipliées par le facteur $\iota_M$ défini dans \eqref{eqn:iota}.

Dans le \S 3, nous étudions la normalisation des opérateurs d'entrelacement satisfaisant aux conditions d'Arthur, cf. \cite[\S 2]{Ar89-IOR1} et \cite[\S 2]{Ar94}. Nous en donnons des formules explicites similaires à celles d'Arthur dans le cas archimédien. Pour le cas non archimédien nous reprenons l'argument dans \cite[Lecture 15]{CLL} à quelques corrections près. Enfin, il faut aussi considérer le cas des revêtements non ramifiées et nous le faisons à l'aide de la théorie de séries principales non ramifiés spécifiques. Pour les revêtements archimédiens à deux feuillets, par exemple ceux provenant des $\mathbf{K}_2$-extensions de Brylinski-Deligne \cite{BD01}, nos facteurs normalisants sont liés à ceux d'Arthurs pour les $\R$-groupes réductifs connexes qui s'expriment en termes de fonctions $L$ archimédiennes.

Dans le \S 4, nous reprenons \cite{HC99}. La méthode de descente semi-simple nous ramène à l'algèbre de Lie, où le revêtement disparaît mais un caractère intervient. Le même phénomène paraît aussi dans \cite{Li10a}. On peut se débarrasser du caractère en travaillant systématiquement avec le module croisé $[G_\text{SC} \times Z_{G^\circ}^\circ \to G]$. Remarquons que l'intégrabilité locale des caractères admissibles irréductibles pour les revêtements a déjà apparu comme une hypothèse dans des travaux sur la correspondance de Howe; c'est un peu surprenant qu'une preuve n'est jamais entamée auparavant.

Il paraît que notre méthode permettrait de montrer le théorème de régularité de Harish-Chandra au niveau du groupe, ainsi que le développement local, dans le cadre tordu par un caractère $\bomega$. C'est la situation rencontrée en l'endoscopie tordue. Nous ne poursuivons pas ce problème dans cet article.

Dans le \S 5, nous établissons la formule des traces locale d'Arthur. Les arguments sont presque identiques à ceux d'Arthur et nous n'en donnons que des esquisses; les intégrales d'Eisentein sont évitées dans notre récit. Comme le formalisme d'Arthur évolue, la tâche est plutôt de faire la mise à jour. Notre référence est \cite[\S 6]{Ar03}; en particulier, les fonctions test sont dans l'espace de Schwartz-Harish-Chandra et les distributions sont rendues invariantes à l'aide du Théorème de Paley-Wiener invariant pour de telles fonctions, dont la démonstration dans le cas archimédien est une variante de celle de \cite{Ar94-PW}. Nous démontrons le ``Théorème 0'' de Kazhdan dans le Théorème \ref{prop:KZ0}. Cette section est quelque peu formelle à l'exception du Théorème de Paley-Wiener invariant, qui nécessite des constructions combinatoires techniques.

Selon Arthur, les caractères pondérés peuvent être non normalisés (ceux dans la formule des traces locale non invariante), normalisés relativement à un choix de facteurs normalisants, ou canoniquement normalisés à l'aide de fonctions $\mu$. Dans l'étude de la formule des traces locale invariante, il nous faut utiliser et comparer tous les trois. Néanmoins, dans les énoncés finaux les caractères pondérés canoniquement normalisés sont toujours préférés.

Récapitulons le rapport entre les résultats principaux comme suit.
$$\xymatrix{
  & \txt{La formule de Plancherel} \ar@{-|>}[ld] \ar@{-|>}[rd] \ar@{-|>}[dd] & \\
  \txt{Normalisation des \\ opérateurs d'entrelacement} \ar@{-|>}[rd] \ar@{-|>}'[r] [rr] & & \txt{Théorème de \\ Paley-Wiener invariant} \ar@{-|>}[ld] \\
  & \txt{La formule des traces \\ locale invariante} \ar@{-|>}[d] & \\
  \txt{Régularité des \\ caractères} \ar@{-|>}[r] & \txt{Densité des\\caractères tempérés} &
}$$

Les corps locaux dans \S\S 4-5 sont supposés de caractéristique nulle.

\paragraph{Notations générales} Sauf mention expresse du contraire, les notations seront celles de \cite{Li10a}, qui sont compatibles avec celles d'Arthur. Si $V$ est un espace vectoriel, son dual est noté $V^\vee$ et l'accouplement entre eux est désigné par $\angles{\check{v}, v}$ ou $\check{v}(v)$.

\paragraph{Remerciements}
Je remercie très vivement J.-L. Waldpsurger pour ses remarques sur le manuscrit, ainsi que T. Ikeda, P. McNamara et P. Mezo pour des conversations utiles. Je remercie aussi le referee pour ses conseils très pertinents.

\section{La formule de Plancherel}\label{sec:Plancherel}
\subsection{Définitions de base}\label{sec:def-base}
On vise à établir la formule de Plancherel pour les revêtements locaux. Le cas archimédien est déjà traité dans les travaux de Harish-Chandra \cite{HC76}. On se limite au cas où $F$ est un corps local non archimédien. Notons $q$ le cardinal du corps résiduel de $F$. Nous suivrons l'approche de Waldspurger \cite{Wa03} en mettant l'accent sur les énoncés et les changements nécessaires pour adapter ses preuves aux revêtements.

Soient $m \in \Z$, $m \geq 1$ et $G$ un $F$-groupe réductif connexe. Considérons un revêtement à $m$ feuillets
$$ 1 \to \bmu_m \to \tilde{G} \stackrel{\rev}{\to} G(F) \to 1 . $$

Nous adoptons les conventions de \cite{Li10a}: les objets associés au revêtement sont affectés de $\sim$, e.g. $\tilde{P}, \tilde{x}$; les symboles sans $\sim$ désignent leurs images par $\rev$. Une décomposition de Lévi d'un parabolique $P$ s'écrit toujours de la forme $P=MU$. On construira des objets associés à $\tilde{G}$ ainsi qu'à ses sous-groupes de Lévi; on fera référence au groupe en question en affectant les notations d'exposant lorsqu'il convient de le faire.

Fixons un sous-groupe de Lévi minimal $M_0$ de $G$, un sous-groupe compact maximal spécial $K$ de $G(F)$ en bonne position relativement à $M_0$, et $P_0 \in \mathcal{P}(M_0)$. On définit ainsi les sous-groupes de Lévi ou paraboliques standards et semi-standards. La proposition suivante est aussi valable pour le cas $F$ archimédien.

\begin{proposition}\label{prop:dagger}\index[iFT2]{$A_M(F)^\dagger$}
   Soient $F$ un corps local, $\rev: \tilde{G} \to G(F)$ un revêtement à $m$ feuillets. Il existe une famille de sous-groupes $A_M(F)^\dagger$ de $A_M(F)$, où $M$ parcourt les sous-groupes de Lévi de $G$, telle que
  \begin{enumerate}
    \item $A_M(F)^\dagger$ est ouvert et fermé dans $A_M(F)$ d'indice fini;
    \item $\widetilde{A_M}^\dagger := \rev^{-1}(A_M(F)^\dagger)$ est central dans $\tilde{M}$;
    \item si $M_1, M_2$ sont des sous-groupes de Lévi et $y M_1 y^{-1} = M_2$ où $y \in G(F)$, alors $y A_{M_1}(F)^\dagger y^{-1} = A_{M_2}(F)^\dagger$;
    \item soit $L$ un Lévi contenant $M$, alors $A_L(F)^\dagger \subset A_M(F)^\dagger$;
    \item on pose\index[iFT2]{$\mathfrak{a}_{M,F}, \tilde{\mathfrak{a}}_{M,F}, \tilde{\mathfrak{a}}_{M,F}^\dagger$}
      \begin{align*}
        \mathfrak{a}_{M,F} & := H_M(M(F)), \\
        \tilde{\mathfrak{a}}_{M,F} & := H_M(A_M(F)), \\
        \tilde{\mathfrak{a}}_{M,F}^\dagger & := H_M(A_M(F)^\dagger),
      \end{align*}
      alors $\tilde{\mathfrak{a}}_{L,F}^\dagger = \tilde{\mathfrak{a}}_{M,F}^\dagger \cap \mathfrak{a}_L$ pour tout $L \supset M$.
  \end{enumerate}
\end{proposition}
\begin{proof}
  Indiquons une construction possible: posons $A_M(F)^\dagger := A_M(F)^m$ pour tout Lévi $M$. Alors on a $\tilde{\mathfrak{a}}_{M,F}^\dagger = m \cdot \tilde{\mathfrak{a}}_{M,F}$. Soit $L \supset M$, la propriété $\tilde{\mathfrak{a}}_{L,F}^\dagger = \tilde{\mathfrak{a}}_{M,F}^\dagger \cap \mathfrak{a}_L$ résulte du fait que $\tilde{\mathfrak{a}}_{L,F} = \tilde{\mathfrak{a}}_{M,F} \cap \mathfrak{a}_L$. Les autres propriétés sont triviales.
\end{proof}

En fait, on a déjà utilisé une version moins précise de ces sous-groupes dans \cite{Li10a}. Fixons une telle famille désormais. Lorsque $M = M_0$, on utilise les notations $A_0(F)$, $A_0(F)^\dagger$ et $\widetilde{A_0}^\dagger$.

Rappelons que dans \cite{Li10a} on a défini l'application de Harish-Chandra\index[iFT2]{$H_{\tilde{G}}$} $H_{\tilde{G}}: \tilde{G} \to \mathfrak{a}_G$ par $H_{\tilde{G}} = H_G \circ \rev$ et $\tilde{G}^1 := \Ker(H_{\tilde{G}}) = \rev^{-1}(G(F)^1)$\index[iFT2]{$\tilde{G}^1$}. Elle diffère de celle de \cite{Wa03} par un signe, mais peu importe. Rappelons que, lorsque $F$ est non archimédien, on choisit la mesure de $A_G(F)$ de sorte que $\mes(A_G(F) \cap \Ker H_G) = 1$. On munit $A_G(F)^\dagger$ de la mesure induite de $A_G(F)$.

Soit $M$ un Lévi de $G$. Étant donné $A_M(F)^\dagger$, on définit\index[iFT2]{$\iota_M$}
\begin{gather}\label{eqn:iota}
  \iota_M := [A_M(F):A_M(F)^\dagger]^{-1},
\end{gather}
alors pour toute fonction $\varphi \in C_c^\infty(G(F))$ invariante par $A_M(F)^\dagger$, l'intégrale
$$\iota_M \int_{A_M(F)^\dagger \backslash G(F)} \varphi(x) \dd x $$
ne change pas si l'on remplace $A_M(F)^\dagger$ par un sous-groupe d'indice fini $A_M(F)^\ddagger$. Si $\varphi$ est invariante par $A_M(F)$, alors cette l'intégrale est égale à $\int_{A_M(F) \backslash G(F)} \varphi(x) \dd x$.

On pose\index[iFT2]{$X(\tilde{G}), \Im X(\tilde{G})$}
$$ X(\tilde{G}) := \Hom(\tilde{G}/\tilde{G}^1, \C^\times). $$
Selon nos définitions, $X(\tilde{G})$ s'identifie à $X(G)$. Il est muni d'une structure de variété algébrique complexe, isomorphe à $(\C^\times)^{\dim_{\R} \mathfrak{a}_G}$, qui se déduit de l'homomorphisme surjectif
$$ \mathfrak{a}_{G,\C}^* \to X(G) $$
qui envoie $\chi \otimes s \in \mathfrak{a}_{G,\C}^*$ sur $|\chi(\cdot)|^s$; rappelons que $\mathfrak{a}_{G,\C}^* = X^*(G) \otimes_\Z \C$ où $X^*(G) := \Hom(G,\Gm)$. Son noyau est de la forme $\frac{2\pi i}{\log q} L$ où $L$ est un réseau dans $X^*(G) \otimes_\Z \Q$. Cela permet de définir la partie réelle $\Re(\chi)$ pour tout $\chi \in X(G)$. Notons $\Im X(G) := \{\chi \in X(G) : \Re(\chi)=0 \}$. On définit ainsi $\Im X(\tilde{G}) = \Im X(G)$.

Soit $L$ un sous-groupe de $\mathfrak{a}_G$, on pose $L^\vee := \Hom(L, 2\pi i \Z) \subset i\mathfrak{a}_G^*$\index[iFT2]{$L^\vee$}. Alors
$$ \Im X(\tilde{G}) = i\mathfrak{a}_G^* / \mathfrak{a}_{G,F}^\vee . $$
C'est un tore compact comme $F$ est non archimédien. Introduisons des mesures de Haar et des constantes comme suit.
\begin{itemize}
  \item On munit $\Im X(\tilde{G})$ de la mesure de Haar de sorte que l'application quotient
    $$ i\mathfrak{a}_G^* / \mathfrak{a}_{G,F}^\vee \to i\mathfrak{a}_G^* / \tilde{\mathfrak{a}}_{G,F}^\vee $$
    préserve localement les mesures et que $\mes(i\mathfrak{a}_G^* / \tilde{\mathfrak{a}}_{G,F}^\vee)=1$. Ceci est compatible avec la convention dans \cite[\S 2.5]{Li10a}.
  \item Pour tout $M \in \mathcal{L}(M_0)$, choisissons la mesure de Haar sur $M(F)$ telle que $\mes(K \cap M(F))=1$. De tels choix sont invariants par $W^G_0$.
  \item On munit $\tilde{M}$ de la mesure de Haar telle que $\mes(\rev^{-1}(E))=\mes(E)$ pour tout $E \subset M(F)$ mesurable. La même convention s'applique aux groupes $\tilde{K}$, $\tilde{K} \cap \tilde{M}$ et $\widetilde{A_M}^\dagger$.
  \item Soit $P=MU \in \mathcal{F}(M_0)$, on munit $U(F)$ de la mesure de Haar telle que $\mes(U(F) \cap K)=1$. On note $\bar{P}=M\bar{U}$ l'opposé de $P$. Rappelons aussi que Harish-Chandra a défini la constante positive\index[iFT2]{$\gamma(G\arrowvert M)$}
    $$ \gamma(P) := \int_{\bar{U}(F)} \delta_P(m(\bar{u})) \dd \bar{u} $$
    où $\bar{u}= u(\bar{u}) m(\bar{u}) k(\bar{u})$ selon la décomposition d'Iwasawa $G(F)=U(F)M(F)K$. Elle est proportionnelle à la mesure de Haar sur $\bar{U}(F)$, qui est déjà choisie. En fait, elle ne dépend que de $G$ et $M$ (cf. \cite[p.240]{Wa03}), donc c'est loisible de la noter $\gamma(G|M)$. On a la formule d'intégrale pour toute $f \in C_c(\tilde{G})$
    \begin{gather}\label{eqn:UMU}
      \gamma(G|M) \int_{\tilde{G}} f(\tilde{x}) \dd \tilde{x} = \int_{U(F)} \int_{\tilde{M}} \int_{\bar{U}(F)} \delta_P(m)^{-1} f(u\tilde{m}\bar{u}) \dd u \dd\tilde{m} \dd \bar{u} .
    \end{gather}
\end{itemize}

On définit le sous-ensemble $\widetilde{M_0}^+$ (resp. $\widetilde{A_0}^{\dagger,+}$) comme l'image réciproque de
$$ \overline{\mathfrak{a}_0^+} := \{ H \in \mathfrak{a}_0 :  \forall \alpha \in \Delta_0, \; \angles{\alpha, H} \geq 0 \} $$
par $H_{\tilde{M_0}}$ (resp. $H_{\tilde{M_0}}$ restreint à $\widetilde{A_0}^\dagger$); ici on a changé les notations de \cite{Wa03} afin d'alléger les symboles. On en déduit la décomposition suivante.

\begin{proposition}\label{prop:KAK}
  Il existe un sous-ensemble fini $\tilde{\Gamma}$ de $\tilde{G}$ tel que
  $$ \tilde{G} = \bigsqcup_{\tilde{\gamma} \in \tilde{\Gamma}} \tilde{\gamma} \tilde{K} \widetilde{M_0}^+ \tilde{K} .$$
\end{proposition}
\begin{proof}
  Cela résulte de la décomposition correspondante $G(F) = K \overline{M_0^+} K$, où $\overline{M_0^+}$ est l'image réciproque de $\overline{\mathfrak{a}_0^+}$ par $H_{M_0}$.
\end{proof}

Enfin, on peut prendre une fonction hauteur $\|\cdot\|_G: G(F) \to \{r \in \R : r \geq 1\}$ adaptée à $K$ comme dans \cite[p.242]{Wa03} et on pose $\|\cdot\|_{\tilde{G}} := \|\rev(\cdot)\|_G$. Ceci joue le rôle de la fonction $\|\cdot\|$ dans \cite{Wa03} qui intervient dans diverses majorations. Il convient aussi de considérer la fonction $\sigma(\tilde{x}) = \sigma(x) := \sup \{1, \log\|x\|\}$.

\subsection{Représentations}\label{sec:representations}
Soit $P=MU$ un sous-groupe parabolique de $G$. Le scindage unipotent \cite[\S 2.2]{Li10a} donne une décomposition canonique $\tilde{P} = \tilde{M}U(F)$. Cela permet de définir le foncteur d'induction parabolique normalisée\index[iFT2]{$\mathcal{I}_{\tilde{P}}$}
$$ \mathcal{I}_{\tilde{P}}(\cdot) = \mathrm{Ind}^{\tilde{G}}_{\tilde{P}}(\delta_P^{1/2} \otimes \cdot) $$
où $\delta_P$ signifie la fonction module de $P(F)$ composée avec $\rev$, ce qui est aussi la fonction module de $\tilde{P}$.

De même, on sait définir le foncteur de Jacquet normalisé: soit $(\pi, V)$ une représentation lisse de $\tilde{G}$, on note $j_{\tilde{P}}: V \to V_{\tilde{P}}$ le quotient maximal sur lequel $U(F)$ opère trivialement. On obtient la représentation lisse $(\pi_{\tilde{P}}, V_{\tilde{P}})$ de $\tilde{M}$ en posant\index[iFT2]{$\pi_{\tilde{P}}$}
$$ \pi_{\tilde{P}}(\tilde{m}) j_{\tilde{P}}(v) = \delta_P(m)^{-1/2} j_{\tilde{P}}(\pi(\tilde{m})v), \qquad v \in V . $$

En outre, le groupe $X(\tilde{G})$ opère sur l'ensemble des classes représentations lisses (resp. admissibles, unitaires) par produit tensoriel. On peut ainsi développer une théorie de décomposition de Bernstein comme le cas de groupes réductifs connexes. Ces foncteurs ont des propriétés algébriques comme le cas de groupes réductifs connexes, par exemple le lemme géométrique de Bernstein-Zelevinsky, l'admissibilité uniforme, la réciprocité de Frobenius et le second théorème d'adjonction, pour en citer quelques uns. Observons en passant que toutes ces opérations préservent l'équivariance sous $\bmu_m$.

Au lieu de reproduire toutes les démonstrations des assertions ci-dessus, c'est peut-être plus raisonnable de les admettre comme les hypothèses. Donnons toutefois une esquisse.
\begin{itemize}
  \item Le groupe topologique $\tilde{G}$ est un $l$-groupe au sens de \cite[1.1]{BZ76}. Il suffit que $\tilde{G}$ soit un $l$-espace; cela résulte du fait que $G(F)$ est un $l$-espace.
  \item De même, $\tilde{G}$ est dénombrable à l'infini (c'est-à-dire il est une réunion dénombrable de sous-ensembles compacts) d'après la propriété correspondante de $G(F)$.
  \item Pour le lemme géométrique, on doit vérifier les conditions (1)-(4) et $(\bigstar)$ dans \cite[5.1]{BZ77}. Pour cela, on utilise la décomposition de Lévi sur le revêtement $\tilde{P}=\tilde{M}U(F)$ et ses conséquence suivantes: l'action adjointe de $\tilde{M}$ sur $U(F)$ se factorise par $\rev: \tilde{M} \to M(F)$, et l'action du groupe de Weyl $W^G_0$ sur les sous-groupes $\tilde{P}, \tilde{M}$, $U(F) \subset \tilde{G}$ se déduit de celle sur $P, M, U \subset G$.
\end{itemize}
Des cas spéciaux du lemme géométrique pour revêtements se trouvent dans \cite[\S I.2]{KP84} et \cite{HM10}.

Passons maintenant à l'aspect ``analytique''. Le résultat suivant est important pour l'étude de coefficients matriciels de modules de Jacquet, cf. \cite[I.2]{Wa03}. Faute d'avoir une référence convenable, on en donnera aussi une démonstration.

\begin{proposition}
  Soit $V$ un sous-espace de dimension finie de $C^\infty(\widetilde{A_G}^\dagger)$ stable par la représentation régulière à droite de $\widetilde{A_G}^\dagger$, notée $\rho$. Alors il existe un sous-ensemble fini $\mathcal{X}$ de $\Hom(\widetilde{A_G}^\dagger, \C^\times)$ et un entier $d \geq 0$, tels que pour tout $\chi \in \mathcal{X}$ et tout $f \in V$, il existe un polynôme $P_{\chi,f}$ sur $\mathfrak{a}_G$, $\deg P_{\chi,f} \leq d$, tel que
  $$ f(\tilde{a}) = \sum_{\chi \in \mathcal{X}} \chi(\tilde{a}) P_{\chi,f}(H_G(a)), \qquad \tilde{a} \in \widetilde{A_G}^\dagger. $$

  Si l'on prend $\mathcal{X}$ minimal vérifiant les propriétés ci-dessus pour tout $f$, alors pour tout $\chi \in \mathcal{X}$, les fonctions $\tilde{a} \mapsto \chi(\tilde{a})$ et $\tilde{a} \mapsto \chi(\tilde{a}) P_{\chi,f}(H_M(a))$ sont dans $V$.
\end{proposition}
\begin{proof}
  Prenons $d = \dim V$. La finitude de $\dim V$ et la commutativité de $\widetilde{A_G}^\dagger$ donnent une décomposition finie $V = \bigoplus_{\chi \in \mathcal{X}} V_\chi$, stable par $\rho$, où
  $$ V_\chi = \bigcap_{\tilde{a} \in \widetilde{A_G}^\dagger} \Ker(\rho(\tilde{a}) - \chi(\tilde{a}))^d . $$

  On se ramène ainsi au cas où $\mathcal{X}=\{\chi\}$ est un singleton. En multipliant les éléments de $V_\chi$ par $\tilde{a} \mapsto \chi(\tilde{a})^{-1}$, on peut supposer de plus que $\chi=1$. Soit $f \in V_1$. Comme $\rho(\tilde{a})$ agit sur $V_1$ comme une matrice unipotente pour tout $\tilde{a}$, la fonction $f$ se factorise par le sous-groupe compact $\widetilde{A_G}^\dagger \cap \tilde{G}^1$. Alors l'équation
  $$ (\rho(\tilde{a}) - \identity)^d f = 0, \qquad \tilde{a} \in \widetilde{A_G}^\dagger $$
  devient un système de suites récurrentes linéaires sur le réseau $\tilde{\mathfrak{a}}_{G,F} \subset \mathfrak{a}_G$, et l'assertion en découle.
\end{proof}

\begin{definition}\index[iFT2]{$V_\chi$}\index[iFT2]{$\Exp(\pi)$}
  Soient $(\pi,V)$ une représentation admissible de $\tilde{G}$ et $\chi \in \Hom(\widetilde{A_G}^\dagger, \C^\times)$. Définissons
  $$ V_\chi = \left\{v \in V : \exists d, \; v \in \bigcap_{\tilde{a} \in \widetilde{A_G}^\dagger} \Ker(\pi(\tilde{a}) - \chi(\tilde{a}))^d  \right\}. $$

  On a alors $V = \bigoplus_\chi V_\chi$. Définissons l'ensemble des exposants de $(\pi,V)$ par
  $$ \Exp(\pi) := \{ \chi \in \Hom(\widetilde{A_G}^\dagger, \C^\times) : V_\chi \neq 0 \}. $$
\end{definition}

\begin{remark}
  Soit $\chi \in \Exp(\pi)$. On peut vérifier aisément que $\Re\chi$ est un objet ``infinitesimal'':  il ne change pas si l'on remplace $A_G(F)^\dagger$ par un sous-groupe ouvert d'indice fini $A_G(F)^\ddagger$ et remplace $\chi$ par sa restriction à $\widetilde{A_G}^\ddagger$.
\end{remark}

Posons $C_\text{lisse}(\tilde{G})$ l'espace de fonctions sur $\tilde{G}$ bi-invariantes par un sous-groupe ouvert et compact. Il y a deux représentations $\rho$, $\lambda$ de $\tilde{G}$ sur cet espace, définies par
\begin{align*}
  (\rho(\tilde{x})f)(\tilde{y}) & = f(\tilde{y}\tilde{x}), \\
  (\lambda(\tilde{x})f)(\tilde{y}) & = f(\tilde{x}^{-1}\tilde{y}).
\end{align*}

Soit $(\pi,V)$ une représentation admissible de $\tilde{G}$. Notons $(\check{\pi},\check{V})$ sa contragrédiente. Soient $v \in V$, $\check{v} \in \check{V}$. On définit le coefficient matriciel comme l'élément de $C_\text{lisse}(\tilde{G})$
$$ \tilde{x} \mapsto \angles{\pi(\tilde{x})v, \check{v}}. $$

Notons $\mathcal{A}(\pi)$ la sous-représentation de $C_\text{lisse}(\tilde{G})$ (par $\rho$ et $\lambda$) engendrée par tous les coefficients de $\pi$, et posons
$$ \mathcal{A}(\tilde{G}) = \sum_\pi \mathcal{A}(\pi) = \bigcup_\pi \mathcal{A}(\pi). $$

Soit $P=MU$ un sous-groupe parabolique de $G$. On dispose de l'application terme constant le long de $\tilde{P}$, cf. \cite[I.6.2]{Wa03}
\begin{align*}
  \mathcal{A}(\tilde{G}) & \to \mathcal{A}(\tilde{M}), \\
  f  & \mapsto f_{\tilde{P}}.
\end{align*}
La fonction $f_{\tilde{P}}$ est caractérisée par la propriété suivante. Pour tout $\tilde{m} \in \tilde{M}$, il existe $\epsilon > 0$ tel que, si $\tilde{a} \in \widetilde{A_M}^\dagger$ et $|\alpha(a)|_F < \epsilon$ pour tout $\alpha \in \Sigma_P$, alors
\begin{gather}\label{eqn:terme-const}
  (\delta_P^{-1/2} f)(\tilde{m}\tilde{a}) = f_{\tilde{P}}(\tilde{m}\tilde{a}).
\end{gather}

La démonstration est basée sur un théorème de Casselman \cite[I.4.1]{Wa03}, la preuve-là s'adapte aux revêtements à l'aide de la  Proposition \ref{prop:KAK}.

\subsection{Fonctions de Schwartz-Harish-Chandra}
Introduisons la fonction $\Xi^G$ de Harish-Chandra sur $G(F)$. Rappelons que $\Xi^G$ est définie comme un coefficient matriciel de l'induction parabolique normalisée de la représentation triviale de $\tilde{P}_0$. On pose simplement\index[iFT2]{$\Xi^{\tilde{G}}$}
$$ \Xi^{\tilde{G}} := \Xi^G \circ \rev. $$

Cette fonction vérifie des majorations liées à $\|\cdot\|$ et $\sigma$, cf. \cite[II]{Wa03}. De même, les résultats d'intégrabilité de \cite[II]{Wa03} demeurent valables pour $\tilde{G}$, ce qui justifie toutes les intégrales que l'on considérera.

Soit $P \in \mathcal{F}(M_0)$. On note
$$ {}^+ \mathfrak{a}^{G *}_P := \sum_{\alpha \in \Delta_P} \R_{>0} \cdot \alpha $$
et on note ${}^+ \bar{\mathfrak{a}}^{G *}_P$ son adhérence.

\begin{proposition}[Cf. {\cite[III.1.1]{Wa03}}]
  Soit $(\pi,V)$ une représentation admissible de $\tilde{G}$ admettant un caractère central unitaire. Les conditions suivantes sont équivalentes.
  \begin{enumerate}
    \item Les valeurs absolues des coefficients de $\pi$ sont de carrés intégrables sur $\tilde{G}/\widetilde{A_G}^\dagger$.
    \item Pour tout parabolique semi-standard $P=MU$ et tout $\chi \in \Exp(\pi_{\tilde{P}})$, on a $\Re\chi \in {}^+ \mathfrak{a}^{G *}_P$.
    \item Pour tout parabolique standard propre maximal $P=MU$ et tout $\chi \in \Exp(\pi_{\tilde{P}})$, on a $\Re\chi \in {}^+ \mathfrak{a}^{G *}_P$.
  \end{enumerate}

  Supposons vérifiées ces conditions et soit $f \in \mathcal{A}(\pi)$. Pour tout $r \in \R$, il existe $c > 0$ tel que
  \begin{gather}\label{eqn:eq-w}
    |f(\tilde{x})| \leq c \Xi(x) (1+\sigma(x))^{-r}, \qquad \tilde{x} \in \tilde{G}.
  \end{gather}
\end{proposition}
On appelle une telle représentation de carré intégrable modulo le centre, ou bien $L^2$ modulo le centre. Si $(\pi, V)$ est de carré intégrable modulo le centre, on définit son degré formel comme la constante positive $d(\pi)$ telle que\index[iFT2]{$d(\pi)$}
$$ \iota_G \int_{\widetilde{A_G}^\dagger \backslash \tilde{G}} \angles{\pi(\tilde{x})v, \check{v}} \angles{v', \pi^\vee(\tilde{x}) \check{v}'} \dd\tilde{x} = d(\pi)^{-1} \angles{v, \check{v}'} \angles{v', \check{v}}, \quad v, v' \in V, \; \check{v}, \check{v}' \in V^\vee . $$
Remarquons que $d(\pi)$ dépend de la mesure de Haar sur $\tilde{G}$ mais pas du choix de $A_G(F)^\dagger$.

Définissons des versions dites faibles de $C_\text{lisse}(\tilde{G})$ et $\mathcal{A}(\tilde{G})$ comme suit.\index[iFT2]{$C^w_\text{lisse}(\tilde{G})$}\index[iFT2]{$\mathcal{A}^w(\tilde{G})$}
\begin{align*}
  C^w_\text{lisse}(\tilde{G}) & := \{ f \in C_\text{lisse}(\tilde{G}) : \exists c, r \text{ tels que \eqref{eqn:eq-w} est vérifié }  \}, \\
  \mathcal{A}^w(\tilde{G}) & := \mathcal{A}(\tilde{G}) \cap C^w_\text{lisse}(\tilde{G}).
\end{align*}

Soit $(\pi,V)$ une représentation admissible de $\tilde{G}$. On dit que $(\pi,V)$ est tempérée si $\mathcal{A}(\pi) \subset \mathcal{A}^w(\tilde{G})$. Les représentations de carré intégrable modulo le centre sont tempérées. Un fait important est
$$ \mathcal{A}^w(\tilde{G}) = \bigcup_{\pi: \text{tempérée}} \mathcal{A}(\pi) = \sum_{\pi: \text{tempérée}} \mathcal{A}(\pi). $$

\begin{proposition}
  Soit $(\pi,V)$ une représentation admissible de $\tilde{G}$. Les conditions suivantes sont équivalentes.
  \begin{enumerate}
    \item $(\pi,V)$ est tempérée.
    \item Pour tout parabolique semi-standard $P=MU$ et tout $\chi \in \Exp(\pi_{\tilde{P}})$, on a $\Re\chi \in {}^+ \bar{\mathfrak{a}}^{G *}_P$.
    \item Pour tout parabolique standard propre maximal $P=MU$ et tout $\chi \in \Exp(\pi_{\tilde{P}})$, on a $\Re\chi \in {}^+ \bar{\mathfrak{a}}^{G *}_P$.
  \end{enumerate}
\end{proposition}

En résumé, nous avons défini les ensembles
$$ \Pi_2(\tilde{M}) \subset \Pi_\text{temp}(\tilde{M}) \subset \Pi_\text{unit}(\tilde{M}) \subset \Pi(\tilde{M}) $$
qui désignent les ensembles de classes d'équivalences de représentations de carré intégrable modulo le centre, tempérées, unitaires et admissibles irréductibles de $\tilde{M}$, respectivement. On peut aussi considérer l'équivariance sous $\bmu_m$ et introduire les sous-ensembles $\Pi_{2,-}(\tilde{M})$, $\Pi_\text{temp,-}(\tilde{M})$, etc.

Soient $r \in \R$ et $f \in C_\text{lisse}(\tilde{G})$. On définit\index[iFT2]{$\nu_r$}
$$ \nu_r(f) := \sup_{\tilde{x} \in \tilde{G}} \left( |f(\tilde{x})| \Xi(x)^{-1} (1+\sigma(x))^r \right) . $$

Soit $H \subset \tilde{G}$ un sous-groupe ouvert compact. Notons $\mathcal{C}_H$ l'espace des fonctions $f$ bi-invariantes par $H$ telles que $\nu_r(f)$ est fini pour tout $r \in \R$. Les semi-normes $\nu_r$ définissent une topologie sur $\mathcal{C}_H$. On pose $\mathcal{C}(\tilde{G}) = \bigcup_H \mathcal{C}_H$, muni de la topologie $\varinjlim$. C'est une algèbre pour le produit de convolution, qui est séparément continu. Les éléments de $\mathcal{C}(\tilde{G})$ s'appellent aussi les fonctions de Schwartz-Harish-Chandra. L'application $\tilde{G} \times \tilde{G} \times \mathcal{C}(\tilde{G}) \to \mathcal{C}(\tilde{G})$ définie par $(\tilde{x},\tilde{y},f) \mapsto \rho(\tilde{x})\lambda(\tilde{y})f$ est continue.\index[iFT2]{$\mathcal{C}(\tilde{G})$}

Si $(\pi,V) \in \Pi_\text{temp}(\tilde{G})$ et $f \in \mathcal{C}(\tilde{G})$, on peut définir l'opérateur $\pi(f)$ de sorte que
$$ \angles{\pi(f)v, \check{v}} = \int_{\tilde{G}} f(\tilde{x}) \angles{\pi(\tilde{x})v, \check{v}} \dd \tilde{x}, \qquad v \in V, \check{v} \in \check{V}. $$
Alors $f \mapsto \pi(f)$ est un homomorphisme d'algèbres. Pour tout $f$, l'opérateur $\pi(f)$ est de rang fini et on peut définir la fonctionnelle linéaire continue $f \mapsto \Theta_\pi(f) := \Tr \pi(f)$, i.e. le caractère de $\pi$.

Citons deux variantes faibles des constructions précédentes. Soit $(\pi,V) \in \Pi_\text{temp}(\tilde{G})$. Soit $P=MU \in \mathcal{F}(M_0)$, alors le module de Jacquet admet une décomposition canonique.
$$ (\pi_{\tilde{P}}, V_{\tilde{P}}) = (\pi_{\tilde{P}}^w, V_{\tilde{P}}^w) \oplus (\pi_{\tilde{P}}^+, V_{\tilde{P}}^+) $$
où $V_{\tilde{P}}^w$ se constitue des $V_\chi$ avec $\Re \chi = 0$, et $V_{\tilde{P}}^+$ se constitue du reste. Alors $(\pi_{\tilde{P}}^w, V_{\tilde{P}}^w)$ est tempérée, cf. \cite[III.3.1]{Wa03}. La réciprocité de Frobenius demeure valable dans la catégorie des représentations tempérées admissibles si l'on remplace le foncteur $\pi \mapsto \pi_{\tilde{P}}$ par $\pi \mapsto \pi_{\tilde{P}}^w$.

On dispose de l'application terme constant faible le long de $\tilde{P}$, analogue de \eqref{eqn:terme-const}
\begin{align*}
  \mathcal{A}^w(\tilde{G}) & \to \mathcal{A}^w(\tilde{M}), \\
  f  & \mapsto f^w_{\tilde{P}}.
\end{align*}
La fonction $f^w_{\tilde{P}}$ est caractérisée par la propriété suivante. Pour tout $\tilde{m} \in \tilde{M}$, on a
\begin{gather}\label{eqn:terme-const-faible}
  \lim_{\tilde{a} \stackrel{P}{\to} \infty} (\delta_P^{-1/2} f)(\tilde{m}\tilde{a}) - f^w_{\tilde{P}}(\tilde{m}\tilde{a}) = 0.
\end{gather}
où la limite signifie que $\tilde{a} \in \widetilde{A_M}^\dagger$ et $-H_M(a)$ tend vers l'infini dans un cône ouvert strictement contenu dans celui déterminé par $\Delta_P$, cf \cite[III.5]{Wa03}.

On définit la fonction $f_{\tilde{P}}^{w,\text{Ind}}: \tilde{G} \times \tilde{G} \to \mathcal{A}^w(\tilde{M})$ en posant\index[iFT2]{$f_{\tilde{P}}^w, f_{\tilde{P}}^{w,\text{Ind}}$}
\begin{gather}\label{eqn:wInd}
  f_{\tilde{P}}^{w,\text{Ind}}(\tilde{x},\tilde{y}) = (\rho(\tilde{x})\lambda(\tilde{y})f)^w_{\tilde{P}}.
\end{gather}

\subsection{Opérateurs d'entrelacement}\label{sec:op-entrelacement}
Le tore complexe $X(\tilde{G})$ opère sur $\Pi(\tilde{G})$ par $(\chi,\omega) \mapsto \omega \otimes \chi$, où $\chi \in X(\tilde{G})$ et $\omega \in \Pi(\tilde{G})$. Si $\chi$ est l'image de $\lambda \in \mathfrak{a}_{G,\C}^*$, on écrit aussi $\pi_\lambda := \pi \otimes \chi$\index[iFT2]{$\pi \otimes \chi, \pi_\lambda$}. L'action induite du tore compact $\Im X(\tilde{G})$ préserve $\Pi_2(\tilde{G})$. Pour toute orbite $\mathcal{O}$ sous $\Im X(\tilde{G})$, on peut choisir un point base $\omega \in \mathcal{O}$ et munir $\mathcal{O}$ de la structure de variété $C^\infty$ par l'application
\begin{align*}
  \Im X(\tilde{G})/\text{Stab}_{\Im X(\tilde{G})}(\omega) & \rightiso \mathcal{O}. \\
  \chi & \mapsto \omega \otimes \chi.
\end{align*}

Le stabilisateur $\text{Stab}_{\Im X(\tilde{G})}(\omega)$ est fini. De façon analogue, pour l'action de $X(\tilde{G})$ on définit l'orbite $\mathcal{O}_\C$ qui est une variété algébrique complexe. Ainsi, on peut parler de fonctions $C^\infty$, régulières ou rationnelles sur $\mathcal{O}_\C$. Ces notions ne dépendent pas du choix du point base $\omega$.

Soient $M \in \mathcal{L}(M_0)$, $P=MU$ et $P'=MU'$ dans $\mathcal{P}(M)$. Soient $(\pi,V)$ une représentation admissible de $\tilde{M}$, $f \in \mathcal{I}_{\tilde{P}} V$, $\tilde{x} \in \tilde{G}$. On dit que l'intégrale
$$ (J_{\tilde{P}'|\tilde{P}}(\pi)f)(\tilde{x}) = \int_{(U \cap U')(F) \backslash U'(F)} f(u'\tilde{x}) \dd u' $$
est absolument convergente, égale à $v \in V$, si pour tout $\check{v} \in \check{V}$ l'intégrale
$$ \int_{(U \cap U')(F) \backslash U'(F)} \angles{f(u'\tilde{x}), \check{v}} \dd u' $$
est absolument convergente et égale à $\angles{v,\check{v}}$. Si $(J_{\tilde{P}'|\tilde{P}}(\pi)f)(\tilde{x})$ est absolument convergent pour tous $f$ et $\tilde{x}$, on dit que $J_{\tilde{P}'|\tilde{P}}(\pi)$ est défini par des intégrales convergentes, et il définit un opérateur d'entrelacement $\mathcal{I}_{\tilde{P}}(\pi) \to \mathcal{I}_{\tilde{P}'}(\pi)$.

On peut identifier $\mathcal{I}_{\tilde{P}}(V)$ à l'induction $\text{Ind}^{\tilde{K}}_{\tilde{K} \cap \tilde{P}}(V)$ par restriction sur $\tilde{K}$. Lorsque $\pi$ varie dans $\mathcal{O}_\C$, cet espace ne change pas. En introduisant les $\C[\tilde{M}/\tilde{M}^1]$-familles de représentations (cf. \cite[I.5]{Wa03}), on peut parler des familles d'opérateurs $\mathcal{I}_{\tilde{P}}(\pi) \to \mathcal{I}_{\tilde{P}'}(\pi)$ régulières ou rationnelles, où $\pi \in \mathcal{O}_\C$, voir \cite[IV.1]{Wa03}.

Soient $P \in \mathcal{P}(M)$ et $\alpha \in \Sigma_P$, la coracine $\alpha^\vee \in \mathfrak{a}_M$ peut être définie en choisissant un parabolique minimal $P_0 \subset P$ et en restreignant $\alpha^\vee \in \Delta_0^\vee$ à $\mathfrak{a}_M$. On voit qu'elle est indépendante de $P_0$.

Nous renvoyons le lecteur à \cite[IV.1]{Wa03} pour les résultats suivants dans le cas des groupes réductifs connexes.

\begin{theorem}\label{prop:J-convergence}\index[iFT2]{$J_{\tilde{P}'\arrowvert\tilde{P}}(\pi)$}
  Soit $(\pi,V)$ une représentation admissible de longueur finie de $\tilde{M}$. Alors il existe $R \in \R$ tel que si $\angles{\Re(\chi), \alpha^\vee} > R$ pour tout $\alpha \in \Sigma_{P} \cap \Sigma_{\bar{P}'}$, alors $J_{\tilde{P}'|\tilde{P}}(\pi \otimes \chi)$ est défini par des intégrales convergentes. L'opérateur $J_{\tilde{P}'|\tilde{P}}(\pi \otimes \chi)$ défini pour de tels $\pi \otimes \chi$ se prolonge en un opérateur rationnel sur $\mathcal{O}_\C$.

  Si $(\pi,V)$ est tempérée, alors $J_{\tilde{P}'|\tilde{P}}(\pi \otimes \chi)$  est défini par des intégrales convergentes pourvu que $\angles{\Re\chi, \alpha^\vee} > 0$ pour tout $\alpha \in \Sigma_P^\text{red} \cap \Sigma_{\bar{P}'}^\text{red}$.
\end{theorem}

Posons $d(P',P) := |\Sigma_{P'}^\text{red} \cap \Sigma_{\bar{P}}^\text{red}|$.

\begin{proposition}\label{prop:prop-entrelacement}
  Soit $(\pi,V)$ une représentation admissible de longueur finie de $\tilde{M}$. Alors
  \begin{enumerate}
    \item $J_{\tilde{P}'|\tilde{P}}(\pi)^\vee = J_{\tilde{P}|\tilde{P}'}(\check{\pi})$, où $\vee$ signifie l'opérateur dual;
    \item $J_{\tilde{P}'|\tilde{P}}(\pi) = J_{\tilde{P}'|\tilde{P}''}(\pi) J_{\tilde{P}''|\tilde{P}}(\pi)$ si $P'' \in \mathcal{P}(M)$ et $d(P',P)=d(P',P'')+d(P'',P)$;
    \item soit $P''=M'' U'' \in \mathcal{F}(M_0)$ contenant $P$ et $P'$, on fait les identifications
    \begin{align*}
      \mathcal{I}_{\tilde{P}}(\pi) & = \mathcal{I}_{\tilde{P}''} \mathcal{I}^{\tilde{M}''}_{\tilde{P} \cap \tilde{M}''} (\pi), \\
      \mathcal{I}_{\tilde{P}'}(\pi) & = \mathcal{I}_{\tilde{P}''} \mathcal{I}^{\tilde{M}''}_{\tilde{P}' \cap \tilde{M}''} (\pi),
    \end{align*}
    alors $J^{\tilde{G}}_{\tilde{P}'|\tilde{P}}(\pi)$ est l'opérateur déduit de $J^{\tilde{M}''}_{\tilde{P}' \cap \tilde{M}''|\tilde{P} \cap \tilde{M}''}(\pi)$ par le foncteur $\mathcal{I}_{\tilde{P}''}(\cdot)$;
    \item soient $w \in W^G_0$ et $\tilde{w} \in \tilde{K}$ un représentant, alors
      $$ J_{w\tilde{P}'|w\tilde{P}}(\tilde{w}\pi) = A(\tilde{w}) J_{\tilde{P}'|\tilde{P}}(\pi) A(\tilde{w})^{-1} $$
      où $A(\tilde{w}): \mathcal{I}_{\tilde{P}}(\pi) \rightiso \mathcal{I}_{w\tilde{P}}(\tilde{w}\pi)$ (ou avec $P'$ au lieu de $P$) est la translation $\varphi(\cdot) \mapsto \varphi(\tilde{w}^{-1} \cdot)$.\index[iFT2]{$\tilde{w}$}
  \end{enumerate}
\end{proposition}

Supposons maintenant que $\pi \in \Pi_2(\tilde{M})$, son orbite sous $X(\tilde{M})$ est notée $\mathcal{O}_\C$. On dit qu'un élément dans $\mathcal{O}_\C$ est $\tilde{G}$-régulier si son stabilisateur dans $W^G(M)$ est trivial. De tel éléments forment un ouvert de Zariski dense dans $\mathcal{O}_\C$.

Montrons qu'il existe un ouvert dense de Zariski dans $\mathcal{O}_\C$ tel que pour tout $\pi'$ dedans, $\mathcal{I}_{\tilde{P}}(\pi')$ est irréductible pour tout $P \in \mathcal{P}(M)$. En effet, il suffit de montrer que $\Hom(\mathcal{I}_{\tilde{P}}(\pi), \mathcal{I}_{\tilde{P}}(\pi))=\C$ lorsque $\pi$ est $G$-régulier, car $\mathcal{I}_{\tilde{P}}(\pi)$ est unitaire de longueur finie. Or cela résulte immédiatement de la réciprocité de Frobenius et du lemme géométrique de Bernstein-Zelevinsky.

Grâce au Théorème \ref{prop:J-convergence} et Proposition \ref{prop:prop-entrelacement}, on montre (cf. \cite[IV.3]{Wa03}) qu'il existe une fonction rationnelle $j$ sur $\mathcal{O}_\C$ telle que
$$ J_{\tilde{P}|\tilde{\bar{P}}}(\pi') J_{\tilde{\bar{P}}|\tilde{P}}(\pi') = j(\pi'), \qquad \pi' \in \mathcal{O}_\C, \; P \in \mathcal{P}(M). $$

Notons $\Sigma_M^\text{red}$ l'ensemble des racines réduites de $A_M$. À chaque $\alpha \in \Sigma_M^\text{red}$ est associé un sous-groupe de Lévi $M_\alpha$ contenant $M$ tel que $\Sigma_M^{M_\alpha, \text{red}} = \{\alpha\}$. Notons $j_\alpha$ la fonction rationnelle définie en remplaçant $\tilde{G}$ par $\tilde{M}_\alpha$. On a alors
$$ j = \prod_{\alpha \in \Sigma_M^\text{red}/\pm} j_\alpha , $$
où $\alpha$ parcourt $\Sigma_M^\text{red}$ à signe près. Définissons la fonction $\mu$ sur $\mathcal{O}_\C$ par
\begin{gather}\label{eqn:HC-mu}
  \mu := j^{-1} .
\end{gather}\index[iFT2]{$\mu(\pi)$}

Pour $\alpha \in \Sigma_M^\text{red}$, on note $\mu_\alpha$ la fonction définie en remplaçant $\tilde{G}$ par $\tilde{M}_\alpha$.

\begin{proposition}[Cf. {\cite[IV.3]{Wa03}}]
  La fonction $\mu$ est rationnelle sur $\mathcal{O}_\C$. Elle est régulière et réelle non négative sur $\mathcal{O}$. On a
  $$ \mu = \prod_{\alpha \in \Sigma_M^\text{red}/\pm} \mu_\alpha  .$$

  De plus, $\mu$ est invariante par $W^G_0$ et par passage à la contragrédiente $\pi \mapsto \check{\pi}$.
\end{proposition}

\begin{remark}\label{rem:dep-entrelaceur}\index[iFT2]{$\check{\alpha}$}
  Soit $\alpha \in \Delta_P$. On note $r_\alpha$ le plus petit rationnel positif tel que $r_\alpha \cdot \alpha^\vee \in \mathfrak{a}_{G,F}$. On note
  \begin{gather}\label{eqn:alpha-check}
    \check{\alpha} := r_\alpha \alpha^\vee.
  \end{gather}

  Supposons choisi un point base $\pi$ dans $\mathcal{O}_\C$. On pose $P' := \bar{P}$, alors les fonctions $\lambda \mapsto J_{\tilde{P}'|\tilde{P}}(\pi_\lambda)$ et $\lambda \mapsto \mu(\pi_\lambda)$, où $\lambda \in \mathfrak{a}_{M,\C}^*$, sont rationnelles en les variables $\{ q^{-\angles{\lambda, \check{\alpha}}} :\alpha \in \Delta_P \}$. Cette propriété est implicite dans la preuve de la rationalité des opérateurs d'entrelacement \cite[p.278]{Wa03}.
\end{remark}

\subsection{Coefficients d'induites et la fonction \texorpdfstring{$c$}{c}}
Fixons toujours $M \in \mathcal{L}(M_0)$. Soient $P \in \mathcal{P}(M)$, $(\omega, E) \in \mathcal{O}$. Posons\index[iFT2]{$L(\omega,\tilde{P})$}
\begin{align*}
  L(\omega,\tilde{P}) & = \mathcal{I}^{\tilde{G} \times \tilde{G}}_{\tilde{P} \times \tilde{P}}(E \boxtimes \check{E}) \\
   & = \mathcal{I}_{\tilde{P}}(E) \boxtimes \mathcal{I}_{\tilde{P}}(E)^\vee \hookrightarrow \End(\mathcal{I}_{\tilde{P}}E).
\end{align*}
Ici $\boxtimes$ désigne le produit tensoriel extérieur de représentations, par exemple $E \boxtimes \check{E}$ est l'espace vectoriel $E \otimes \check{E}$ considéré comme une représentation de $\tilde{M} \times \tilde{M}$, et ainsi de suite. Le revêtement $\tilde{G} \times \tilde{G} \to G(F)\times G(F)$ n'appartient pas rigoureusement à notre classe car son noyau n'est pas cyclique, mais peu importe.

Pour $v \otimes \check{v} \in L(\omega,\tilde{P})$, on peut définir le coefficient
$$ E^{\tilde{G}}_{\tilde{P}}(v \otimes \check{v}): \tilde{x} \mapsto \angles{\omega(\tilde{x})v, \check{v}}. $$
Ceci induit une application linéaire $E^{\tilde{G}}_{\tilde{P}}: L(\omega, \tilde{P}) \to C_\text{lisse}(\tilde{G})$. Plus généralement, soient $P'=M'U' \in \mathcal{F}(M_0)$ tel que $M' \supset M$ et $P \in \mathcal{P}^{M'}(M)$, posons\index[iFT2]{$E^{\tilde{P}'}_{\tilde{P}}$}
$$ E^{\tilde{P}'}_{\tilde{P}}: \mathcal{I}^{\tilde{G} \times \tilde{G}}_{\tilde{P}' \times \tilde{\bar{P}}'} L^{\tilde{M}'}(\omega,\tilde{P}) \to \mathcal{I}^{\tilde{G} \times \tilde{G}}_{\tilde{P}' \times \tilde{\bar{P}}'} C_\text{lisse}(\tilde{M}') $$
l'application qui se déduit de $E^{\tilde{M}'}_{\tilde{P}}$ par fonctorialité.

Soient $M, M' \in \mathcal{L}(M_0)$. Posons\index[iFT2]{$\mathcal{W}(M'\arrowvert G\arrowvert M), \mathcal{W}(M\arrowvert G), W(M'\arrowvert G\arrowvert M)$}
\begin{align*}
  \mathcal{W}(M'|G|M) & := \{ w \in W^G_0 : w M \subset M' \}, \\
  \mathcal{W}(M|G) & := \mathcal{W}(M|G|M),\\
  W(M'|G|M) & := W^{M'}_0 \backslash \mathcal{W}(M'|G|M).
\end{align*}

Ces ensembles sont éventuellement vides. Ils opèrent sur des paraboliques par conjugaison. Dans ce qui suit il convient de fixer des représentants dans $\tilde{K}$ de ces ensembles. Néanmoins, les résultats ultérieurs seront indépendants des choix. Soient $s \in \tilde{K}$ un représentant d'un élément dans $W^G_0$ et $\omega$ une représentation de $\tilde{M}$, on note $s\omega$ la représentation de $s\tilde{M}$ obtenue par transport de structure. Elle est indépendante du choix du représentant à isomorphisme près.

Pour $P \in \mathcal{P}(M)$, $s \in \mathcal{W}(M'|G|M)$, définissons
\begin{align*}
  P_s & := (M' \cap sP) U' , \\
  P_s^\diamond & := (M' \cap sP) \bar{U}'
\end{align*}

Notons $\lambda(s): \mathcal{I}_{\tilde{P}}(\omega) \mapsto \mathcal{I}_{s\tilde{P}}(s\omega)$ la translation à gauche par $s$ en se rappelant que $\tilde{I}^{\tilde{G}}_{\tilde{P}}(\omega)$ est un espace de fonctions sur $\tilde{G}$. Vu les identifications
\begin{align*}
  \mathcal{I}_{\tilde{P}'} \mathcal{I}^{\tilde{M}'}_{\tilde{M}' \cap s\tilde{P}} & = \mathcal{I}_{\tilde{P}_s}, \\
  \mathcal{I}_{\tilde{\bar{P}}'} \mathcal{I}^{\tilde{M}'}_{\tilde{M}' \cap s\tilde{P}} & = \mathcal{I}_{\tilde{P}_s^\diamond},
\end{align*}
on définit\index[iFT2]{$c_{\tilde{P}'\arrowvert\tilde{P}}(s,\omega)$}
\begin{gather*}
  c_{\tilde{P}'|\tilde{P}}(s,\omega): L(\omega,\tilde{P}) \to \mathcal{I}^{\tilde{G} \times \tilde{G}}_{\tilde{P}' \times \tilde{\bar{P}}'} L^{\tilde{M}'}(s\omega, \tilde{M}' \cap s\tilde{P}), \\
  v \otimes \check{v} \mapsto \gamma(G|M')^{-1} \left( J_{\tilde{P}_s|s\tilde{P}}(s\omega)\lambda(s) v \otimes J_{\tilde{P}_s^\diamond|s\tilde{P}}(s\check{\omega})\lambda(s) \check{v} \right) .
\end{gather*}

\begin{theorem}
  Fixons $P \in \mathcal{P}(M)$, $P' \in \mathcal{P}(M')$ et $\mathcal{O}$ une $\Im X(\tilde{M})$-orbite contenant une représentation de carré intégrable modulo le centre. Soient $\omega \in \mathcal{O}$ un élément $\tilde{G}$-régulier et $\psi \in L(\omega, \tilde{P})$, alors
  $$ (E^{\tilde{G}}_{\tilde{P}}\psi)^{\mathrm{Ind},w}_{\tilde{P}'} = \sum_{s \in W(M'|G|M)} E^{\tilde{P}'}_{\tilde{M} \cap s\tilde{P}}(c_{\tilde{P}'|\tilde{P}}(s,\omega)\psi). $$
\end{theorem}

Rappelons \eqref{eqn:wInd} pour la définition de $(\cdots)^{\mathrm{Ind},w}_{\tilde{P}'}$. Il faut aussi des fonctions auxiliaires. Soient $M \in \mathcal{L}(M_0)$, $\omega \in \mathcal{O}$ une $\Im X(\tilde{M})$-orbite contenant une représentation de carré intégrable modulo le centre, $P,P' \in \mathcal{P}(M)$ et $s \in \mathcal{W}(G|M)$. Définissons\index[iFT2]{${}^\circ c_{\tilde{P}'\arrowvert\tilde{P}}(s,\omega)$}
$$ {}^\circ c_{\tilde{P}'|\tilde{P}}(s,\omega) := c_{\tilde{P}'|\tilde{P}}(1,s\omega)^{-1} c_{\tilde{P}'|\tilde{P}}(s,\omega) \in \Hom_{\tilde{G}\times\tilde{G}}(L(\omega, \tilde{P}), L(s\omega, \tilde{P})). $$

Les opérateurs $c_{\tilde{P}'|\tilde{P}}(1,s\omega)^{-1}$, $c_{\tilde{P}'|\tilde{P}}(s,\omega)$ et ${}^\circ c_{\tilde{P}'|\tilde{P}}(s,\omega)$ sont définis comme des fonctions rationnelles en $\omega \in \mathcal{O}_{\mathfrak{C}}$, puisque les opérateurs d'entrelacement le sont. On peut ainsi parler de la régularité de ces opérateurs sur $\mathcal{O}$; cf. \cite[V.1]{Wa03}.

\begin{proposition}
  L'application $\omega \mapsto {}^\circ c_{\tilde{P}'|\tilde{P}}(s,\omega)$ est régulière sur $\mathcal{O}$. Pour tout $\omega \in \mathcal{O}$,  l'opérateur ${}^\circ c_{\tilde{P}'|\tilde{P}}(s,\omega)$ est unitaire. On a
  $$ E^{\tilde{G}}_{\tilde{P}'}({}^\circ c_{\tilde{P}'|\tilde{P}}(s,\omega) \psi) = E^{\tilde{G}}_{\tilde{P}}(\psi) $$
  pour tout $\psi \in L(\omega, \tilde{P})$.
\end{proposition}

\subsection{Énoncé de la formule de Plancherel}\label{sec:Plancherel-enonce}
Considérons un parabolique $P=MU \in \mathcal{F}(M_0)$. Soit $\mathcal{O}$ une $\Im X(\tilde{M})$-orbite rencontrant $\Pi_2(\tilde{M})$. Soient $(\omega,E), (\omega', E') \in \mathcal{O}$ tels que $\omega \simeq \omega'$. Alors les espaces $\mathcal{I}^{\tilde{G} \times \tilde{G}}_{\tilde{P} \times \tilde{P}}(E \boxtimes \check{E})$ et $\mathcal{I}^{\tilde{G} \times \tilde{G}}_{\tilde{P} \times \tilde{P}}(E' \boxtimes \check{E}')$ sont canoniquement isomorphes: l'isomorphisme étant induit d'un isomorphisme $\omega \rightiso \omega'$ quelconque et de son dual.

Définissons $C^\infty(\mathcal{O}, \tilde{P})$\index[iFT2]{$C^\infty(\mathcal{O}, \tilde{P})$} comme l'espace des fonctions $\psi: \omega \to \psi_\omega \in L(\omega, \tilde{P})$ respectant les isomorphismes ci-dessus, telles que $\psi$ est $C^\infty$ sur $\mathcal{O}$. Rappelons que l'espace de $L(\omega, \tilde{P})$ ne change pas lorsque $\omega$ varie dans $\mathcal{O}$, pourvu que l'on le réalise comme un espace de fonctions sur $\tilde{K} \times \tilde{K}$; notons-le $L_{\tilde{K}}(\mathcal{O})$. Si $H$ est un sous-groupe ouvert compact de $\tilde{G}$, $L_{\tilde{K}}(\mathcal{O})^{H \times H}$ est de dimension finie. Comme $\Im X(\tilde{M})$ est compact, $C^\infty(\Im X(\tilde{M}))$ est muni de la famille des semi-normes $\|\varphi \|_D := \sup |D\varphi|$, où $D$ parcourt les opérateurs différentiels sur $\Im X(\tilde{M})$. Donc l'espace
$$ C^\infty(\Im X(\tilde{M})) \otimes L_{\tilde{K}}(\mathcal{O})^{H \times H} $$
est muni d'une topologie canonique. Il contient $C^\infty(\mathcal{O},\tilde{P})^{H \times H}$ comme un sous-espace fermé. On munit $C^\infty(\mathcal{O},\tilde{P})$ de la topologie $\varinjlim$ en variant $H$.

\begin{proposition}[Cf. {\cite[VI.2.1]{Wa03}}]
  Soit $\psi \in C^\infty(\mathcal{O},\tilde{P})$, alors
  \begin{enumerate}
    \item pour tout $\tilde{x} \in \tilde{G}$, la fonction $\omega \mapsto (E^{\tilde{G}}_{\tilde{P}}\psi_\omega)(\tilde{x})$ est $C^\infty$ sur $\mathcal{O}$;
    \item pour tout $P' = M' U' \in \mathcal{F}(M_0)$ et $\tilde{m}' \in \tilde{M}'$, les fonctions $\omega \mapsto (E^{\tilde{G}}_{\tilde{P}}\psi_\omega)_{\tilde{P}'}(\tilde{m}')$ et $\omega \mapsto (E^{\tilde{G}}_{\tilde{P}}\psi_\omega)^w_{\tilde{P}'}(\tilde{m}')$ sont $C^\infty$ sur $\mathcal{O}$.
  \end{enumerate}
\end{proposition}

C'est donc loisible de poser pour $\psi \in C^\infty(\mathcal{O}, \tilde{P})$,
$$ f_\psi(\tilde{x}) = \int_{\mathcal{O}} \mu(\omega) (E^{\tilde{G}}_{\tilde{P}}\psi_\omega)(\tilde{x}) \dd\omega, \qquad \tilde{x} \in \tilde{G}, $$
où $\mathcal{O}$ est muni de la mesure telle que $\Im X(\tilde{M}) \to \mathcal{O}$ préserve localement les mesures.

\begin{proposition}[Cf. {\cite[VI.3.1]{Wa03}}]
  L'application $\psi \mapsto f_\psi$ est une application linéaire continue de $C^\infty(\mathcal{O}, \tilde{P})$ sur $\mathcal{C}(\tilde{G})$.
\end{proposition}

On fait varier les $(\mathcal{O},\tilde{P})$. Notons $\Theta$ l'ensemble des paires $(\mathcal{O}, P)$ où $P=MU \in \mathcal{F}(M_0)$ et $\mathcal{O}$ est comme ci-dessus. Posons $C^\infty(\Theta) := \bigoplus_{(\mathcal{O},P) \in \Theta} C^\infty(\mathcal{O}, \tilde{P})$\index[iFT2]{$C^\infty(\Theta)$}. On écrit un élément $\psi$ de $C^\infty(\Theta)$ sous la forme $\psi = (\psi[\mathcal{O},P])_{\mathcal{O},P}$. On définit l'application $\kappa: C^\infty(\Theta) \to \mathcal{C}(\tilde{G})$ par
$$ \kappa(\psi) = \sum_{(\mathcal{O},P) \in \Theta} \gamma(G|M) |W^M_0| |W^G_0|^{-1} |\mathcal{P}(M)|^{-1} f_{\psi[\mathcal{O},P]}. $$

On note $C^\infty(\Theta)^\text{inv}$ le sous-espace de $C^\infty(\Theta)$ des éléments $\psi$ tels que
$$ \psi[s\mathcal{O}, P']_{s\omega} = {}^\circ c_{\tilde{P}'|\tilde{P}}(s,\omega) \psi[\mathcal{O},P]_\omega $$
pour tout $(\mathcal{O},P) \in \Theta$ et tout $P'$. On définit la projection $\text{pr}^\text{inv}: C^\infty(\Theta) \to C^\infty(\Theta)^\text{inv}$  par
$$ (\text{pr}^\text{inv} \psi)[\mathcal{O},P]_\omega = |W^G_0|^{-1} |\mathcal{P}(M)|^{-1} \sum_{s \in W^G_0} \sum_{P' \in \mathcal{P}(sM)} {}^\circ c_{\tilde{P}'|\tilde{P}}(s,\omega)^{-1} \psi[s\mathcal{O}, P']_{s\omega}. $$

Un fait important est que $\kappa$ se factorise par $\text{pr}^\text{inv}$; on utilise toujours le symbole $\kappa$ pour l'application $C^\infty(\Theta)^\text{inv} \to \mathcal{C}(\tilde{G})$. Voir \cite[VI.3.2]{Wa03}.

L'étape suivante est de définir l'inverse de $\kappa$. Soit $f \in \mathcal{C}(\tilde{G})$, notons $\check{f}$ la fonction $\tilde{x} \mapsto f(\tilde{x}^{-1})$. Alors $\check{f} \in \mathcal{C}(\tilde{G})$. Soient $M \in \mathcal{L}(M_0)$ et $P \in \mathcal{P}(M)$. Soit $(\omega,E) \in \Pi_2(\tilde{M})$, notons $(\pi,V) := \mathcal{I}_{\tilde{P}}(\omega,E)$ et
$$ \hat{f}(\omega, \tilde{P}) := d(\omega) \pi(\check{f}). $$

On vérifie que $\hat{f}(\omega, \tilde{P}) \in L(\omega,\tilde{P})$ comme dans \cite[VII.1]{Wa03}. Soit $(\mathcal{O},P) \in \Theta$, on obtient ainsi une fonction $\psi_f[\mathcal{O},P]$ sur $\mathcal{O}$ par $\psi_f[\mathcal{O},P]_\omega = \hat{f}(\omega,P)$, pour tout $\omega \in \mathcal{O}$.

\begin{proposition}
  Pour tout $(\mathcal{O},P) \in \Theta$, l'application $f \mapsto \psi_f[\mathcal{O},P]$ définit une application linéaire continue $\mathcal{C}(\tilde{G}) \to C^\infty(\mathcal{O},\tilde{P})$.
\end{proposition}

On arrive à l'énoncé de la formule de Plancherel.
\begin{theorem}\label{prop:Plancherel}
  L'application $f \mapsto \psi_f[\mathcal{O},P]$ induit une application $\mathcal{C}(\tilde{G}) \to C^\infty(\Theta)^\mathrm{inv}$. Elle est l'inverse bilatéral de $\kappa$. En particulier, $\kappa$ est un isomorphisme.
\end{theorem}

Pour $(\mathcal{O}, P) \in \Theta$ et $\omega \in \mathcal{O}$, on note $\Theta^{\tilde{G}}_\omega$ le caractère de $\mathcal{I}_{\tilde{P}}(\omega)$.

\begin{corollary}\label{prop:Plancherel-f(1)}
  Soit $f \in \mathcal{C}(\tilde{G})$, alors
  \begin{align*}
    f(1) & = \sum_{(\mathcal{O},P) \in \Theta} \gamma(G|M) |W^M_0| |W^G_0|^{-1} |\mathcal{P}(M)|^{-1} \int_{\mathcal{O}} d(\omega) \mu(\omega) \Theta^{\tilde{G}}_\omega(\check{f}) \dd\omega ,\\
    & = \sum_{M \in \mathcal{L}(M_0)} \gamma(G|M) |W^M_0| |W^G_0|^{-1} \int_{\Pi_2(\tilde{M})} d(\omega) \mu(\omega) \Theta_\omega^{\tilde{G}}(\check{f}) \dd\omega .
  \end{align*}
\end{corollary}

\begin{remark}\label{rem:Plancherel-mesures}
  Afin de se débarrasser de tout souci de mesures, vérifions la dépendance de ladite formule sur des divers choix. Fixons $(\mathcal{O},P) \in \Theta$ avec $P=MU$.
  \begin{itemize}
    \item Le degré formel $d(\omega)$ pour $\omega \in \mathcal{O}$ est inversement proportionnelle à la mesure $\dd m$ sur $\tilde{M}$.
    \item D'après la définition des opérateur d'entrelacement, la fonction $\mu$ est proportionnelle à $(\dd \bar{u} \dd u)^{-1}$ où $\dd u$ (resp, $\dd \bar{u}$) désigne la mesure sur $U(F)$ (resp. $\bar{U}(F)$).
    \item $\Theta_\omega^{\tilde{G}}(f)$ est proportionnel à la mesure $\dd g$ sur $\tilde{G}$.
    \item $\gamma(G|M)$ est proportionnel à $\dd \bar{u} \dd m \dd u (\dd g)^{-1}$ pourvu que $\gamma(G|M)$ soit défini par la formule \eqref{eqn:UMU} dans \S\ref{sec:def-base}.
    \item Il n'y a aucune dépendance du choix de $A_M(F)^\dagger$.
  \end{itemize}

  Par conséquent, le produit indexé par $(\mathcal{O},P)$ dans le Corollaire \ref{prop:Plancherel-f(1)} est indépendant de tout choix. Cela nous permettra de varier certains choix de mesures dans \S\ref{sec:formule-traces}.
\end{remark}

\section{Normalisation des opérateurs d'entrelacement}
On se place toujours dans le cadre d'un revêtement local
$$ 1 \to \bmu_m \to \tilde{G} \stackrel{\rev}{\to} G(F) \to 1, $$
où $F$ est un corps local et $G$ est un $F$-groupe réductif connexe. On fixe un sous-groupe de Lévi minimal $M_0$ et un sous-groupe compact maximal spécial $K$ de $G(F)$ en bonne position relativement à $M_0$. On a les ensembles $\Pi(\tilde{M})$, $\Pi_\text{temp}(\tilde{M})$, $\Pi_2(\tilde{M})$ de classes d'équivalence des représentations irréductibles de $\tilde{M}$; ils sont défini dans \S\ref{sec:Plancherel} lorsque $F$ est non archimédien, dans le cas $F$ archimédien leurs définitions sont bien connues. En rajoutant l'équivariance sous $\bmu_m$, on définit leurs variantes spécifiques $\Pi_-(\tilde{M})$, $\Pi_{\text{temp},-}(\tilde{M})$, $\Pi_{2,-}(\tilde{M})$, etc.

Nous avons précisé les choix de mesures dans le cas non archimédien. Pour le cas archimédien, nous suivons le choix fait dans \cite{Ar89-IOR1} qui sera rappelé dans \S\ref{sec:normalisation-arch}.

\subsection{Facteurs normalisants}
Soit $M \in \mathcal{L}(M_0)$. Rappelons que, pour tout corps local $F$, on peut définir le groupe $X(\tilde{M}) := \Hom(\tilde{M}/\tilde{M}^1, \C^\times)$ et son sous-groupe $\Im X(\tilde{M})$. Ces groupes opèrent sur $\Pi(\tilde{M})$. Les $X(\tilde{M})$-orbites admettent une structure de variété complexe algébrique via la surjection canonique $\mathfrak{a}_{M,\C}^* \twoheadrightarrow X(\tilde{M})$. Soit $\pi \in \tilde{M}$, l'action de $X(\tilde{M})$ s'écrit sous la forme habituelle $(\pi, \chi) \mapsto \pi \otimes \chi$ avec $\chi \in X(\tilde{M})$, ou $(\pi, \lambda) \mapsto \pi_\lambda$ avec $\lambda \in \mathfrak{a}_{M,\C}^*$. Si $F$ est archimédien, l'action de $\mathfrak{a}_{M,\C}^*$ est libre et tout élément de $\Pi_2(\tilde{M})$ admet une écriture unique $\pi_\lambda$ avec $\pi \in \Pi_2(\tilde{M}^1)$ et $\lambda \in i\mathfrak{a}_M^*$.

\begin{definition}[cf. {\cite[\S 2]{Ar89-IOR1}} et {\cite[\S 2]{Ar94}}]\label{def:normalisant}\index[iFT2]{facteur normalisant}\index[iFT2]{$r_{\tilde{P}'\arrowvert\tilde{P}}(\pi), R_{\tilde{P}'\arrowvert\tilde{P}}(\pi)$}
  Soient $M \in \mathcal{L}(M_0)$ et $\mathcal{O}_\C$ une $X(\tilde{M})$-orbite dans $\Pi(\tilde{M})$. On dit qu'une famille de fonctions méromorphes sur $\mathcal{O}_\C$
  $$ r^{\tilde{L}}_{\tilde{P}'|\tilde{P}}(\pi), \quad L \in \mathcal{L}(M),\; P,P' \in \mathcal{P}^L(M), \; \pi \in \mathcal{O}_\C $$
  est une famille de facteurs normalisants si elle satisfait aux conditions suivantes. Posons d'abord
  $$ R^{\tilde{L}}_{\tilde{P}'|\tilde{P}}(\pi) := r^{\tilde{L}}_{\tilde{P}'|\tilde{P}}(\pi)^{-1} J^{\tilde{L}}_{\tilde{P}'|\tilde{P}}(\pi). $$
  \renewcommand{\labelenumi}{(\textbf{R\arabic{enumi}})}\begin{enumerate}
    \item Pour tous $P, P', P'' \in \mathcal{P}^L(M)$, on a $R^{\tilde{L}}_{\tilde{P}''|\tilde{P}}(\pi) = R^{\tilde{L}}_{\tilde{P}''|\tilde{P}'}(\pi) R^{\tilde{L}}_{\tilde{P}'|\tilde{P}}(\pi)$.
    \item Si $\pi$ est unitaire, alors
      $$ R^{\tilde{L}}_{\tilde{P}'|\tilde{P}}(\pi_\lambda)^* = R^{\tilde{L}}_{\tilde{P}|\tilde{P}'}(\pi_{-\bar{\lambda}}), \qquad \lambda \in \mathfrak{a}_{M,\C}^* . $$
      En particulier, $R^{\tilde{L}}_{\tilde{P}'|\tilde{P}}(\pi)$ est un opérateur unitaire.
    \item Cette famille est compatible au transport de structure par le groupe de Weyl, au sens suivant
      $$ R^{\tilde{L}}_{w\tilde{P}'|w\tilde{P}}(\tilde{w}\pi) = A(\tilde{w}) R^{\tilde{L}}_{\tilde{P}'|\tilde{P}}(\pi) A(\tilde{w})^{-1} $$
      où $w \in W^L_0$ et $\tilde{w} \in \tilde{K}$ est un représentant, cf. la Proposition \ref{prop:prop-entrelacement}.
    \item Avec les notations dans \S\ref{sec:op-entrelacement}, on a
      $$ r^{\tilde{L}}_{\tilde{P}'|\tilde{P}}(\pi) = \prod_{\alpha \in (\Sigma^L_{P'})^\text{red} \cap (\Sigma^L_{\bar{P}})^\text{red}} r^{\tilde{M}_\alpha}_{\tilde{\bar{P}}_\alpha|\tilde{P}_\alpha}(\pi), $$
      où $P_\alpha := P \cap M_\alpha$.
    \item Soit $P''=M'' U'' \in \mathcal{F}^L(M_0)$ contenant $P$ et $P'$, alors $R^{\tilde{L}}_{\tilde{P}'|\tilde{P}}(\pi)$ est l'opérateur déduit de $R^{\tilde{M}''}_{\tilde{P}' \cap \tilde{M}''|\tilde{P} \cap \tilde{M}''}(\pi)$ par le foncteur $\mathcal{I}^{\tilde{L}}_{\tilde{P}''}(\cdot)$.
    \item Si $F$ est archimédien, les coefficients $\tilde{K} \cap \tilde{L}$-finis de $R^{\tilde{L}}_{\tilde{P}'|\tilde{P}}(\pi_\lambda)$ sont des fonctions rationnelles  en $\angles{\lambda, \alpha^\vee}$ où $\alpha \in \Delta^L_P$ . Si $F$ est non archimédien, notons $q$ le cardinal du corps résiduel de $q$, alors $r^{\tilde{L}}_{\tilde{P}'|\tilde{P}}(\pi_\lambda)$ est une fonction rationnelle en les variables $q^{-\angles{\lambda, \check{\alpha}}}$, où $\alpha \in \Delta^L_P$ et $\check{\alpha} \in \Q_{>0} \cdot \alpha^\vee$ est défini dans \eqref{eqn:alpha-check}; donc les coefficients de $R^{\tilde{L}}_{\tilde{P}'|\tilde{P}}(\pi_\lambda)$ le sont aussi.
    \item Si $\pi$ est tempérée, $\lambda \mapsto r_{\tilde{P}'|\tilde{P}}(\pi_\lambda)$ n'a ni zéros ni pôles lorsque $\angles{\Re\lambda, \alpha^\vee}>0$ pour tout $\alpha \in \Delta^L_P$.
    \item Si $F$ est archimédien, $\pi \in \Pi_\text{temp}(\tilde{M}^1)$ et $\lambda \in i\mathfrak{a}_M^*$, on pose
      $$ q^{\tilde{L}}_{\tilde{P}'|\tilde{P}}(\pi_\lambda) := \prod_{\alpha \in (\Sigma^L_{P'})^\text{red} \cap (\Sigma^L_{\bar{P}})^\text{red}} \angles{\lambda, \alpha^\vee}^{n_\alpha} $$
      où $n_\alpha$ est l'ordre du pôle de $\lambda \mapsto r_\alpha(\pi_\lambda)$ en $\lambda=0$. Alors pour tout opérateur différentiel invariant $D$ sur $i\mathfrak{a}_M^*$, il existe des constantes $C, N > 0$ telles que
      $$ |D (q^{\tilde{L}}_{\tilde{P}'|\tilde{P}}(\pi_\lambda) r^{\tilde{L}}_{\tilde{P}'|\tilde{P}}(\pi_\lambda))^{-1}| \leq C (1+ \|\mu_\pi + \lambda \|)^N $$
      pour tout $\pi$ et tout $\lambda$, où nous adoptons les notations
      \begin{itemize}
        \item $\mathfrak{h}$ est une sous-algèbre de Cartan de $\mathfrak{g}$;
        \item $W^M_\C$ le groupe de Weyl complexe de $M$;
        \item $\mu_\pi \in \mathfrak{h}_\C^*/W^M_\C$ est le caractère infinitesimal de $\pi$;
        \item $\|\cdot\|$ une norme hermitienne $W^M_\C$-invariante sur $\mathfrak{h}_\C$ telle que
          $$ \| \mu_{\pi_\lambda} \|^2 = \|\mu_\pi + \lambda \|^2 = \|\mu_\pi\|^2 + \|\lambda\|^2, \quad \forall \lambda \in i\mathfrak{a}_M^* . $$
      \end{itemize}
      Signalons que le caractère infinitesimal ici est défini de la même façon que dans le cas des groupes réductifs connexes sur $\R$, voir \cite[1.6.5 et 3.2.3]{Wall88}.
  \renewcommand{\labelenumi}{\arabic{enumi}}\end{enumerate}
\end{definition}

La condition (\textbf{R7}) interviendra dans l'étude des quotients de Langlands. Dans cet article nous ferons usage d'une notion moins stricte.

\begin{definition}\label{def:normalisant-faible}\index[iFT2]{facteur normalisant!faible}\index[iFT2]{facteur normalisant!complémentaire}
  On dit qu'une famille de fonctions méromorphes sur $\mathcal{O}_\C$
  $$ r^{\tilde{L}}_{\tilde{P}'|\tilde{P}}(\pi), \quad L \in \mathcal{L}(M),\; P,P' \in \mathcal{P}^L(M), \; \pi \in \mathcal{O}_\C $$
  est une famille de facteurs normalisants faible si elle vérifie toutes les conditions de la Définition \ref{def:normalisant} sauf (\textbf{R7}). Deux familles de facteurs normalisants faibles $r$, $r^\vee$ sont dites complémentaires si
  \begin{gather}\label{eqn:normalisation-c}
    r^{\tilde{L},^\vee}_{\tilde{P}'|\tilde{P}}(\pi) = r^{\tilde{L}}_{\tilde{P}|\tilde{P}'}(\pi)
  \end{gather}
  pour tous $P, P'$ et $L$.
\end{definition}
\begin{remark}
  Étant donnée une famille de facteurs normalisants faible $r$, on peut toujours définir une famille complémentaire $r^\vee$ par \eqref{eqn:normalisation-c}. Mais $r^\vee$ et $r$ ne peuvent pas vérifier à la fois (\textbf{R7}): on peut le voir en prenant $P'$ égal à l'opposé de $P$.
\end{remark}

La construction des facteurs normalisants se réduit au cas $L=G$, $\mathcal{O}_\C \cap \Pi_{2}(\tilde{M}) \neq \emptyset$ et $P, P'$ des paraboliques propres maximaux, au sens suivant.

\begin{proposition}\label{prop:normalisant-red}
  Supposons choisie des familles de facteurs normalisants $r^{\tilde{L}}_{\tilde{Q}'|\tilde{Q}}(\cdot)$ pour tout $L \in \mathcal{L}(M_0)$, $L \neq G$ et $Q, Q'$ quelconques, qui sont compatibles au transport de structure par $W^G_0$.

  Soient $r_{\tilde{P}'|\tilde{P}}(\cdot)$ des fonctions méromorphes sur les $X(\tilde{M})$-orbites rencontrant $\Pi_2(\tilde{M})$ vérifiant toutes les conditions de la Définition \ref{def:normalisant} sauf (\textbf{R4}) et (\textbf{R5}), où $M \in \mathcal{L}(M_0)$ est standard propre maximal, $P, P' \in \mathcal{P}(M)$. Alors les facteurs $r^{\tilde{L}}_{\tilde{Q}'|\tilde{Q}}(\cdot)$ et $r_{\tilde{P}'|\tilde{P}}(\cdot)$ ci-dessus font partie d'une famille de facteurs normalisants pour $\tilde{G}$, dont la construction est canonique.
\end{proposition}
\begin{proof}
  C'est exactement ce qu'Arthur fait dans \cite[\S 2]{Ar89-IOR1}. En résumé, grâce à (\textbf{R5}), la condition (\textbf{R7}) et un critère simple de l'unitarisabilité des quotients de Langlands \cite[Theorem 7]{KZ77}, qui est valable pour les revêtements sur tout corps local, permettent de se ramener au cas tempérée. On se ramène finalement au cas de représentations de carré intégrable modulo le centre d'après le fait qu'une représentation tempérée est un constituant d'une induite parabolique normalisée d'une représentation de carré intégrable modulo le centre. La compatibilité au transport de structure est garantie par (\textbf{R3}). À chaque étape, on peut supposer $M$ propre maximal grâce à (\textbf{R4}). La préservation de (\textbf{R8}) est claire.
\end{proof}

\begin{remark}
  Pour les applications à la formule des trace, on considérera le cadre où $F$ est un corps de nombres, $\rev: \tilde{G} \to G(\A)$ un revêtement adélique, $S$ un ensemble fini de places de $F$ et $\tilde{G}_S \to G(F_S)$ la fibre de $\rev$ au-dessus de $G(F_S)$. On peut toujours considérer les normalisations des opérateurs d'entrelacement dans ce cadre, cf. \cite[\S 1]{Ar89-IOR1}. Les résultats découlent sans peine du cas $F$ local en prenant
  \begin{align*}
    \pi & = \bigotimes_{v \in S} \pi_v, \\
    r_{\tilde{P}'|\tilde{P}}(\pi) & = \prod_{v \in S} r_{\tilde{P}'|\tilde{P}}(\pi_v), \\
    R_{\tilde{P}'|\tilde{P}}(\pi) & = \bigotimes_{v \in S} R_{\tilde{P}'|\tilde{P}}(\pi_v).
  \end{align*}
  Cf. \cite[p.31]{Ar89-IOR1}.
\end{remark}

\subsection{Le cas archimédien}\label{sec:normalisation-arch}
Conservons les conventions de \S\ref{sec:normalisation-arch}. Nous donnerons une construction canonique de facteurs normalisants dans le cas archimédien à la Arthur \cite{Ar89-IOR1} avec formules explicites. C'est loisible de supposer que $F=\R$. 

Soit $\mathfrak{r}$ une $\R$-algèbre de Lie. On adopte la convention $\mathfrak{r}_\C := \mathfrak{r}(\R) \otimes_\R \C$. Notons $\theta$ l'involution de Cartan associé à $K$. Prenons une forme bilinéaire $G(\R)$-invariante $B$ sur $\mathfrak{g}(\R)$ telle que $X \mapsto -B(X,\theta X)$ est positive définie. Si $T$ est un $\R$-tore maximal $\theta$-stable de $M$, alors $B$ est non dégénérée sur $\mathfrak{t}(\R)$. Prenons $T$ un $\R$-tore fondamental dans $M$ anisotrope modulo $A_M$, ce qui existe d'après \cite[\S 39]{HC66} car $\Pi_2(\tilde{M}) \neq \emptyset$.  Le groupe $\mathrm{Gal}(\C/\R) = \{\identity, \sigma\}$ opère sur $X_*(T_\C)$, d'où une action sur $\mathfrak{t}_\C$.

On se ramène au cas où $M$, $P$, $P'$, $\pi$ sont comme dans la Proposition \ref{prop:normalisant-red}, en particulier $P' = \bar{P}$. Or la notation $P'$ sera conservée pour des raisons de typographie.

On pose
$$  \alpha_{P'|P} := \prod_{\alpha} \sqrt{\frac{1}{2} B(\alpha, \alpha)} $$
où $\alpha$ parcourt les racines de $(\mathfrak{g}_\C, \mathfrak{t}_\C)$ dont les restrictions sur $\mathfrak{a}_M$ appartiennent à $\Sigma_{P'}$. On munit $(\mathfrak{u}_{P'} \cap \mathfrak{u}_P)(\R)$ de la mesure induite par la forme positive définie $X \mapsto -B(X, \theta X)$; on choisit la mesure de Haar sur $(U_{P'} \cap U_P)(\R)$ de sorte que
$$ \int_{(U_{P'} \cap U_P)(\R)} \phi(u) \dd u =  \alpha_{P'|P} \int_{(\mathfrak{u}_{P'} \cap \mathfrak{u}_P)(\R)} \phi(\exp X) \dd X, \quad \forall \phi \in C_c((U_{P'} \cap U_P)(\R)). $$
On montre que cette mesure ne dépend pas du choix de $B$.

Notons $\rho_M$ la demi-somme des racines positives de $(\mathfrak{m}_\C, \mathfrak{t}_\C)$ par rapport à une base fixée. On prend $\dd\chi \in \mathfrak{t}^*_\C$ tel que $\dd\chi+\rho_M$ est un représentant du caractère infinitesimal de $\pi$. On peut prendre $\lambda_0 \in \mathfrak{t}^*_\C$, $\mu \in i\mathfrak{t}^*$ tels que
$$ \angles{\dd\chi, H} = \angles{\lambda_0, H - \sigma H} + \frac{1}{2}\angles{\mu - \rho_M, H + \sigma H}, \qquad H \in \mathfrak{t}(\C) , $$
Cf. \cite[(A.1)]{Ar89-IOR1}. Ici $\lambda_0$ est pour l'essentiel la dérivée du caractère central de $\pi$, et $\mu$ est le paramètre de Harish-Chandra. D'après Harish-Chandra, c'est connu que $\dd\chi$ se relève en un caractère de $\tilde{T}$; notons-le $\chi$.

Notons $\Sigma_P(G,T)$ l'ensemble des racines de $(G,T)$ dont les restrictions sur $\mathfrak{a}_M$ appartiennent à $\Sigma_P$ et étudions les $\{\identity,\sigma\}$-orbites de $\Sigma_P(G,T)$. Supposons d'abord que $\alpha$ est une racine réelle dans $\Sigma_P(G,T)$. Il y a au plus une telle racine. Rappelons la définition de l'élément $\gamma \in \tilde{T}$ dans \cite[\S 30, Lemma 2]{HC76}. On note $H_\alpha$ l'élément dans $\mathfrak{t}(\R)$ tel que $B(H,H_\alpha)=\angles{\alpha,H}$ pour tout $H \in \mathfrak{t}(\R)$, et on note $H' := 2 \angles{\alpha, H_\alpha}^{-1} H_\alpha$. Prenons $X' \in \mathfrak{g}_\alpha(\R)$ et $Y' \in \mathfrak{g}_{-\alpha}(\R)$ tels que $(H',X',Y')$ est un $\mathfrak{sl}_2$-triplet, c'est-à-dire
$$ [H',X']=2X', \; [H', Y']=-2Y', \; [X',Y']=H' . $$

Posons $\gamma := \exp(\pi(X'-Y')) \in \tilde{T}$. Le triplet $(H',X',Y')$ induit un homomorphisme $\Phi_\alpha: \widetilde{SL}(2,\R) \to \tilde{G}$, où $\widetilde{SL}(2,\R)$ est un revêtement de $\SL(2,\R)$. Prenons $k \in \{0, \ldots, 2m-1\}$ tel que
$$ (-1)^{\angles{\rho_P, \alpha^\vee}} \chi(\gamma) = e^{\frac{\pi k}{mi}}. $$

D'autre part, si $\alpha$ est une racine complexe, en remplaçant $\alpha$ par $\sigma\alpha$ si besoin est, on suppose toujours que
$$ \angles{\sigma\mu-\mu, \alpha^\vee} \in \Z_{\leq 0}. $$

Définissons le facteur normalisant $r_{\tilde{P}'|\tilde{P}}(\pi)$. S'il existe une racine réelle dans $\Sigma_P(G,T)$, on la notera $\alpha_0$ et on écrira $\exists \alpha_0$; dans ce cas-là on définit $\chi$, $\gamma$ et $k$ d'après ce qui précède. Posons\index[iFT2]{$\Gamma_\C$}
\begin{align*}
  \Gamma_\C(z) & := 2(2\pi)^{-z} \Gamma(z), \\
  G(z) & := \sqrt{\pi} \cdot \frac{\Gamma(\frac{z}{2}) \Gamma(\frac{z}{2}+\frac{1}{2})}{\Gamma(\frac{z}{2} + \frac{k}{2m}) \Gamma(\frac{z}{2} + 1 - \frac{k}{2m})}, \quad \text{si } \exists \alpha_0 ;
\end{align*}
(et désolé pour l'abus du symbole $\pi$). Pour tout $\lambda \in \mathfrak{a}_{M,\C}$ on définit
\begin{equation}
  r_{\tilde{P}'|\tilde{P}}(\pi_\lambda) :=
  \begin{cases}
    \prod_{\substack{\alpha \in \Sigma_P(G,T) \\ \mod \{\identity, \sigma\} \\ \alpha \neq \alpha_0}} \dfrac{\Gamma_\C(\angles{\mu+\lambda, \alpha^\vee})}{\Gamma_\C(\angles{\mu+\lambda,\alpha^\vee}+1)} \cdot G(\angles{\mu+\lambda, \alpha_0^\vee}), & \text{si } \exists\alpha_0, \\
    & \\
    \prod_{\substack{\alpha \in \Sigma_P(G,T) \\ \mod \{\identity, \sigma\}}} \dfrac{\Gamma_\C(\angles{\mu+\lambda, \alpha^\vee})}{\Gamma_\C(\angles{\mu+\lambda,\alpha^\vee}+1)}, & \text{sinon},
  \end{cases}
\end{equation}
où les représentants $\alpha$ des $\{\identity,\sigma\}$-orbites dans $\Sigma_P(G,T)$ sont choisis comme précédemment. C'est méromorphe en $\lambda$, on définit ainsi les opérateurs $R_{\tilde{P}'|\tilde{P}}(\pi_\lambda) := r_{\tilde{P}'|\tilde{P}}(\pi_\lambda)^{-1} J_{\tilde{P}'|\tilde{P}}(\pi_\lambda)$ qui sont méromorphes en $\lambda$.

\begin{theorem}
  Les facteurs $r_{\tilde{P}'|\tilde{P}}(\pi_\lambda)$ font partie d'une famille de facteurs normalisants dont la construction est canonique.
\end{theorem}
\begin{proof}
  On utilise toujours la Proposition \ref{prop:normalisant-red}. On note $\Sigma_c(G,T) := \Sigma_P(G,T) \setminus \{\alpha_0\}$ si $\exists\alpha_0$. C'est une conséquence de la formule explicite de la fonction $\mu$ \cite[\S 36]{HC76} (voir le récit dans \cite[Proposition 3.1 + Lemma A.1]{Ar89-IOR1}), qui est aussi valable pour revêtements, que
  $$ \mu(\pi) = \begin{cases}
    \gamma_{P'|P}^2 \alpha_{P'|P}^2 (2\pi)^{-\dim U_P} \frac{\mu_0(\chi)}{\pi} |\prod_{\alpha \in \Sigma_c(G,T)} \angles{\mu, \alpha^\vee} |, & \text{si } \exists \alpha_0, \\
    & \\
    \gamma_{P'|P}^2 \alpha_{P'|P}^2 (2\pi)^{-\dim U_P} |\prod_{\alpha \in \Sigma_P(G,T)} \angles{\mu, \alpha^\vee} |, & \text{sinon},
  \end{cases}$$
  où $\mu_0(\chi)$ est l'expression
  $$
    \frac{\pi \angles{\mu, \alpha_0^\vee}}{i} \cdot \sinh\left( \frac{\pi\angles{\mu, \alpha_0^\vee}}{i} \right)
    \cdot \left[ \cosh\left( \frac{\pi\angles{\mu, \alpha_0^\vee}}{i} \right) - \frac{1}{2} (-1)^{\angles{\rho_P, \alpha_0^\vee}} (\chi(\gamma)+\chi(\gamma)^{-1}) \right]^{-1} .
  $$

  Avec le choix de mesures dans \cite[\S 3]{Ar89-IOR1}, on a $J_{\tilde{P}'|\tilde{P}}(\pi)^* J_{\tilde{P}'|\tilde{P}}(\pi) = \gamma_{P'|P}^2 \alpha_{P'|P}^2 \mu(\pi)^{-1}$ et pour démontrer (\textbf{R1}) il faut montrer l'égalité $J_{\tilde{P}'|\tilde{P}}(\pi)^* J_{\tilde{P}'|\tilde{P}}(\pi) = |r_{\tilde{P}'|\tilde{P}}(\pi)|^2$; le cas général où $\pi$ est remplacé par $\pi_\lambda$ ($\lambda \in \mathfrak{a}_{M,\C}^*$) en découlera par prolongation méromorphe.

  Supposons d'abord que $\exists \alpha_0$. Remarquons que
  $$ \frac{1}{2} (-1)^{\angles{\rho_P, \alpha_0^\vee}} (\chi(\gamma)+\chi(\gamma)^{-1}) = \cos\left( \frac{\pi k}{m} \right). $$

  Notons $v := \angles{\mu, \alpha_0^\vee} \in i\R$, on a
  \begin{align*}
    \mu_0(\chi)^{-1} & = \frac{\cosh(\frac{\pi v}{i}) - \cos\left(\frac{\pi k}{m}\right)}{ \frac{\pi v}{i}\sinh(\frac{\pi v}{i}) } = \frac{\cos(\pi v) - \cos\left(\frac{\pi k}{m}\right)}{ \frac{\pi v}{i}\sinh(\frac{\pi v}{i}) } \\
    & = \frac{2 \sin\left( \pi \left( \frac{v}{2} + \frac{k}{2m} \right) \right) \sin\left( \pi \left( -\frac{v}{2} + \frac{k}{2m} \right)\right) }{ \frac{v}{i}\sinh\left(\frac{v}{i}\right) }.
  \end{align*}

  D'une part, les formules $\Gamma(\bar{z})= \overline{\Gamma(z)}$ et $|\Gamma(iy)|^2 = \frac{\pi}{y \sinh(\pi y)}$ pour $y \in \R$ entraînent que le dénominateur est égal à $\left( \Gamma(v)\Gamma(-v) \right)^{-1} \pi^2$. D'autre part, il résulte de la formule de réflexion $\Gamma(1-z)\Gamma(z) = \frac{\pi}{\sin(\pi z)}$ que le numérateur est égal à
  $$ \left[ \Gamma\left( \frac{v}{2}+\frac{k}{2m} \right) \Gamma\left( 1-\frac{k}{2m}+\frac{v}{2} \right) \Gamma\left( -\frac{v}{2}+\frac{k}{2m} \right) \Gamma\left( 1-\frac{k}{2m}-\frac{v}{2} \right) \right]^{-1} 2\pi^2 .$$

  Donc $\mu_0(\chi)^{-1} = g(v)g(-v)$ où
  \begin{align*}
    g(z) & := \frac{\sqrt{2} \; \Gamma(z)}{ \Gamma\left( \frac{z}{2}+\frac{k}{2m} \right) \Gamma\left( \frac{z}{2}+1-\frac{k}{2m} \right)} \\
    & \\
    & = (\sqrt{2\pi})^{-1} \cdot 2^z \cdot \frac{\Gamma\left( \frac{z}{2} \right) \Gamma\left(\frac{z}{2}+\frac{1}{2} \right)}{ \Gamma\left( \frac{z}{2}+\frac{k}{2m} \right) \Gamma\left( \frac{z}{2}+1-\frac{k}{2m} \right) }
  \end{align*}
  où la deuxième égalité résulte de la formule de multiplication $\Gamma(z) = (\sqrt{\pi})^{-1} 2^{z-\frac{1}{2}} \Gamma(\frac{z}{2}) \Gamma(\frac{z}{2}+\frac{1}{2})$. On voit que $G(v)G(-v) = 2\pi^2 g(v)g(-v)$.

  Montrons (\textbf{R1}). Si $\exists \alpha_0$, la formule $|\Gamma_\C(iy)\Gamma_\C(1+iy)^{-1}|^2 = (2\pi)^2 |y|^{-2}$ pour $y \in \R$ et le raisonnement ci-dessus entraînent que
  \begin{align*}
    |r_{\tilde{P}'|\tilde{P}}(\pi)|^2 &= \prod_{\alpha \neq \alpha_0} (2\pi)^2 |\angles{\mu, \alpha}|^{-2} \cdot G(\pi\angles{\mu, \alpha_0^\vee}) G(-\pi\angles{\mu, \alpha_0^\vee}) \\
    & = \prod_{\alpha \in \Sigma_c(G,T)} (2\pi) |\angles{\mu, \alpha}|^{-1} \cdot g(\angles{\mu, \alpha_0^\vee}) g(-\angles{\mu, \alpha_0^\vee}) \cdot 2\pi^2 \\
    & = (2\pi)^{\dim U_P} \prod_{\alpha \in \Sigma_c(G,T)} |\angles{\mu, \alpha}|^{-1} \frac{\pi}{\mu_0(\chi)} \\
    & = \gamma_{P'|P}^2 \alpha_{P'|P}^2 \; \mu(\pi)^{-1}.
  \end{align*}

  Si $\nexists \alpha_0$, les mêmes manipulations des fonctions $\Gamma_\C$ donnent
  $$ |r_{\tilde{P}'|\tilde{P}}(\pi)|^2 = (2\pi)^{\dim U_P} \prod_{\alpha \in \Sigma_P(G,T)} |\angles{\mu, \alpha^\vee}|^{-1} = \gamma_{P'|P}^2 \alpha_{P'|P}^2 \; \mu(\pi)^{-1}. $$

  Montrons (\textbf{R2}). Comme $\Gamma(\bar{z})=\overline{\Gamma(z)}$, il en résulte que $\overline{r_{\tilde{P}'|\tilde{P}}(\pi_\lambda)} = r_{\tilde{P}|\tilde{P}'}(\pi_{-\bar{\lambda}})$, ce qui suffit pour conclure.

  Si $w \in W^G(M)$ est l'élément non trivial, alors il envoie $P$ à $P'$, $\lambda$ à $-\lambda$ et $\mu$ à $-\mu$, d'où $r_{w\tilde{P}'|w\tilde{P}}(\tilde{w}\pi_{w\lambda}) = r_{\tilde{P}'|\tilde{P}}(\pi_\lambda)$. Cela entraîne (\textbf{R3}).

  Les conditions (\textbf{R6}) et (\textbf{R8}) résultent des arguments dans \cite[\S 3]{Ar89-IOR1}, car Arthur n'utilise que la formule du déterminant de L. Cohn \cite[10.4.4]{Wall92} et la propriété que $r_{\tilde{P}'|\tilde{P}}(\pi_\lambda)$ est de la forme
  $$ \frac{\prod_{i=1}^N \Gamma(c\lambda - a_i)}{\prod_{i=1}^N \Gamma(c\lambda - b_i)}, \quad c \in \R^\times, \; a_i, b_i \in \C , $$
  à une constante multiplicative près. La condition (\textbf{R7}) résulte d'une propriété bien connue des fonctions $\Gamma$ car on a supposé $k \geq 0$ dans le cas $\exists \alpha_0$.
\end{proof}

Lorsque $\rev: \tilde{G} \to G(\R)$ provient d'une $\mathbf{K}_2$-extension de Brylinski-Deligne \cite[\S 10]{BD01}, on a forcément $m = 1$ ou $2$ et les facteurs normalisants sont étroitement liés à ceux d'Arthur. Cela est contenu dans les remarques suivantes.

\begin{remark}
  Supposons que $\exists \alpha_0$ et $\frac{k}{2m} \in \{0, \frac{1}{2}\}$. Alors on a
  $$ \frac{(-1)^{\angles{\rho_P, \alpha_0^\vee}}}{2} \left(\chi(\gamma) + \chi(\gamma)^{-1} \right) = (-1)^{N_0} $$
  où $N_0 := 1$ si $\frac{k}{2m}=0$ et $N_0 := 0$ si $\frac{k}{2m}=\frac{1}{2}$. On pose\index[iFT2]{$\Gamma_\R$}
  $$\Gamma_\R(z) := \pi^{-\frac{z}{2}} \Gamma\left(\frac{z}{2}\right)$$
  et on vérifie que
  $$ G(z) = \frac{\Gamma_\R(z+N_0)}{\Gamma_\R(z+N_0+1)}. $$

  On a toujours $\frac{k}{2m} \in \{0, \frac{1}{2}\}$ lorsque $m=1$. En comparant la définition des facteurs normalisants et l'interprétation de la constante $N_0$ dans \cite[Appendix]{Ar89-IOR1}, il en résulte que nos facteurs normalisants sont exactement ceux d'Arthur dans le cas archimédien.
\end{remark}

\begin{remark}
  Supposons que $\exists \alpha_0$ et $\frac{k}{2m} \in \{\frac{1}{4}, \frac{3}{4}\}$. On déduit de la formule de duplication pour les fonctions $\Gamma$ que
  $$ G(z) = \frac{\sqrt{\pi} \; \Gamma\left( \frac{z}{2} \right) \Gamma\left( \frac{z}{2}+\frac{1}{2} \right)}{ \Gamma\left( \frac{z}{2} + \frac{1}{4}\right) \Gamma\left( \frac{z}{2}+\frac{3}{4} \right) } = \sqrt{2}\cdot \frac{\Gamma_\R(2z)}{\Gamma_\R(2z+1)}. $$
  D'où
  $$ r_{\tilde{P}'|\tilde{P}}(\pi_\lambda) = \prod_{\alpha \neq \alpha_0} \dfrac{\Gamma_\C(\angles{\mu+\lambda, \alpha^\vee})}{\Gamma_\C(\angles{\mu+\lambda,\alpha^\vee}+1)} \cdot \sqrt{2} \cdot \frac{\Gamma_\R(\angles{\mu+\lambda, 2\alpha_0^\vee})}{\Gamma_\R(\angles{\mu+\lambda, 2\alpha_0^\vee}+1)}. $$

  Au facteur $\sqrt{2}$ près, c'est le facteur normalisant d'Arthur pour les groupes réductifs connexes sauf que $\alpha_0^\vee$ est remplacé par $2\alpha_0^\vee$. Cela reflète un phénomène de ``renversement du diagramme de Dynkin'' qui se passe pour certains revêtements à deux feuillets. Cf. \cite{ABPTV07}.
\end{remark}

\subsection{Le cas non archimédien}
Supposons $F$ local non archimédien de corps résiduel $\F_q$. Le résultat suivant ainsi que sa démonstration sont dus à \cite[Lecture 15]{CLL}, à quelques corrections près.

\begin{theorem}
  Il existe une famille de facteurs normalisants pour $\tilde{G}$.
\end{theorem}
\begin{proof}
  On peut supposer que $M$ est un Lévi propre maximal de $G$, $P,P' \in \mathcal{P}(M)$ tels que $P'=\bar{P}$ et $\pi \in \Pi_2(\tilde{M})$ grâce à la Proposition \ref{prop:normalisant-red}. Vu la Remarque \ref{rem:dep-entrelaceur} il suffit de considérer les représentations $\pi \otimes \chi$ où $\chi$ provient de $(\mathfrak{a}^G_{M,\C})^*$ via l'application
  $$ \mathfrak{a}_{M,\C}^* \to X(\tilde{M}). $$
  De telles $\chi$ forment une sous-tore complexe $X'$ de $X(\tilde{M})$. On note $\alpha$ l'unique élément de $\Delta_P$, et $\check{\alpha}$ l'élément dans $\Q_{>0} \cdot \alpha^\vee$ défini dans \eqref{eqn:alpha-check}. L'application $(\mathfrak{a}^G_{M,\C})^* \to \C^\times$ définie par $\lambda \mapsto z = q^{-\angles{\lambda,\check{\alpha}}}$ permet d'identifier $X'$ à $\C^\times$. On écrit l'action par $X'$ par $\pi \mapsto \pi \otimes z$. On définit ainsi $\mu_P(\pi,z) \in \C(z)$ telle que $\mu_P(\pi, z) := \mu(\pi \otimes z)$.

  Soit $P(z) \in \C(z)$, on adopte la convention
  $$ P(z)^* := \overline{P(\bar{z}^{-1})} \in \C(z). $$
  On construira $V_P(\pi, z) \in \C(z)$ telle que
  \begin{itemize}
    \item $V_P(\pi, z)$ n'a ni zéros ni pôles dans la région $0<|z|<1$;
    \item $\mu_P(\pi,z) = V_P(\pi, z) V_P(\pi, z)^*$;
    \item $V_P(\pi, z)$ est déterminé par $\pi \otimes z$.
  \end{itemize}
  Si ces conditions sont satisfaites, on pose $r_{\tilde{P}'|\tilde{P}}(\pi_\lambda) := V_P(\pi, q^{-\angles{\lambda, \check{\alpha}}})$.

  Fixons $\pi \in \Pi_2(\tilde{M})$ et construisons $V_P(\pi,z)$. On sait que $\mu_P(\pi,z) = \mu_P(\pi,z)^*$ et $\mu_P(\pi,z) \geq 0$ si $|z|=1$. D'où
  \begin{itemize}
    \item la distribution des zéros (resp. pôles) dans $\C \cup \{\infty\}$ est symétrique par rapport à l'inversion $z \mapsto \bar{z}^{-1}$;
    \item les zéros (resp. pôles) dans le cercle $|z|=1$ ont multiplicités paires.
  \end{itemize}

  Notons $\alpha_1^{-1}, \ldots, \alpha_r^{-1}$ (resp. $\beta_1^{-1}, \ldots, \beta_r^{-1}$) les zéros (resp. pôles) de $\mu(\pi,z)$ avec multiplicités, tels que $|\alpha_i| \leq 1$ (resp. $|\beta_i| \leq 1$) pour tout $i$. Montrons qu'il existe un unique $b \in \R_{>0}$ tel que
  $$ \mu_P(\pi,z) = b^2 \prod_{i=1}^r \frac{(1 - \alpha_i z)(z - \bar{\alpha}_i)}{(1 - \beta_i z)(z - \bar{\beta}_i)}. $$

  Si l'on pose $P(z; t) := z^{-1} (1- t z)$ pour $t \in \C$, alors $P(z; t)^* = z - \bar{t}$. On note $Q(z)$ le produit $\prod_{i=1}^r (\cdots)$ dans l'expression précédente, il s'écrit comme
  $$ Q(z) := \prod_{i=1}^r \frac{P(z; \alpha_i)P(z; \alpha_i)^*}{P(z; \beta_i)P(z; \beta_i)^*}. $$

  On a $Q(z)=Q(z)^*$, d'où $\text{ord}_{z=0} (Q(z)) = \text{ord}_{z=\infty} (Q(z))$. Idem pour $\mu_P(\pi, z)$. D'autre part $\text{ord}_{z=t} (Q(z)) = \text{ord}_{z=t} (\mu_P(\pi, z))$ pour tout $t \in \C$ avec $0 < |t| < \infty$ par construction. La formule de produit sur $\mathbb{P}^1_\C$ entraîne qu'il existe $a \in \C^\times$ tel que $a Q(z) = \mu_P(\pi,z)$. En l'évaluant en $z_0$ tel que $|z_0|=1$ et $\mu_P(\pi,z_0) > 0$, on voit que $a \in \R_{>0}$. On prend donc $b := \sqrt{a}$.

  Posons
  $$ V_P(\pi, z) := b \prod_{i=1}^r \frac{1 - \alpha_i z}{1 - \beta_i z}. $$
  Il en résulte immédiatement que $\mu_P(\pi,z) = V_P(\pi,z) V_P(\pi,z)^*$. Par construction $V_P(\pi,z)$ n'a ni zéros ni pôles dans la région $0<|z|<1$. Si l'on remplace $\pi$ par $\pi \otimes z_0$ où $|z_0|=1$, alors $\mu_P(\pi, z)$ est remplacé par $\mu_P(\pi \otimes z_0, z) = \mu_P(\pi, z_0 z)$ et $\alpha_i$ (resp. $\beta_i$) est remplacé par $z_0 \alpha_i$ (resp. $z_0 \beta_i$) pour tout $i$. Soit $t \in \C$, on a défini $P(z;t) \in \C(t)$ et on a
  \begin{align*}
    P(z; z_0 t) P(z; z_0 t)^* & = \frac{(1 - z_0 tz)(z - \overline{z_0 t})}{z} = \frac{(1 - t (z_0 z))(z_0 z - \bar{t} z_0 \bar{z}_0)}{z_0 z} \\
    & = \frac{1 - t (z_0 z)}{z_0 z} \cdot (z_0 z - \bar{t}) = P(z_0 z; t) P(z_0 z; t)^* .
  \end{align*}
  Attention: ici $P(z_0 z; t)^*$ signifie l'élément dans $\C(t)$ obtenu en remplaçant $z$ par $z_0 z$ dans $P(z; t)^*$. On conclut en prenant $t=\alpha_i, \beta_i$ ($i=1,\ldots,r$) que $V_P(\pi \otimes z_0, z) = V_P(\pi, z_0 z)$, donc $V_P(\pi, z)$ est déterminé par $\pi \otimes z$.

  Vérifions les conditions dans la Définition \ref{def:normalisant}. (\textbf{R1}) et (\textbf{R7}) sont déjà vérifiées. (\textbf{R3}) résulte du transport de structure car notre construction est canonique.

  Montrons (\textbf{R2}). Puisque $P'=\bar{P}$, on a $\mu_{P'}(\pi, z) = \mu_P(\pi, z^{-1})$, d'où
  $$ \mu_{P'}(\pi, z) = b^2 \prod_{i=1}^r \frac{(1 - \alpha_i z^{-1})(z^{-1} - \bar{\alpha}_i)}{(1 - \beta_i z^{-1})(z^{-1} - \bar{\beta}_i)} = b^2 \prod_{i=1}^r \frac{(z - \alpha_i)(1 - \bar{\alpha}_i z)}{(z - \beta_i)(1 - \bar{\beta}_i z)}. $$

  On en déduit que $V_{P'}(\pi, z) = b \prod_{i=1}^r (1 - \bar{\alpha}_i z)(1 - \bar{\beta}_i z)^{-1}$. Soit $\lambda \in \mathfrak{a}_{M,\C}^*$, on a
  \begin{align*}
     r_{\tilde{P}|\tilde{P}'}(\pi_\lambda) & = V_{P'}(\pi, q^{\angles{\lambda, \check{\alpha}}}) = b \prod_{i=1}^r \frac{1-\bar{\alpha}_i q^{\angles{\lambda,\check{\alpha}}}}{1-\bar{\beta}_i q^{\angles{\lambda,\check{\alpha}}}}, \\
    r_{\tilde{P}'|\tilde{P}}(\pi_{-\bar{\lambda}}) & = V_P(\pi, q^{\angles{\bar{\lambda}, \check{\alpha}}}) = b \prod_{i=1}^r \frac{1-\alpha_i q^{\angles{\bar{\lambda},\check{\alpha}}}}{1-\beta_i q^{\angles{\bar{\lambda},\check{\alpha}}}}.
  \end{align*}
  D'où $r_{\tilde{P}|\tilde{P}'}(\pi_\lambda) = \overline{r_{\tilde{P}'|\tilde{P}}(\pi_{-\bar{\lambda}})}$.
\end{proof}

\subsection{Le cas non ramifié}
On considère maintenant le cas $F$ local non archimédien, $K$ un sous-groupe hyperspécial de $G(F)$ en bonne position relativement à $M_0$, et $\rev: \tilde{G} \to G(F)$ un revêtement non ramifié au sens de \cite[\S 3.1]{Li10a}. On fixe une section $s: K \to \tilde{G}$ de $\rev$, par laquelle on identifie $K$ à un sous-groupe de $\tilde{G}$. Pour tout Lévi $M \in \mathcal{L}(M_0)$, les données $K^M := K \cap M(F)$ et $s|_{K^M}: K^M \rightarrow \tilde{M}$ définissent encore un revêtement non ramifié.

Une représentation admissible irréductible $(\pi,V)$ de $\tilde{G}$ est dite non ramifiée si $V^K \neq \{0\}$; dans ce cas-là on a $\dim V^K = 1$ d'après \cite[Corollaire 3.2.6]{Li10a}.

\begin{theorem}\label{prop:normalisant-nr}
  Il existe une famille de facteurs normalisants faibles
  $$ r^{\tilde{L}}_{\tilde{P}'|\tilde{P}}(\pi), \quad M \in \mathcal{L}(M_0), \; L \in \mathcal{L}(M) , \; P, P' \in \mathcal{P}^L(M), $$
  où $(\pi,V) \in \Pi(\tilde{M})$ est non ramifiée, tels que si $v \in \mathcal{I}_{\tilde{P}}(\pi)^K$, alors la restriction de $R^{\tilde{L}}_{\tilde{P}'|\tilde{P}}(\pi_\lambda)v$ à $\tilde{K}$ ne dépend pas de $\lambda$.
\end{theorem}

Avant d'entamer la preuve, rappelons brièvement la théorie des séries principales non ramifiées spécifiques. Soit $(\pi,V) \in \Pi_-(\tilde{M})$ non ramifiée. Il résulte de l'isomorphisme de Satake \cite[Proposition 3.2.5]{Li10a} que $\pi$ est une sous-représentation de $\text{Ind}^{\tilde{M}}_{\tilde{P}_0}(\chi)$, où $\chi$ est une représentation irréductible spécifique non ramifiée de $\tilde{M}_0$. Attention: $\tilde{M}_0$ n'est pas commutatif en général et $\chi$ n'est pas forcément dans $X(\tilde{M}_0)$. Néanmoins il admet une description simple: notons $\tilde{H} := Z_{\tilde{M}}(K \cap M_0(F))$ (voir \cite[Definition 3.1.1]{Li10a})\index[iFT2]{$\tilde{H}$}, c'est un sous-groupe commutatif maximal de $\tilde{M}_0$. Alors $\chi$ est de la forme
$$ \chi = \text{Ind}^{\tilde{M}_0}_{\tilde{H}}(\chi_0) $$
où $\chi_0$ est un caractère spécifique de $\tilde{H}$, trivial sur $K \cap \tilde{H}$; sa restriction à $Z(\tilde{M}_0)$ est uniquement déterminée par $\chi$. Cela résulte d'une variante du théorème de Stone-von Neumann, cf. \cite[\S 3]{We09}. Donc $\pi$ se réalise comme une sous-représentation de $\text{Ind}^{\tilde{M}}_{\tilde{H}}(\chi_0)$, ce que nous appelons une série principale non ramifiée spécifique\index[iFT2]{série principale non ramifiée spécifique}.

\begin{proof}[Démonstration du Théorème \ref{prop:normalisant-nr}]
  La restriction de $\pi$ au sous-groupe central $\bmu_m$ est un caractère $\bomega_{\pi}$. Quitte à pousser l'extension $1 \to \bmu_m \to \tilde{G} \to G(F) \to 1$ en avant par $\bomega_{\pi}$, on peut supposer $\pi$ spécifique. D'après ce qui précède, on réalise $\pi$ comme une sous-représentation irréductible de $\text{Ind}^{\tilde{M}}_{\tilde{H}}(\chi_0)$. Vu la transitivité de l'induction parabolique normalisée, on se ramène au cas $M=M_0$, $\pi = \text{Ind}^{\tilde{M}_0}_{\tilde{H}}(\chi_0)$. C'est aussi loisible de supposer que $L=G$.

  L'espace sous-jacent de $\mathcal{I}_{\tilde{P}}(\pi_\lambda)$ s'identifie à un espace de fonctions sur $\tilde{K}$ muni d'un produit hermitien invariant, noté $(|)$. Cet espace ne dépend pas de $\lambda$; son sous-espace de $K$-invariants est aussi indépendant de $\lambda$ et est de dimension $1$.

  Pour tout $P \in \mathcal{P}(M)$, prendre $v_P \in \mathcal{I}_{\tilde{P}}(\pi)^K$ la fonction telle que $v_P(1)=1$. On définit $r_{\tilde{P}|\tilde{P}'}(\pi_\lambda)$ de sorte que
  \begin{gather}\label{eqn:normalisant-nr}
    R_{\tilde{P}'|\tilde{P}}(\pi_\lambda) v_P = v_{P'}
  \end{gather}
  pour tous $P, P' \in \mathcal{P}(M)$ et presque tout $\lambda$. Le fait que $J_{\tilde{P}'|\tilde{P}}(\pi_\lambda)$ est rationnel en les variables $\{ q^{-\angles{\lambda, \check{\alpha}}} : \alpha \in \Delta_P \}$ entraîne que $r_{\tilde{P}'|\tilde{P}}(\pi_\lambda)$ l'est aussi. D'autre part \eqref{eqn:normalisant-nr} entraîne que $R_{\tilde{P}'|\tilde{P}}(\pi_\lambda) v_P$ est indépendant de $\lambda$. On vérifie (\textbf{R1}), (\textbf{R3}) et (\textbf{R6}) sans difficulté.

  Vérifions (\textbf{R2}). Les induites $\mathcal{I}_{\tilde{P}}(\pi_\lambda)$ et $\mathcal{I}_{\tilde{P}'}(\pi_\lambda)$ sont irréductibles si $\lambda$ est en position générale (voir les remarques après la Proposition \ref{prop:prop-entrelacement}), donc il existe $c_\lambda \in \C^\times$ tel que $R_{\tilde{P}'|\tilde{P}}(\pi_\lambda)^* = c_\lambda  R_{\tilde{P}|\tilde{P}'}(\pi_{-\bar{\lambda}})$. D'autre part, on a
  \begin{align*}
    (R_{\tilde{P}'|\tilde{P}}(\pi_\lambda) v_P | v_{P'}) & = (v_{P'}|v_{P'}), \\
    (v_P | R_{\tilde{P}|\tilde{P}'}(\pi_{-\bar{\lambda}}) v_{P'}) & = (v_P | v_P).
  \end{align*}
  Or $(v_{P'}|v_{P'})=(v_P|v_P)$, d'où $c_\lambda=1$ et $R_{\tilde{P}'|\tilde{P}}(\pi_\lambda)^* = R_{\tilde{P}|\tilde{P}'}(\pi_{-\bar{\lambda}})$.

  Les propriétés (\textbf{R4}), (\textbf{R5}) découlent des propriétés parallèles des opérateurs $J_{\tilde{P}'|\tilde{P}}(\pi)$ et de notre définition de $r_{\tilde{P}'|\tilde{P}}(\pi)$.
\end{proof}

\begin{remark}
  Il sera plus raisonnable d'établir une formule à la Gindikin-Karpelevi\u{c} pour les séries principales non ramifiées spécifiques, puis en déduire une formule explicite de $r_{\tilde{P}'|\tilde{P}}(\pi)$. On en obtiendra (\textbf{R7}). En fait, de nombreux cas ont été établis par McNamara \cite{Na11}, y compris les revêtements des groupes déployés simplement connexes construits par Matsumoto. Nous ne poursuivons pas cette approche ici.
\end{remark}

\section{Intégrales orbitales et caractères}\label{sec:int-caractere}
Dans cette section, on étudiera des distributions invariantes sur un revêtement $\rev: \tilde{G} \to G(F)$, où $F$ est un corps local de caractéristique nulle. Lorsque $F$ est archimédien, les résultats que nous cherchons sont déjà établis par Bouaziz \cite{Bo94a, Bo94b}. Par conséquent, on se limitera au cas $F$ non archimédien.

Soit $M$ un $F$-groupe réductif. Une distribution sur $\mathfrak{m}(F)$ signifie une fonctionnelle linéaire sur $C_c^\infty(\mathfrak{m}(F))$. Pour $f \in C_c^\infty(\mathfrak{m}(F))$ et $x \in M(F)$, on écrira $f^x: X \mapsto f(xXx^{-1})$. Le groupe $M(F)$ opère sur les distributions par $\angles{{}^x \Theta, f} = \angles{\Theta, f^x}$.

Le support d'une distribution $\theta$ a encore un sens et sera noté $\Supp\,\theta$. Les mêmes notions s'appliquent aux distributions sur $\tilde{G}$. On définit les distributions spécifiques ou anti-spécifiques comme dans \cite{Li10a}.

\subsection{Théorie sur l'algèbre de Lie}\label{sec:theorie-alg}
On se donne
\begin{itemize}
  \item $F$ un corps local non archimédien,
  \item $G$ un $F$-groupe réductif, sa composante neutre est noté $G^\circ$.
  \item $\bomega: G(F) \to \C^\times$ un caractère unitaire continu, trivial sur $Z_{G^\circ}(F)$.
\end{itemize}

Notre hypothèse sur $\bomega$ est justifiée car les résultats de cette sous-section deviendront triviaux si $\bomega$ n'est pas trivial sur $Z_{G^\circ}(F)$. La composante neutre de $Z_{G^\circ}$, notée $Z_{G^\circ}^\circ$, est un $F$-tore.

Munissons $G(F)$ d'une mesure de Haar. On utilise les notions définies dans \cite[\S 5]{Li10a}. On note $\pi: G_\text{SC} \to G$ le revêtement simplement connexe de $G^\circ_\text{der}$ et $\iota: Z_{G^\circ}^\circ \to G$ l'inclusion. Alors on a un module croisé\index[iFT2]{$G' \to G$}
$$ G' := G_\text{SC} \times Z_{G^\circ}^\circ \stackrel{(\pi,\iota)}{\longrightarrow} G. $$

L'action de $G$ sur $G_\text{SC}$ est la conjugaison et son action sur $Z_{G^\circ}^\circ$ est triviale. On note\index[iFT2]{$\Xi$}
$$\Xi := \Coker(G'(F) \to G(F)).$$
C'est un groupe fini.

Un fait important à se rappeler est que $\bomega \circ (\pi, \iota) = 1$. Un élément $X \in \mathfrak{g}(F)$ est dit $\bomega$-bon sous $G(F)$\index[iFT2]{$\bomega$-bon} si $\bomega$ est trivial sur $Z_G(X)(F)$. On définit ainsi les classes de conjugaison $\bomega$-bonnes sous $G(F)$; l'ensemble de telles classes est noté $\Gamma(\mathfrak{g}(F))^\bomega$. On introduit le $\C^\times$-torseur $\dot{\Gamma}(\mathfrak{g}(F))^\bomega \to \Gamma(\mathfrak{g}(F))^\bomega$; les éléments dans $\dot{\Gamma}(\mathfrak{g}(F))^\bomega$ sont les $\bomega$-bonnes classes de conjugaison par $G(F)$ munies d'une mesure complexe non triviale $\mu$ telle que $\angles{\mu, f^y} = \bomega(y) \angles{\mu,f}$ pour tout $f$. On définit $\Gamma_\text{reg}(\mathfrak{g}(F))^\bomega$, $\Gamma_\text{reg}(\mathfrak{g}(F))^\bomega$ (resp. $\Gamma_\text{nil}(\mathfrak{g}(F))^\bomega$, $\Gamma_\text{nil}(\mathfrak{g}(F))^\bomega$) en se limitant aux classes semi-simples régulières (resp. nilpotentes). Ces définitions généralisent celles de
\cite{Li10a} pour les groupes connexes.

Dans ce qui suit, on s'intéressera aux distributions $\theta$ sur $\mathfrak{g}(F)$ avec ${}^y \theta = \bomega(y)\theta$ pour tout $G(F)$. On définit des espaces vectoriels en dualité\index[iFT2]{$\mathcal{I}(\mathfrak{g}(F))_\bomega, \mathcal{I}(\mathfrak{g}'(F))_\bomega$}\index[iFT2]{$\mathcal{J}(\mathfrak{g}(F))^\bomega, \mathcal{J}(\mathfrak{g}'(F))^\bomega$}
\begin{align*}
  \mathcal{I}(\mathfrak{g}(F))_\bomega & := C_c^\infty(\mathfrak{g}(F))/ \left\langle f^y - \bomega(y) f : f \in C_c^\infty(\mathfrak{g}(F)), \; y \in G(F) \right\rangle, \\
  \mathcal{J}(\mathfrak{g}(F))^\bomega & := \{\theta : \text{distribution sur } \mathfrak{g}(F), \; \forall y \in G(F), {}^y \theta = \bomega(y)\theta \}.
\end{align*}
Si $\bomega=1$, on y omet le symbole $\bomega$; dans ce cas-là ce sont l'espace de Hecke invariant et l'espace de distributions invariantes, respectivement.

Les applications $f \mapsto f^{y^{-1}}$, $y \in G(F)$ induisent une action de $\Xi$ sur $\mathcal{I}(\mathfrak{g}'(F))$. Son action contragrédiente sur $\mathcal{J}(\mathfrak{g}'(F))$ se déduit de $\theta \mapsto {}^y \theta$. Posons $\Pi(\Xi)$ l'ensemble de représentations irréductibles de $\Xi$, on obtient donc les décompositions
\begin{align*}
  \mathcal{I}(\mathfrak{g}'(F)) & = \bigoplus_{\xi \in \Pi(\Xi)} \mathcal{I}(\mathfrak{g}'(F))_\xi , \\
  \mathcal{J}(\mathfrak{g}'(F)) & = \bigoplus_{\xi \in \Pi(\Xi)} \mathcal{J}(\mathfrak{g}'(F))^\xi ,
\end{align*}
où $(\cdots)^\xi$ désigne la composante $\xi$-isotypique et nous adoptons la convention $(\cdots)_\xi = (\cdots)^{(\xi^\vee)}$. Notons $\text{pr}_\xi$, $\text{pr}^\xi$ les projections $\xi$-isotypiques correspondantes. Vu notre hypothèse, $\bomega$ induit un caractère de $\Xi$, noté encore $\bomega$. Les espaces $\mathcal{I}(\mathfrak{g}'(F))_\bomega$ et $\mathcal{J}(\mathfrak{g}'(F))^\bomega$ sont en dualité.\index[iFT2]{$\text{pr}^\bomega, \text{pr}_\bomega$}

Notons que $\mathfrak{g} = \mathfrak{g}'$ comme $F$-algèbres de Lie. Le résultat suivant est évident.
\begin{lemma}\label{prop:dist-SC}
  On a
  \begin{align*}
    \mathcal{I}(\mathfrak{g}(F))_\bomega & = \mathcal{I}(\mathfrak{g}'(F))_\bomega , \\
    \mathcal{J}(\mathfrak{g}(F))^\bomega & = \mathcal{J}(\mathfrak{g}'(F))^\bomega .
  \end{align*}
\end{lemma}

Pour tout $\dot{X} \in \dot{\Gamma}(\mathfrak{g}(F))^\bomega$, on peut définir les intégrales orbitales normalisées $J^\bomega_G(\dot{X}, f)$, qui n'est que $\angles{\dot{X}, f}$ car $\dot{X}$ est regardée comme une mesure complexe. Soit $X \in \mathfrak{g}(F)$ qui est $\bomega$-bon, le symbole $\dot{X}$ désignera toujours un élément $\dot{X} \in \dot{\Gamma}(\mathfrak{g}(F))^\bomega$ tel que la classe de conjugaison sous-jacente de $\dot{X}$ contient $X$. Notons $G_X := Z_{G^\circ}(X)^\circ$. Si une mesure de Haar sur $G_X(F)$ est choisie, on peut choisir $\dot{X}$ de sorte que\index[iFT2]{$J_G^\bomega(\dot{X}, \cdot)$}
\begin{gather}\label{eqn:int-orb-omega}
  J_G^\bomega(\dot{X}, f) = |D^G(X)|^{\frac{1}{2}} \int_{G_X(F) \backslash G(F)} \bomega(x) f(x^{-1} X x) \dd x
\end{gather}
où $D^G(X) := \det(\ad(X)|\mathfrak{g}/\mathfrak{g}_X)$ est le discriminant de Weyl sur l'algèbre de Lie, qui ne dépend que de $G^\circ$. Cf. \cite[1.21]{BZ76}.

\begin{lemma}\label{prop:int-orb-SC}
  Soit $X \in \mathfrak{g}(F)$.
  \begin{enumerate}
    \item $X$ est $\bomega$-bon si et seulement si pour tout (ou ce qui revient au même, pour un) $\dot{X}' \in \dot{\Gamma}(\mathfrak{g}'(F))$ à support contenu dans la $G(F)$-orbite de $X$, on a $\mathrm{pr}^\bomega \dot{X}' \neq 0$.
    \item Supposons que $X$ est $\bomega$-bon. Soit $\dot{X} \in \dot{\Gamma}(\mathfrak{g}(F))^\bomega$ à support dans la $G(F)$-orbite contenant $X$, alors il existe un unique élément $\dot{X}' \in \dot{\Gamma}(\mathfrak{g}'(F))$ tel que
    \begin{itemize}
      \item $\mathrm{pr}^\bomega \dot{X}'$ correspond à $\dot{X}$ sous l'isomorphisme du Lemme \ref{prop:dist-SC};
      \item $X$ appartient à la $G'(F)$-orbite sous-jacente de $\dot{X}'$.
    \end{itemize}
    De plus, tout $\dot{X}' \in \dot{\Gamma}(\mathfrak{g}'(F))$ avec $\mathrm{pr}^\bomega \dot{X}' \neq 0$ est obtenu de cette manière.
  \end{enumerate}
\end{lemma}
\begin{remark}
  C'est plus commode de noter $\dot{X}'$ par $\dot{X}$. Alors ledit lemme signifie que
  $$ J_G^\bomega(\dot{X}, f) = J_{G'}(\dot{X}, \mathrm{pr}_\bomega f), \quad f \in \mathcal{I}(\mathfrak{g}(F)). $$
  Dorénavant, on adopte cette convention.
\end{remark}

Le résultat précédent permet de généraliser des résultats dans le cas $\bomega=1$ au cas général. Donnons des exemples importants.

\begin{theorem}[Germes de Shalika avec caractère]\label{prop:shalika-omega}
  Soit $T \subset G$ un $F$-tore maximal tel que $\bomega|_{T(F)}=1$. Posons $\mathfrak{t}_\mathrm{reg} := \mathfrak{t} \cap \mathfrak{g}_\mathrm{reg}$ et choisissons une mesure de Haar sur $T(F)$, qui permet définir par \eqref{eqn:int-orb-omega} les intégrales orbitales $J_G^\bomega(\dot{H}, \cdot)$ pour tout $H \in \mathfrak{t}_\mathrm{reg}(F)$.

  Alors il existe des fonctions $\Gamma_{\dot{u}} : \mathfrak{t}(F) \to \C$ où $\dot{u} \in \dot{\Gamma}_\mathrm{nil}(\mathfrak{g}(F))^\bomega$, telles que
  \begin{enumerate}
    \item $\Gamma_{z\dot{u}} = z^{-1} \Gamma_{\dot{u}}$ pour tout $z \in \C^\times$;
    \item $\Gamma_{\dot{u}}(t^2 H) = |t|^{-\dim G/G_u} \Gamma_{\dot{u}}(H)$ pour tout $t \in F^\times$;
    \item $\Gamma_{\dot{u}}(H+Z) = \Gamma_{\dot{u}}(H)$ si $Z \in \mathfrak{z}(F)$;
    \item $\Gamma_{\dot{u}}$ est localement constante sur $\mathfrak{t}_\mathrm{reg}(F)$ et localement bornée sur $\mathfrak{t}(F)$;
    \item pour tout $f \in C_c^\infty(\mathfrak{g}(F))$, il existe $\mathcal{U}$ un voisinage ouvert de $0$ dans $\mathfrak{t}(F)$ tel que pour tout $H \in \mathfrak{t}_\mathrm{reg}(F) \cap \mathcal{U}$ on a
    $$ J_G^\bomega(\dot{H}, f) = \sum_{u \in \Gamma_\mathrm{nil}(\mathfrak{g}(F))^\bomega} \Gamma_{\dot{u}}(H) J_G^\bomega(\dot{u}, f) $$
    où le produit dans la somme ne dépend pas du choix de $\dot{u}$.
  \end{enumerate}
\end{theorem}

\begin{proposition}\label{prop:int-orb-dense-omega}
  Soit $f \in \mathcal{I}(\mathfrak{g}(F))_\bomega$. Supposons que $J_G^\bomega(\dot{X}, f)=0$ pour tout $\dot{X} \in \dot{\Gamma}_\mathrm{reg}(\mathfrak{g}(F))^\bomega$, alors $f=0$. Autrement dit, les intégrales orbitales $J_G^\bomega(\dot{X}, \cdot)$ sont faiblement denses dans $\mathcal{J}(\mathfrak{g}(F))^\bomega$.
\end{proposition}
\begin{proof}
  Vu le Lemme \ref{prop:dist-SC}, on peut regarder $f$ comme un élément de $\mathcal{I}(\mathfrak{g}'(F))_\bomega$. D'après la densité des intégrales orbitales régulières \cite[Lemma 4.1]{HC99} appliquée à $G'$, il suffit de montrer que $J_{G'}(\dot{X}', f)=0$ pour tout $\dot{X}' \in \dot{\Gamma}_\text{reg}(\mathfrak{g}'(F))$ avec $\text{pr}^\bomega \dot{X}' \neq 0$. D'après le Lemme \ref{prop:int-orb-SC}, de tels $\dot{X}'$ correspondent aux éléments de $\dot{\Gamma}_\text{reg}(\mathfrak{g}(F))^\bomega$. On conclut en appliquant l'hypothèse.
\end{proof}

L'étape suivante est de généraliser \cite[Part II]{HC99}. On fixe un sous-groupe compact maximal spécial $K$ de $G(F)$ en bonne position relativement à un Lévi minimal $M_0$. On aura besoin de la transformée de Fourier sur $\mathfrak{g}(F)$. Pour ce faire, il convient de fixer
\begin{itemize}
  \item un caractère additif unitaire non trivial $\psi: F \to \C^\times$,
  \item une forme bilinéaire non dégénérée $B$ sur $\mathfrak{g}(F)$ qui est invariante par $G(F)$.
\end{itemize}

Montrons l'existence de $B$, qui n'est pas triviale car $G$ peut être non connexe. On pose $H := G(F)/G^\circ(F)$, $Z := Z_{G^\circ}^\circ$. Vu la décomposition $\mathfrak{g} = \mathfrak{z} \oplus \mathfrak{g}_\text{der}$, s'il existe une forme bilinéaire non dégénérée $B_Z$ sur $\mathfrak{z}(F)$ qui est invariante par $H$, alors on peut prendre $B = B_Z + B_\text{der}$ où $B_\text{der}$ est la forme de Killing de $\mathfrak{g}_\text{der}$. L'existence de $B_Z$ est garantie par le résultat suivant.

\begin{proposition}
  Soient $F$ un corps de caractéristique nulle, $H$ un groupe fini, $Z$ un $F$-tore muni d'une action $H \to \Aut_{F-\mathrm{tore}}(Z)$, alors il existe une forme bilinéaire non dégénérée $B_Z$ sur $\mathfrak{z}(F)$ qui est invariante par $H$.
\end{proposition}
\begin{proof}
  La variété des formes bilinéaires non dégénérées $H$-invariantes sur $\mathfrak{z}$ est un ouvert de Zariski d'un espace affine. Donc l'ensemble de ses $F$-points est dense pour la topologie de Zariski. Pour montrer que cette variété admet un $F$-point, il suffit de montrer qu'elle est non vide en tant qu'une variété algébrique, donc c'est loisible de remplacer $F$ par une extension quelconque. On se ramène ainsi au cas $Z$ déployé.

  Supposons $Z$ déployé. Puisque les tores déployés ainsi que les homomorphismes entres eux sont définis sur $\Z$, on se ramène au cas $F=\Q$. L'argument précédent nous ramène encore au cas $F=\R$. C'est bien connu qu'il existe une forme définie positive $H$-invariante sur $\mathfrak{z}(\R)$. Cela permet de conclure.
\end{proof}

La transformée de Fourier est définie par\index[iFT2]{transformée de Fourier}
$$ f \mapsto \hat{f}(\cdot) = \int_{\mathfrak{g}(F)} f(X) \psi(B(X, \cdot)) \dd X, \quad f \in C_c^\infty(\mathfrak{g}(F)), $$
où $\mathfrak{g}(F)$ est muni de la mesure autoduale par rapport à $\psi \circ B$. La transformée de Fourier d'une distribution $\theta$ est notée $\hat{\theta}$, définie par $\angles{\hat{\theta}, f} = \angles{\theta, \hat{f}}$ pour tout $f \in C_c^\infty(\mathfrak{g}(F))$.

On considérera les $\mathfrak{o}_F$-réseaux ``bien adaptés'' à $(M_0, K)$\index[iFT2]{réseau bien adapté}. Au lieu de donner la définition précise dans \cite[Definition 10.6]{HC99}, il suffit de donner une construction directe: on prend un sommet spécial $x$ correspondant à $K$ dans l'immeuble de Bruhat-Tits, alors les réseaux de Moy-Prasad $\mathfrak{g}(F)_{x,r}$ sont adaptés à $(M_0, K)$ pour tout $r>0$. Cf. \cite{AD02}.

Pour $\Omega$ un sous-ensemble ouvert compact de $\mathfrak{g}(F)$, on pose
\begin{align*}
  \mathcal{J}(\Omega) & := \left\{ \theta \in \mathcal{J}(\mathfrak{g}(F)) : \Supp\,\theta \subset \bigcup_{x \in G(F)}  x\Omega x^{-1} \right\}, \\
  \mathcal{J}(\Omega)^\bomega & := \mathcal{J}(\mathfrak{g}(F))^\bomega \cap \mathcal{J}(\Omega).
\end{align*}
On pose aussi\index[iFT2]{$\mathcal{J}_0^\bomega$}
\begin{align*}
  \mathcal{J}_0 & := \bigcup_{\Omega} \mathcal{J}(\Omega), \\
  \mathcal{J}_0^\bomega & := \bigcup_{\Omega} \mathcal{J}(\Omega)^\bomega = \mathcal{J}(\mathfrak{g}(F))^\bomega \cap \mathcal{J}_0.
\end{align*}

Définissons une norme $|\cdot|$ sur $\mathfrak{g}(F)$ comme dans \cite[\S 2]{HC99}. Soient $t>0$, $L$ un réseau bien adapté à $(M_0,K)$, et $V$ un voisinage de $\mathfrak{g}_\text{nil}(F)^{|\cdot|=1}$ dans $\mathfrak{g}(F)^{|\cdot|=1}$. Définissons
\begin{align*}
  \mathcal{J}(V,t,L)^\bomega & := \{ \theta \in \mathcal{J}(\mathfrak{g}(F))^\bomega : \forall X \in \mathfrak{g}(F), \; |X|>t \text{ et } \angles{\theta, \mathbbm{1}_{X+L}} \neq 0 \Rightarrow X \in F \cdot V \}
\end{align*}
où $\mathbbm{1}_{X+L}$ désigne la fonction caractéristique de $X+L$. On vérifie que $\mathcal{J}(V,t,L)^\bomega \subset \mathcal{J}(V,t',L)^\bomega$ si $t'>t$. On pose\index[iFT2]{$\mathcal{J}(V,\infty,L)^\bomega$}
\begin{align*}
  \mathcal{J}(V, \infty, L)^\bomega & := \bigcup_{t>0} \mathcal{J}(V,t,L)^\bomega .
\end{align*}

Soient $L' \subset \mathfrak{g}(F)$ un $\mathfrak{o}_F$-réseau quelconque et $\theta \in \mathcal{J}(\Omega)^\bomega$, on note $j_{L'} \theta$ la restriction de $\theta$ à l'espace $C_c(\mathfrak{g}(F)/L')$. Le résultat suivant interviendra dans la démonstration du Théorème \ref{prop:dist-admissible}.

\begin{proposition}\label{prop:Howe-res}
  Soient $L$ un $\mathfrak{o}_F$-réseau bien adapté à $(M_0,K)$ et $V$ un voisinage de $\mathfrak{g}_\text{nil}(F)^{|\cdot|=1}$ dans $\mathfrak{g}(F)^{|\cdot|=1}$. Quitte à rétrécir $V$, il existe un $\mathfrak{o}_F$-réseau $L' \supset L$ tel que
  $$ j_{L'} \mathcal{J}(V, \infty, L)^\bomega \subset j_{L'} \mathcal{J}_0^\bomega . $$
\end{proposition}
\begin{proof}
  On choisit $\dot{\Xi} \subset G(F)$ un ensemble de représentants de $\Xi$ tel que $1 \in \dot{\Xi}$. On pose
  $$ L' := \sum_{\xi \in \dot{\Xi}} \xi \cdot L \supset L .$$

  Soit $\theta \in \mathcal{J}(V,t,L)^\bomega$. D'après \cite[Corollary 11.4]{HC99}, il existe $\theta_0 \in \mathcal{J}_0(\Omega)$ tel que
  $$\forall \varphi \in C_c(\mathfrak{g}(F)/L), \quad \angles{\theta, \varphi} = \angles{\theta_0, \varphi}. $$

  On a $C_c(\mathfrak{g}(F)/L') \subset C_c(\mathfrak{g}(F)/L)$. De plus, si $\varphi \in C_c(\mathfrak{g}(F)/L')$ alors $\varphi^\xi \in C_c(\mathfrak{g}(F)/\xi^{-1} L') \subset C_c(\mathfrak{g}(F)/L)$ pour tout $\xi \in \dot{\Xi}$. Pour tout $\varphi \in C_c(\mathfrak{g}(F)/L')$, on a donc
  \begin{align*}
    \angles{\theta, \varphi} &= \angles{\text{pr}^\bomega \theta, \varphi} = \angles{\theta, \text{pr}_\bomega \varphi} \\
    & = |\Xi|^{-1} \sum_{\xi \in \dot{\Xi}} \bomega(\xi)^{-1} \angles{\theta, \varphi^\xi} \\
    & = |\Xi|^{-1} \sum_{\xi \in \dot{\Xi}} \bomega(\xi)^{-1} \angles{\theta_0, \varphi^\xi} \\
    & = \angles{\theta_0, \text{pr}_\bomega \varphi} = \angles{\text{pr}^\bomega \theta_0, \varphi}.
  \end{align*}

  Autrement dit $j_{L'} \theta = j_{L'} (\text{pr}^\bomega \theta_0)$. Puisque $\Xi$ est fini, on vérifie que $\text{pr}^\bomega \theta_0 \in \mathcal{J}_0^\bomega$. Cela achève la preuve.
\end{proof}

Le résultat suivant ne sera pas utilisé dans cet article, toutefois on en donne l'énoncé à cause de son importance.

\begin{theorem}\label{prop:Howe-fini}
  Soient $\Omega \subset \mathfrak{g}(F)$ compact, $L \subset \mathfrak{g}(F)$ un réseau bien adapté à $(K,M_0)$. Alors $j_L \mathcal{J}(\Omega)^\bomega$ est de dimension finie.
\end{theorem}

Soit $D \subset \mathfrak{g}(F)$. On dit que $D$ est un $G$-domaine si $D$ est ouvert, fermé et $G(F)$-invariant\index[iFT2]{$G$-domaine}. Pour $X \in \mathfrak{g}(F)$ semi-simple, on note $\mathcal{O}(X)$ l'ensemble des $G(F)$-orbites dans $\mathfrak{g}(F)$ dont l'adhérence contient $X$. C'est un ensemble fini.

\begin{theorem}\label{prop:regularite-dist}
  Soient $\Omega$ un ouvert compact dans $\mathfrak{g}(F)$ et $\theta \in \mathcal{J}(\Omega)^\bomega$, alors $\hat{\theta}$ est représentée par une fonction localement intégrable $g$ sur $\mathfrak{g}(F)$ telle que
    \begin{itemize}
      \item $g$ est localement constante sur $\mathfrak{g}_\text{reg}(F)$;
      \item $|D^G|^{\frac{1}{2}}g$ est localement bornée sur $\mathfrak{g}(F)$.
    \end{itemize}
\end{theorem}

\begin{theorem}\label{prop:Howe-dist}
  Soient $\Omega$ un ouvert compact dans $\mathfrak{g}(F)$. Alors il existe un $G$-domaine $D$ contenant $0$ tel que pour tout $\theta \in \mathcal{J}(\Omega)^\bomega$, il existe des coefficients $c_{\dot{u}}(\theta)$, où $\dot{u} \in \dot{\Gamma}_\mathrm{nil}(\mathfrak{g}(F))^\bomega$, tels que
  \begin{itemize}
    \item $c_{z\dot{u}} = z^{-1} c_{\dot{u}}$ pour tout $z \in \C^\times$;
    \item posons $\mu_{\dot{u}} := J_G^\bomega(\dot{u},\cdot)$, alors
    $$ \hat{\theta}|_D = \sum_{u \in \Gamma_\mathrm{nil}(\mathfrak{g}(F))^\bomega} c_{\dot{u}}(\theta) \widehat{\mu_{\dot{u}}}|_D. $$
  \end{itemize}

  Pour tout voisinage $V$ de $0$ dans $\mathfrak{g}(F)$. Les distributions $\widehat{\mu_{\dot{u}}}$ sont linéairement indépendants sur $V \cap \mathfrak{g}_\text{reg}(F)$.
\end{theorem}

\begin{proof}[Démonstration des Théorèmes \ref{prop:Howe-fini}, \ref{prop:regularite-dist} et \ref{prop:Howe-dist}]
  Vu les Lemmes \ref{prop:dist-SC} et \ref{prop:int-orb-SC}, on se ramène aux assertions concernant $\mathcal{J}(\mathfrak{g}'(F))^\bomega$, qui découlent immédiatement du cas établi par Harish-Chandra \cite{HC99}. Remarquons aussi que le passage de $G(F)$ à $G'(F)$ n'affecte pas les notions de la transformée de Fourier, des réseaux bien-adapté et des $G$-domaines. 
\end{proof}

\subsection{Théorie sur le groupe: descente semi-simple}
Revenons à la théorie sur le groupe. Considérons un revêtement local
$$ 1 \to \bmu_m \to \tilde{G} \stackrel{\rev}{\to} G(F) \to 1 $$
où $G$ est un $F$-groupe réductif connexe. On normalise les mesures comme dans \S\ref{sec:Plancherel}. On fixe
\begin{itemize}
  \item $\sigma \in G(F)$ semi-simple, $G^\sigma := Z_G(\sigma)$, $G_\sigma := Z_G(\sigma)^0$;
  \item $\tilde{\sigma} \in \rev^{-1}(\sigma)$;
  \item $\widetilde{G_\sigma} := \rev^{-1}(G_\sigma(F))$.
\end{itemize}
On obtient ainsi le caractère continu $[\cdot,\sigma]: G^\sigma(F) \to \bmu_m$ défini par
$$ [y,\sigma] = \tilde{y}^{-1} \tilde{\sigma}^{-1} \tilde{y} \tilde{\sigma}, \quad y \in G^\sigma(F), $$
avec $\tilde{y} \in \rev^{-1}(y)$ quelconque. On a
$$\tilde{\sigma}\tilde{y}=[y,\sigma]^{-1} \tilde{y}\tilde{\sigma}.$$

Posons
\begin{align*}
  G^\sigma(F)^\diamond & := \Ker [\cdot,\sigma] , \\
  G_\sigma(F)^\diamond & := \Ker [\cdot,\sigma]|_{G_\sigma(F)}, \\
  \widetilde{G_\sigma}^\diamond & := \rev^{-1}(G_\sigma(F)^\diamond).
\end{align*}

On choisit les objets suivants
\begin{itemize}
  \item $\mathcal{V}^\flat \subset \mathfrak{g}_\sigma(F)$ de la forme $\mathfrak{g}_\sigma(F)_r := \bigcup_{x} \mathfrak{g}_\sigma(F)_{x,r}$ avec $r \gg 0$ \cite[\S 3.1]{AD02}, de tels ouverts ne dépendent que de $r$ et forment une base locale en $0$ pour la topologie définie par les ouverts $G^\sigma(F)$-invariants;
  \item $\mathcal{W}^\flat := \exp(\mathcal{V}^\flat) \subset \widetilde{G_\sigma}^\diamond$, on peut prendre $r \gg 0$ de telle sorte que l'exponentielle définit un homéomorphisme de $\mathcal{V}^\flat$ sur $\mathcal{W}^\flat$, et que $\mathcal{W}^\flat$ est ouvert et invariant par $G^\sigma(F)$.
\end{itemize}
On peut aussi supposer que
$$ \widetilde{\mathcal{W}}^\flat := \bigcup_{\noyau \in \bmu_m} \noyau \mathcal{W}^\flat $$
est une réunion disjointe.

On définit ainsi
$$ G(F) \times^\sigma \widetilde{\mathcal{W}}^\flat := \frac{G(F) \times \widetilde{\mathcal{W}}^\flat}{(gh, \tilde{t}) \sim (g, [h,\sigma] h\tilde{t}h^{-1}), \quad \forall h \in G^\sigma(F)}. $$
Notons $\Phi:G(F) \times \tilde{G} \to \tilde{G}$ l'application $(g,\tilde{t}) \mapsto g\tilde{\sigma}\tilde{t}g^{-1}$. Faisons opérer  $G(F)$ sur $G(F) \times \tilde{G}$ (translation à gauche sur la composante $G(F)$) et sur $\tilde{G}$ (l'action adjointe), alors $\Phi$ est $G(F)$-équivariant. Le fait suivant est dû, pour l'essentiel, à Harish-Chandra.

\begin{proposition}\label{prop:submersion}
  L'application $\Phi$ est submersive. Quitte à rétrécir $\mathcal{V}^\flat$, elle induit un homéomorphisme $G(F)$-équivariant
  \begin{gather*}
    \bar{\Phi}: G(F) \times^\sigma \widetilde{\mathcal{W}}^\flat \rightiso \widetilde{\mathcal{W}}
  \end{gather*}
  où  $\widetilde{\mathcal{W}} := \Phi(G(F) \times \widetilde{\mathcal{W}}^\flat)$. Il existe une application ``intégrer le long des fibres''
  $$ \Phi_*: C_c^\infty(G(F) \times \widetilde{\mathcal{W}}^\flat) \to C_c^\infty(\widetilde{\mathcal{W}})$$
  caractérisée par l'égalité
  $$ \int_{G(F) \times \widetilde{\mathcal{W}}^\flat} \phi(g,\tilde{t}) F(\Phi(g,\tilde{t})) \dd g \dd \tilde{t} = \int_{\widetilde{\mathcal{W}}} (\Phi_*\phi)(\tilde{x}) F(\tilde{x}) \dd\tilde{x}, \qquad \forall F \in C^\infty(\widetilde{\mathcal{W}}). $$
  qui est une application surjective $G(F)$-équivariante.
\end{proposition}

On notera $\Phi^*$ le dual de $\Phi_*$ au niveau des distributions. Observons que tous les espaces en vue sont munis d'actions évidentes de $\bmu_m$, et toutes ces applications sont $\bmu_m$-équivariantes. On note $\mathbbm{1}$ la distribution $\phi \mapsto \int_{G(F)} \phi(g)\dd g$ sur $G(F)$. 

\begin{proposition}\label{prop:descente}
  Soit $\Theta$ une distribution invariante spécifique sur $\widetilde{\mathcal{W}}$. Alors il existe une unique distribution spécifique $\Theta^\flat$ sur $\widetilde{\mathcal{W}}^\flat$, caractérisée par 
  $$ \angles{\Theta, \Phi_* f} = \angles{\mathbbm{1} \otimes \Theta^\flat, f}, \qquad f \in C^\infty_{c,\asp}(G(F) \times \widetilde{\mathcal{W}}^\flat). $$
  Soit $y \in G^\sigma(F)$, alors
  $$ {}^y \Theta^\flat = [y, \sigma] \Theta^\flat, \qquad y \in G^\sigma(F). $$
\end{proposition}
\begin{proof}
 L'image par $\Phi^*$ de $\Theta$ est une distribution spécifique $G(F)$-invariante sur $G(F) \times \widetilde{\mathcal{W}}^\flat$, donc est de la forme $\mathbbm{1} \otimes \Theta^\flat$. La formule pour $\angles{\Theta, \Phi_* f}$ en découle. La deuxième assertion découle de la première, du fait que $\Theta$ est $G(F)$-invariant et de la Proposition \ref{prop:submersion}.
\end{proof}
\begin{corollary}
  Notons désormais $\theta := \exp^* (\Theta^\flat|_{\mathcal{W}^\flat})$, alors $\theta$ est une distribution sur $\mathcal{V}^\flat$ telle que ${}^y \theta = [y,\sigma]\theta$ pour tout $y \in G^\sigma(F)$. De plus, $\theta$ détermine $\Theta^\flat$.
\end{corollary}

\begin{remark}\label{rem:descente-analytique}
  Proprement dit, ce principe de descente semi-simple n'est pas celui de Harish-Chandra, car on n'utilise que des ouverts stables par $\bmu_m$ et le caractère $[\cdot,\sigma]$ intervient. Pour obtenir une version purement ``analytique'', il suffit de considérer $\theta$ comme une distribution $G^\sigma(F)^\diamond$-invariante sur $\mathcal{V}^\flat$.
\end{remark}

Appliquons la théorie de \S\ref{sec:theorie-alg} au groupe $G^\sigma$ et le caractère $[\cdot,\sigma]$. Utilisons les conventions de \cite[\S 6.3]{Li10a} pour les intégrales orbitales sur $\tilde{G}$. Le résultat suivant est immédiat selon la définition de $\Phi$ et la Proposition \ref{prop:submersion}.

\begin{lemma}\label{prop:descente-int-orb}
  Un élément $\tilde{\gamma} = \tilde{\sigma} \exp X$ avec $X \in \mathcal{V}^\flat$ est bon si et seulement si $X$ est $[\cdot,\sigma]$-bon sous $G^\sigma(F)$. Soient $\dot{\tilde{\gamma}} \in \dot{\Gamma}(\tilde{G})$ et $\Theta := J_{\tilde{G}}(\dot{\tilde{\gamma}}, \cdot)$, alors il existe un unique $\dot{X} \in \dot{\Gamma}(\mathfrak{g}_\sigma(F))^{[\cdot,\sigma]}$ tel que
  $$ \theta = J_{G^\sigma}^{[\cdot,\sigma]}(\dot{X}, \cdot). $$

  Inversement, toute distribution $\theta$ de cette forme se remonte en une intégrale orbitale anti-spécifique sur $\tilde{G}$.
\end{lemma}

À l'instar du cas de l'algèbre de Lie, définissons
$$ \mathcal{I}(\tilde{G}) := C_c^\infty(\tilde{G})/ \left\langle f^y - f : f \in C_c^\infty(\tilde{G}), \; y \in G(F) \right\rangle $$
et posons $\mathcal{J}(\tilde{G})$ son dual. Notons $\mathcal{I}_{\asp}(\tilde{G})$ la partie anti-spécifique de $\mathcal{I}(\tilde{G})$, son dual est noté par $\mathcal{J}_-(\tilde{G})$, qui n'est que l'espace des distributions invariantes spécifiques. Le résultat suivant est parallèle à la Proposition \ref{prop:int-orb-dense-omega}. Dualement à l'application $\Theta \mapsto \theta$, on a une application $\mathcal{I}(\mathcal{V}^\flat)^{[\cdot,\sigma]} \to \mathcal{I}_{\asp}(\widetilde{\mathcal{W}})$, notée $f^\flat \mapsto f$, satisfaisant à
$$ \angles{\Theta, f} = \angles{\theta, f^\flat}. $$
De plus, $f^\flat \mapsto f$ est surjective d'après la Proposition \ref{prop:submersion}.

\begin{proposition}\label{prop:int-orb-dense}
  Soit $f \in \mathcal{I}_{\asp}(\tilde{G})$ tel que pour tout $\dot{\tilde{\gamma}} \in \dot{\Gamma}(\tilde{G})$ avec $\gamma$ fortement régulier, on a $J_{\tilde{G}}(\dot{\tilde{\gamma}}, f)=0$. Alors $f=0$.
\end{proposition}
\begin{proof}
  On peut prendre $f^\flat \in \mathcal{I}(\mathcal{V}^\flat)^{[\cdot,\sigma]}$ tel que $f^\flat \mapsto f$. Soit $\dot{X} \in \dot{\Gamma}_\text{reg}(\mathfrak{g}_\sigma(F))^{[\cdot,\sigma]}$. Si $X \notin \mathcal{V}^\flat$ alors $J_{G^\sigma}^{[\cdot,\sigma]}(\dot{X}, f^\flat)=0$. Supposons donc $X \in \mathcal{V}^\flat$, alors le Lemme \ref{prop:descente-int-orb} affirme que $J_{G^\sigma}^{[\cdot,\sigma]}(\dot{X}, f^\flat)$ remonte en l'intégrale orbitale de $f$ le long de la classe de conjugaison de  $\tilde{\sigma}\exp X$. Puisque $\mathcal{V}^\flat$ est supposé suffisamment petit, la classe de $\tilde{\sigma}\exp X$ est fortement régulière. D'où $J_{G^\sigma}^{[\cdot,\sigma]}(\dot{X}, f^\flat)=0$ pour tout $\dot{X} \in \dot{\Gamma}_\text{reg}(\mathfrak{g}_\sigma(F))^{[\cdot,\sigma]}$.

  Maintenant c'est une conséquence de la Proposition \ref{prop:int-orb-dense-omega} que $f^\flat = 0$, d'où $f=0$.
\end{proof}

\subsection{Distributions admissibles invariantes spécifiques}
Soit $H$ est groupe compact, on note $\Pi(H)$ l'ensemble de classes d'équivalences de représentations irréductibles de $H$. Si $\mathbf{d} \in \Pi(H)$, on note $\xi_\mathbf{d} \in C^\infty(H)$ son caractère. La représentation triviale de $H$ est notée $\mathbbm{1}_H$.

Nous reprendrons les arguments de \cite[Part III]{HC99}.

\begin{definition}\label{def:dist-admissible}\index[iFT2]{distribution admissible}
  Soit $\Theta$ une distribution sur un ouvert $\mathcal{W}$ de $\tilde{G}$. Soit $K_0$ un sous-groupe ouvert compact de $\tilde{G}$. On dit que $\Theta$ est $(\tilde{G}, K_0)$-admissible en $\tilde{\gamma} \in \tilde{G}$ si
  \begin{itemize}
    \item $\tilde{\gamma}K_0 \subset \mathcal{W}$;
    \item pour tout sous-groupe ouvert $K_1 \subset K_0$ et $\mathbf{d} \in \hat{K}_1$, on a
    $$ \Theta * \xi_\mathbf{d} = 0 $$
    où $*$ désigne la convolution, sauf s'il existe $x \in G(F)$ tel que les restrictions à $xK_0 x^{-1} \cap K_1$ de $\mathbbm{1}_{xK_0 x^{-1}}$ et de $\mathbf{d}$ s'entrelacent.
  \end{itemize}
  On dit que $\Theta$ est admissible en $\tilde{\gamma}$ s'il existe $K_0$ tel que $\Theta$ est $(\tilde{G},K_0)$-admissible en $\tilde{\gamma}$. Si $\Theta$ est admissible partout, on dit qu'elle est admissible.
\end{definition}
Cette définition s'applique à tout groupe localement profini. Notre but est le résultat suivant.

\begin{theorem}\label{prop:dist-admissible}
  Soit $\Theta$ une distribution invariante spécifique de $\tilde{G}$. Soient $\sigma \in G(F)$, $\tilde{\sigma} \in \rev^{-1}(\sigma)$. Si $\Theta$ est admissible en $\tilde{\sigma}$, alors $\Theta$ est représentée par une fonction localement intégrable au voisinage de $\tilde{\sigma}$, qui est localement constante sur $\tilde{G}_\mathrm{reg}$. La fonction $|D^G|^{\frac{1}{2}} \Theta$ est localement bornée.

  De plus, il existe des coefficients uniques $c_{\dot{u}}$, où $\dot{u} \in \dot{\Gamma}_{\mathfrak{nil}}(\mathfrak{g}_\sigma(F))^{[\cdot,\sigma]}$, tels que
  \begin{itemize}
    \item $c_{z\dot{u}} = z^{-1} c_{\dot{u}}$ pour tout $z \in \C^\times$,
    \item pour $X \in \mathfrak{g}_\sigma(F)$ suffisamment voisin de $0$, avec des conventions de \S\ref{sec:theorie-alg}, on a
      $$\Theta(\tilde{\sigma}\exp X) = \sum_{u \in \Gamma_{\mathfrak{nil}}(\mathfrak{g}_\sigma(F))^{[\cdot,\sigma]}} c_{\dot{u}} \widehat{\mu_{\dot{u}}}(X). $$
  \end{itemize}
\end{theorem}

Enregistrons d'abord une conséquence qui justifiera toute opération de caractères dans les sections suivantes.
\begin{corollary}\label{prop:caractere-admissible}
  Soit $(\pi,V)$ une représentation admissible irréductible de $\tilde{G}$ et notons $\Theta_\pi$ son caractère. Alors $\Theta_\pi$ vérifie les assertions du Théorème \ref{prop:dist-admissible} en tout $\tilde{\sigma} \in \tilde{G}$ semi-simple.
\end{corollary}
\begin{proof}
  Prenons $K_0 \subset \tilde{G}$ un sous-groupe ouvert compact tel que $V^{K_0} \neq \{0\}$. Le corollaire résulte du fait que $\Theta_\pi$ est $(\tilde{G}, K_0)$-admissible partout; voir \cite{C87}.
\end{proof}

Fixons $\Theta$, $\sigma$ et $\tilde{\sigma}$ comme dans l'énoncé. Indiquons très grossièrement l'approche de \cite{HC99}.

\paragraph{Étape 1}
Ces assertions étant locales et $G(F)$-invariantes, on effectue la descente semi-simple avec un voisinage $G^\sigma(F)$-invariant $\mathcal{V}^\flat \subset \mathfrak{g}_\sigma(F)$ et  $\widetilde{W}^\flat = \bmu_m \exp(\mathcal{V}^\flat)$, tous suffisamment petits, comme dans la Proposition \ref{prop:descente}. On obtient ainsi la distribution descendue $\theta \in \mathcal{J}(\mathcal{V}^\flat)^{[\cdot,\sigma]}$. D'après la remarque \ref{rem:descente-analytique}, cette étape est de nature analytique si l'on se limite à l'action de $G^\sigma(F)^\diamond$, donc est identique à celle de \cite[\S 18]{HC99}.

\paragraph{Étape 2}
Soient $K_0$, $K_1$ et $\mathbf{d}$ comme dans la Définition \ref{def:dist-admissible}. Afin d'exploiter la condition sur l'entrelacement de $\mathbbm{1}_{xK_0 x^{-1}}$ et $\mathbf{d}$, on utilise la théorie d'orbites co-adjointes de Howe. Précisons.

On prend un $s$-réseau $L \subset \mathfrak{g}(F)$ (cf. \cite[\S 17]{HC99}), i.e. un réseau suffisamment petit, de sorte que $K := \exp L$ est ouvert et compact de $\tilde{G}$, et il est muni de la structure de groupe à l'aide de la formule de Baker-Campbell-Hausdorff; on suppose de plus que $\frac{1}{2} L$ les vérifie aussi, et on pose $K^{1/2} := \exp(\frac{1}{2}L)$, c'est distingué dans $K$.

Notons $\Pi(K)$ l'ensemble de classes d'équivalence de représentations irréductibles de $K$ et $\Pi^{1/2}(K)$ l'ensemble de $K^{1/2}$-orbites dans $\Pi(K)$. Notons $L^*$ le dual de $L$ par rapport à $\psi \circ B$.

La théorie de Howe fournit une bijection
\begin{align*}
  \Pi^{1/2}(K) & \longrightarrow \{K^{1/2}-\text{orbites dans } \mathfrak{g}(F)/L^* \} \\
  \mathbf{d} & \longmapsto \mathcal{O}_\mathbf{d},
\end{align*}
qui est caractérisée par la formule de caractère à la Kirillov
\begin{gather}\label{eqn:caractere-Kirillov}
  d(\mathbf{d}) \xi_\mathbf{d}(\exp \lambda) = \sum_{X \in \mathcal{O}_\mathbf{d}} \psi(B(X,\lambda)),
\end{gather}
où $d(\mathbf{d})$ est le degré formel et on a $d(\mathbf{d}) = [K : K_X]^{\frac{1}{2}}$, ici $X \in \mathfrak{g}(F)/L^*$ désigne un élément quelconque dans $\mathcal{O}_\mathbf{d}$ et $K_X$ désigne son stabilisateur sous l'action de $K$.

On en déduit que, soient $(L_i, K_i, \mathbf{d}_i)$ comme ci-dessus et $\mathcal{O}_i$ l'orbite co-adjointe associée ($i=1,2$), alors pour $x \in G(F)$, les $K_i^{1/2}$-orbites de représentations $\mathbf{d}_1$ et ${}^x \mathbf{d}_2$ restreintes à $K_1 \cap x K_2 x^{-1}$, où ${}^x \mathbf{d}_2$ est la représentation de $x K_2 x^{-1}$ transportée de $\mathbf{d}_2$ par $\Ad(x)$, s'entrelacent si et seulement si $\mathcal{O}_1 \cap {}^x \mathcal{O}_2 \neq \emptyset$.

Observons que ce formalisme ne fait pas intervenir le revêtement. Les arguments de \cite[\S\S 19-20]{HC99} s'y adaptent immédiatement.

\paragraph{Étape 3}
Appliquons le formalisme de \S\ref{sec:theorie-alg} au groupe $G^\sigma(F)$. En particulier, on a fixé une norme $|\cdot|$ sur $\mathfrak{g}_\sigma(F)$.

Prenons des $s$-réseaux $L \subset \mathfrak{g}(F)$ et $\Lambda \subset \mathfrak{g}_\sigma(F)$ tel que $\Lambda \subset L$. On demande de plus que $\Lambda$ est bien adapté (par rapport à un sous-groupe compact maximal spécial et un Lévi de $G_\sigma(F)$), alors $\Lambda^* := \{X \in \mathfrak{g}_\sigma(F): \forall v \in \Lambda, \; \psi \circ B(X, v) = 1 \}$ l'est aussi. Prenons un voisinage $V$ de $\mathfrak{g}_{\sigma,\text{nil}}(F)^{|\cdot|=1}$ dans $\mathfrak{g}_\sigma(F)^{|\cdot|=1}$, suffisamment petit de sorte que la Proposition \ref{prop:Howe-res} s'applique.

On en déduit des sous-groupes compacts ouverts $K := \exp L$ et $K_\sigma := \exp\Lambda$. Quitte à rétrécir $\Lambda$, on peut supposer que
$$ [\cdot,\sigma]|_{K_\sigma}=1. $$

\begin{lemma}[cf. {\cite[Lemma 21.2]{HC99}}]
  Quitte à rétrécir $V$, il existe $\nu > 0$ tel que pour tout $X \in \mathfrak{g}_\sigma(F)$ avec $|X| > q^\nu$, on a
  $$ \angles{\hat{\theta}, \mathbbm{1}_{X+\Lambda^*}} \neq 0 \; \Rightarrow X \in F \cdot V . $$

  Autrement dit, $\hat{\theta} \in \mathcal{J}(V, q^\nu, \Lambda^*)^\bomega$.
\end{lemma}
\begin{proof}
  Il suffit de remarquer que l'on n'utilise que la partie ``analytique'' des arguments de descente dans \cite{HC99}. En particulier, on n'utilise que la conjugaison par des éléments dans $G^\sigma(F)^\diamond$. D'autre part, la théorie des orbites co-adjointes est encore utilisable sur le revêtement, comme expliqué à l'étape 2.
\end{proof}

\begin{proof}[Démonstration du Théorème \ref{prop:dist-admissible}]
  D'après la Proposition \ref{prop:Howe-res}, il existe
  \begin{itemize}
    \item un sous-ensemble ouvert compact $\Omega \in \mathfrak{g}_\sigma(F)$,
    \item $\hat{\theta}_1 \in \mathcal{J}(\Omega)^\bomega$,
    \item un $\mathfrak{o}_F$-réseau $\Lambda_1^* \supset \Lambda^*$,
  \end{itemize}
  tels que
  $$ j_{\Lambda_1^*} \hat{\theta}_1 = j_{\Lambda_1^*} \hat{\theta}.$$

  On note $\Lambda_1 \subset \Lambda$ le $\mathfrak{o}_F$-réseau tel que $\Lambda_1^* = (\Lambda_1)^*$. En inversant la transformée de Fourier, on en déduit que $\theta_1|_{\Lambda_1} = \theta|_{\Lambda_1}$.

  La distribution $\theta_1$ vérifie les assertions dans le Théorème \ref{prop:regularite-dist} sur un voisinage $G_\sigma(F)$-invariant de $0$ dans $\mathfrak{g}_\sigma(F)$, donc $\theta$ les vérifie aussi sur un voisinage de $0$ dans $\mathfrak{g}_\sigma(F)$. En particulier, quitte à rétrécir $\mathcal{V}^\flat$, $\theta$ est représentée par une fonction localement intégrable $F$ avec $F(y^{-1} X y) = [y,\sigma] F(X)$. Donc $\Theta|_{\widetilde{W}^\flat}$ est représentée par la fonction invariante localement intégrable $g\tilde{\sigma}\exp(X) g^{-1} \mapsto F(X)$. Le développement local en termes des distributions $\widehat{\mu_{\dot{u}}}$ résulte immédiatement du Théorème \ref{prop:Howe-dist}.
\end{proof}

\section{La formule des traces locale}\label{sec:formule-traces}
\subsection{Le noyau tronqué}
Soient $F$ un corps local de caractéristique nulle et $G$ un $F$-groupe réductif connexe. On considère un revêtement
$$ 1 \to \bmu_m \to \tilde{G} \stackrel{\rev}{\to} G(F) \to 1 . $$

Fixons un sous-groupe de Lévi minimal $M_0$ et un sous-groupe compact maximal spécial $K$, supposés en bonne position. Lorsque $F$ est archimédien, on normalise des mesures de Haar sur les groupes linéaires en question comme dans \cite{Ar91}, ce qui déterminent les mesures sur les revêtements selon la convention dans \cite{Li10a}. Lorsque $F$ est non archimédien, les mesures sont normalisées comme dans \S\ref{sec:Plancherel} sauf que l'on choisit les mesures sur les radicaux unipotentes de sorte que $\gamma(G|M)=1$ pour tout $M \in \mathcal{L}(M_0)$, ce qui est loisible selon la Remarque \ref{rem:Plancherel-mesures}.

La formule des traces locale concerne la représentation $R$ de $\tilde{G} \times \tilde{G}$ sur $L^2(\tilde{G})$ définie par
$$ (R(\tilde{y}_1, \tilde{y}_2)\phi)(\tilde{x}) = \phi(\tilde{y}_1^{-1} \tilde{x} \tilde{y}_2), \quad \phi \in L^2(\tilde{G}). $$

Soit $f \in C_c^\infty(\tilde{G})$, on peut former l'opérateur $R(f)$. Pour la formule des traces locale, on s'intéresse plutôt aux fonctions test dans les espaces\index[iFT2]{$\mathcal{H}(\tilde{G})$}
\begin{align*}
  \mathcal{C}(\tilde{G}): & \quad \text{fonctions de Schwartz-Harish-Chandra}, \\
  \mathcal{H}(\tilde{G}): & \quad \text{fonctions lisses à support compact et $\tilde{K}$-finies},
\end{align*}
ainsi que leurs variantes $\bmu_m$-équivariantes $\mathcal{C}_-(\tilde{G})$, $\mathcal{C}_{\asp}(\tilde{G})$, $\mathcal{H}_-(\tilde{G})$ et $\mathcal{H}_{\asp}(\tilde{G})$. D'ici jusqu'à la Proposition \ref{prop:dist-temperee}, on prend
$$ f(\tilde{x}, \tilde{y}) = f_1(\tilde{x}) f_2(\tilde{y}), \quad f_1 \in \mathcal{H}_-(\tilde{G}), \; f_2 \in \mathcal{H}_{\asp}(\tilde{G}). $$

On l'abrégera souvent par l'expression $f=f_1 f_2$.

On montre que $R(f)$ est de noyau $K(\tilde{x},\tilde{y}) = \int_{G(F)} f_1(\tilde{x}\tilde{w}) f_2(\tilde{w}\tilde{y}) \dd w$, où $\tilde{w} \in \rev^{-1}(w)$ est quelconque; ce choix ne change pas le produit en question selon notre hypothèse sur $f_1, f_2$. Conservons cette convention dans les intégrales dans cette section.

Notons $\Gamma_\text{ell}(G(F))$ l'espace des orbites semi-simples $F$-elliptiques dans $G(F)$. Il est muni d'une mesure de Radon de sorte que pour toute $\phi \in C_c(\Gamma_\text{ell}(G(F)) \cap \Gamma_\text{reg}(G(F)))$, on a
$$ \int_{\Gamma_\text{ell}(G(F))} \phi(\gamma) \dd\gamma = \sum_{T} |W(G(F),T(F))|^{-1} \int_{T(F)} \phi(t) \dd t $$
où $T$ parcourt les classes de conjugaison des $F$-tores maximaux elliptiques dans $G$, et
$$ W(G(F),T(F)) := N_G(T)(F)/T(F).$$

Ici $T(F)$ est muni de la mesure de Haar de sorte que $\mes(T(F)/A_G(F))=1$. Idem pour tout Lévi de $G$. Alors la formule d'intégrale de Weyl s'écrit
$$ \int_{G(F)} h(x) \dd x = \sum_{M \in \mathcal{L}(M_0)} |W^M_0| |W^G_0|^{-1} \int_{\Gamma_{\text{ell}}(M(F))} |D(\gamma)| \iota_M \int_{A_M(F)^\dagger \backslash G(F)} h(x^{-1} \gamma x) \dd x \dd \gamma $$
pour toute $h \in C_c(G(F))$, où $D(\gamma)$ signifie le discriminant de Weyl; on peut restreindre l'intégrale à $\Gamma_{\text{ell},G-\text{reg}}(M(F))$. Puisque $K(x,x) = K(\tilde{x},\tilde{x}) = \int_{G(F)} f_1(\tilde{w}) f_2(x^{-1}\tilde{w}x) \dd w$ ne dépend que de $x$, on peut appliquer cette formule et déduire que $K(x,x)$ est égal à
$$ \sum_{M \in \mathcal{L}(M_0)} |W^M_0| |W^G_0|^{-1} \int_{\Gamma_{\text{ell}}(M(F))} |D(\gamma)| \iota_M \int_{A_M(F)^\dagger \backslash G(F)} f_1(x_1^{-1} \tilde{\gamma} x_1) f_2(x^{-1} x_1^{-1} \tilde{\gamma} x_1 x) \dd x_1 \dd \gamma . $$

C'est le développement géométrique. Pour le côté spectral, on utilise la formule de Plancherel (Corollaire \ref{prop:Plancherel-f(1)}) et la Remarque \ref{rem:Plancherel-mesures} sur le choix des mesures. Soit $M \in \mathcal{L}(M_0)$, munissons $\Pi_{2,-}(\tilde{M})$ de la mesure de sorte que pour toute $h \in C_{c,\asp}^\infty(\tilde{G})$, l'égalité du Corollaire \ref{prop:Plancherel-f(1)} soit vérifiée. Comme les mesures sont choisies de sorte que $\gamma(G|M)=1$ pour tout $M \in \mathcal{L}(M_0)$, on a
$$ h(1) = \sum_{M \in \mathcal{L}(M_0)} |W^M_0| |W^G_0|^{-1} \int_{\Pi_{2,-}(\tilde{M})} d(\sigma) \mu(\sigma) \Tr(\mathcal{I}_{\tilde{P}}(\sigma,h)) \dd\sigma . $$

Fixons $\tilde{x} \in \tilde{G}$ et posons
$$ h(\tilde{y}) = \int_{G(F)} f_1(\tilde{x}\tilde{w}) f_2(\tilde{w} \tilde{y}\tilde{x}) \dd w . $$
On vérifie que $h(1) = K(x,x)$ et $h = \check{f}_1 * f_2^x$; en particulier, $h \in C_{c,\asp}^\infty(\tilde{G})$. Notons $V$ l'espace vectoriel sous-jacent de $\sigma$ et choisissons $\mathcal{B}_{\tilde{P}}(\sigma)$ une base orthonormée de l'espace hilbertien des opérateurs de Hilbert-Schmidt $\mathcal{I}_{\tilde{P}}(V) \to \mathcal{I}_{\tilde{P}}(V)$, formée d'éléments $\tilde{K}$-finis. On en déduit comme dans \cite[\S 2]{Ar91} que
\begin{align*}
  \Tr(\mathcal{I}_{\tilde{P}}(\sigma,h)) & = \sum_{S \in \mathcal{B}_{\tilde{P}}(\sigma)} \Tr(\mathcal{I}_{\tilde{P}}(\sigma,\tilde{x}) S(f)) \overline{\Tr(\mathcal{I}_{\tilde{P}}(\sigma,\tilde{x}) S)}, \\
  S(f) & := d(\sigma) \mathcal{I}_{\tilde{P}}(\sigma, f_2) S \mathcal{I}_{\tilde{P}}(\sigma, \check{f}_1).
\end{align*}

Alors la formule de Plancherel entraîne que $K(x,x)$ est égal à
$$ \sum_{M \in \mathcal{L}(M_0)} |W^M_0| |W^G_0|^{-1} \int_{\Pi_{2,-}(\tilde{M})} \mu(\sigma) \sum_{S \in \mathcal{B}_{\tilde{P}}(\sigma)} \Tr(\mathcal{I}_{\tilde{P}}(\sigma,\tilde{x}) S(f)) \overline{\Tr(\mathcal{I}_{\tilde{P}}(\sigma,\tilde{x}) S)} \dd\sigma .$$

On note $\mathfrak{a}_0 = \mathfrak{a}_{M_0}$ comme d'habitude. On munit $\mathfrak{a}_0$ (resp. $G(F)$) d'une norme (resp. une fonction hauteur) $\|\cdot\|$ comme dans \cite[\S 4]{Ar91}. Adoptons le formalisme de $(G,M)$-familles de \cite{Li10a}.

Soit $M \in \mathcal{L}(M_0)$. Si $\mathcal{Y} = \{Y_P : P \in \mathcal{P}(M)\}$ est un ensemble $(G,M)$-orthogonal, on note $S_M(\mathcal{Y})$ son enveloppe convexe dans $\mathfrak{a}^G_M$. Soit $T \in \mathfrak{a}_0$. Pour tout $P_0 \in \mathcal{P}(M_0)$, on note $T_{P_0}$ l'unique élément dans $\overline{\mathfrak{a}_{P_0}^+} \cap (W^G_0 \cdot T)$. Alors $\{T_{P_0} : P_0 \in \mathcal{P}(M_0)\}$ forme un ensemble $(G,M_0)$-orthogonal positif; on le note abusivement par $T$. Posons
$$ d(T) := \inf \{\angles{\alpha, T_{P_0}} : P_0 \in \mathcal{P}(M_0), \alpha \in \Delta_{P_0} \} \geq 0. $$

On dit que $T$ est suffisamment régulier si $d(T) \gg 0$. Dans ce qui suit, nous supposerons toujours que $T \in \mathfrak{a}_{M_0, F}$. Définissons $u(x,T)$ la fonction caractéristique de l'ensemble
$$ \{ x = k_1 m k_2 : k_1, k_2 \in K, H_{M_0}(m) \in S_{M_0}(T) \}. $$

On la regarde comme une fonction sur $A_G(F)^\dagger K \backslash G(F) / K$. Le noyau tronqué est défini par l'intégrale absolument convergente
$$ K^T(f) := \iota_G \int_{A_G(F)^\dagger \backslash G(F)} K(x,x) u(x,T) \dd x . $$

\subsection{Le côté géométrique}\label{sec:LTF-geo}
On a
$$ K^T(f) = \sum_{M \in \mathcal{L}(M_0)} |W^M_0| |W^G_0|^{-1} \int_{\Gamma_\text{ell}(M(F))} K^T(\gamma, f) \dd\gamma $$
où
\begin{align*}
  K^T(\gamma, f) & = |D(\gamma)| \iota_M^2 \int_{(A_M(F)^\dagger \backslash G(F))^2} f_1(x_1^{-1} \tilde{\gamma} x_1 ) f_2(x_2^{-1} \tilde{\gamma} x_2) u_M^\dagger(x_1,x_2,T) \dd x_1 \dd x_2 , \\
  u_M^\dagger(x_1, x_2, T) & := \iota_G \iota_M^{-1} \int_{A_G(F)^\dagger \backslash A_M(F)^\dagger} u(x_1^{-1} a x_2, T) \dd a .
\end{align*}
Cf. \cite[p.30]{Ar91}. Soient $M \in \mathcal{L}(M_0)$, $x_1,x_2 \in G(F)$. On définit l'ensemble $(G,M)$-orthogonal $\mathcal{Y}_M(x_1, x_2, T)$ associé à
$$ Y_P(x_1, x_2, T) := T_P + H_P(x_1) - H_{\bar{P}}(x_2), \quad P \in \mathcal{P}(M), $$
où $T_P$ est la projection sur $\mathfrak{a}_M$ de $T_{P_0}$, $P_0 \in \mathcal{P}(M_0)$, $P_0 \subset P$ quelconque. C'est un $(G,M)$-ensemble orthogonal positif pourvu que $d(T)$ soit suffisamment régulier par rapport à $(x_1,x_2)$, ce que nous supposons.

D'autre part, Arthur \cite[(3.8)]{Ar91} a défini la fonction combinatoire $\sigma_M(\cdot, \mathcal{Y}_M(x_1, x_2, T))$ sur $\mathfrak{a}^G_M$. Comme $\mathcal{Y}_M(x_1, x_2, T)$ est positif, $\sigma_M(\cdot, \mathcal{Y}_M(x_1, x_2, T))$ est la fonction caractéristique de $S_M(\mathcal{Y})$; en particulier, elle est à support compact. Donc on peut définir la fonction poids
$$ v_M^\dagger(x_1,x_2,T) := \iota_G \iota_M^{-1} \int_{A_G(F)^\dagger \backslash A_M(F)^\dagger} \sigma_M(H_M(a), \mathcal{Y}_M(x_1, x_2, T)) \dd a .$$

Posons ainsi
\begin{align*}
  J^T(\gamma, f) & := |D(\gamma)| \iota_M^2 \int_{(A_M(F)^\dagger \backslash G(F))^2} f_1(x_1^{-1} \tilde{\gamma} x_1 ) f_2(x_2^{-1} \tilde{\gamma} x_2) v_M^\dagger(x_1,x_2,T) \dd x_1 \dd x_2 , \\
  J^T(f) & := \sum_{M \in \mathcal{L}(M_0)} |W^M_0| |W^G_0|^{-1} \int_{\Gamma_\text{ell}(M(F))} J^T(\gamma,f) \dd\gamma.
\end{align*}

\begin{lemma}\label{prop:lemme-geo-crucial}
  Soit $\delta > 0$, il existe des constantes $C, \epsilon_1, \epsilon_2 > 0$ telles que
  $$ |u_M^\dagger(x_1,x_2,T) - v_M^\dagger(x_1,x_2,T)| \leq C e^{-\epsilon_1 \|T\|} $$
  pour tout $(T,x_1,x_2)$ tel que $d(T) \geq \delta \|T\|$ et $\|x_i\| \leq e^{\epsilon_2 \|T\|}$, $i=1,2$.
\end{lemma}
\begin{proof}
  Si $F$ est non archimédien, alors il existe des constantes $C_1, \epsilon_1, \epsilon_2$ telles que
  $$ u_M(x_1^{-1} a x_2, T) = \sigma_M(H_M(a), \mathcal{Y}_M(x_1,x_2,T)), \quad \forall a \in A_M(F) $$
  pour tout $(T,x_1,x_2)$ comme dans l'assertion, d'après \cite[p.38]{Ar91}. On conclut en intégrant cette égalité sur $A_G(F)^\dagger \backslash A_M(F)^\dagger$.

  Si $F$ est archimédien, l'argument de \cite[pp.39-42]{Ar91} marche sans modification. En fait, on se ramène à comparer des intégrales sur des $\R$-espaces vectoriels $\mathfrak{a}^Q_M$, $\mathfrak{a}^G_Q$, où $Q \in \mathcal{F}(M)$. De tels arguments ne font pas intervenir $A_G(F)^\dagger$ et $A_M(F)^\dagger$.
\end{proof}

\begin{corollary}
  Soit $\delta > 0$, alors il existe des constantes $C,\epsilon > 0$ telles que
  $$ |K^T(f) - J^T(f)| \leq C e^{-\epsilon \|T\|} $$
  pour tout $T$ tel que $d(T) \geq \delta\|T\|$.
\end{corollary}
\begin{proof}
  On itère \cite[\S 4]{Ar91}. En fait, la démonstration ne repose que sur le Lemme \ref{prop:lemme-geo-crucial} et des majorations qui n'ont rien à faire avec le revêtement.
\end{proof}

Pour l'instant, supposons que $F$ est non archimédien. Introduisons une version discrète de la construction dans \cite[\S 4.2]{Li10a}. Notons $q_F := |\mathfrak{o}_F/\mathfrak{p}_F|$ et posons
\begin{align*}
  \tilde{\mathcal{L}}_M^\dagger & := (\tilde{\mathfrak{a}}_{M,F}^\dagger + \mathfrak{a}_G)/\mathfrak{a}_G, \\
  \tilde{\mathcal{L}}_M & := (\tilde{\mathfrak{a}}_{M,F} + \mathfrak{a}_G)/\mathfrak{a}_G, \\
  \mathcal{L}_M & := (\mathfrak{a}_{M,F} + \mathfrak{a}_G)/\mathfrak{a}_G, \\
  \mathcal{L}_0 & := \mathcal{L}_{M_0}.
\end{align*}

\begin{lemma}\label{prop:facteur-a-1}
  On a
  $$ [\tilde{\mathcal{L}}_M : \tilde{\mathcal{L}}_M^\dagger] = [\tilde{\mathfrak{a}}_{M,F} : \tilde{\mathfrak{a}}_{M,F}^\dagger] [\tilde{\mathfrak{a}}_{G,F} : \tilde{\mathfrak{a}}_{G,F}^\dagger]^{-1}. $$
\end{lemma}
\begin{proof}
  D'une part, on a la suite exacte
  $$\xymatrix{
    1 \ar[r] & \dfrac{(\tilde{\mathfrak{a}}_{M,F}^\dagger + \mathfrak{a}_G) \cap \tilde{\mathfrak{a}}_{M,F}}{\tilde{\mathfrak{a}}_{M,F}^\dagger} \ar[r] & \dfrac{\tilde{\mathfrak{a}}_{M,F}}{\tilde{\mathfrak{a}}_{M,F}^\dagger} \ar[r] & \dfrac{\tilde{\mathfrak{a}}_{M,F}}{(\tilde{\mathfrak{a}}_{M,F}^\dagger + \mathfrak{a}_G) \cap \tilde{\mathfrak{a}}_{M,F}} \ar[r] \ar@{=}[d] & 1 \\
    & & & \dfrac{\tilde{\mathfrak{a}}_{M,F} + \mathfrak{a}_G}{\tilde{\mathfrak{a}}_{M,F}^\dagger + \mathfrak{a}_G} &
  }, $$
  et $(\tilde{\mathfrak{a}}_{M,F} + \mathfrak{a}_G)/(\tilde{\mathfrak{a}}_{M,F}^\dagger + \mathfrak{a}_G) = \tilde{\mathcal{L}}_M/\tilde{\mathcal{L}}_M^\dagger$. D'autre part,
  \begin{align*}
    \dfrac{(\tilde{\mathfrak{a}}_{M,F}^\dagger + \mathfrak{a}_G) \cap \tilde{\mathfrak{a}}_{M,F}}{\tilde{\mathfrak{a}}_{M,F}^\dagger} & = \dfrac{\tilde{\mathfrak{a}}_{M,F}^\dagger + (\mathfrak{a}_G \cap \tilde{\mathfrak{a}}_{M,F})}{\tilde{\mathfrak{a}}_{M,F}^\dagger} = \dfrac{\tilde{\mathfrak{a}}_{M,F}^\dagger + \tilde{\mathfrak{a}}_{G,F}}{\tilde{\mathfrak{a}}_{M,F}^\dagger} \\
    & = \dfrac{\tilde{\mathfrak{a}}_{G,F}}{\tilde{\mathfrak{a}}_{M,F}^\dagger \cap \tilde{\mathfrak{a}}_{G,F}} = \dfrac{\tilde{\mathfrak{a}}_{G,F}}{\tilde{\mathfrak{a}}_{G,F}^\dagger}
  \end{align*}
  car $\tilde{\mathfrak{a}}_{G,F}^\dagger = \mathfrak{a}_G \cap \tilde{\mathfrak{a}}_{M,F}^\dagger$. Cela permet de conclure.
\end{proof}

On pose $A_M(F)^1 := A_M(F) \cap \Ker H_M$ et $A_M(F)^{\dagger, 1} := A_M(F)^\dagger \cap \Ker H_M$.
\begin{lemma}\label{prop:facteur-a-2}
  On a
  $$ \iota_M [A_M(F)^1 : A_M(F)^{\dagger, 1}] [\tilde{\mathfrak{a}}_{M,F}: \tilde{\mathfrak{a}}_{M,F}^\dagger] = 1. $$
\end{lemma}
\begin{proof}
  Il suffit d'appliquer le lemme du serpent au diagramme suivant.
  $$\xymatrix{
    1 \ar[r] & A_M(F)^{\dagger, 1} \ar[r] \ar[d] & A_M(F)^\dagger \ar[r] \ar[d] & \tilde{\mathfrak{a}}_{M,F}^\dagger \ar[r] \ar[d] & 1 \\
    1 \ar[r] & A_M(F)^1 \ar[r] & A_M(F) \ar[r] & \tilde{\mathfrak{a}}_{M,F} \ar[r] & 1
  }$$
\end{proof}

Soient $k \in \R_{>0}$, $M \in \mathcal{L}(M_0)$ et $P \in \mathcal{P}(M)$. Posons
\begin{align*}
  \mu_{\alpha,k} & := k \log q_F \cdot \alpha^\vee, \quad \alpha \in \Delta_P, \\
  \mathcal{L}_{M,k} & := k \log q_F \cdot \Z\Delta_P^\vee.
\end{align*}
Alors $\mathcal{L}_{M,k}$ est un réseau dans $\mathfrak{a}^G_M$, qui est indépendant de $P$. Si $k \in \Z_{>0}$ est suffisamment divisible, alors $\mathcal{L}_{M,k} \subset \tilde{\mathcal{L}}_M$.

On a déjà normalisé la mesure de Haar sur $i\mathfrak{a}_G^*$ (resp. $i\mathfrak{a}_M^*$) dans \S\ref{sec:def-base}, ce qui détermine la mesure duale sur $\mathfrak{a}_G$ (resp. $\mathfrak{a}_M$) comme dans \cite[\S 2.5]{Li10a}. On vérifie que la mesure induite sur $\mathfrak{a}^G_M$ satisfait à $\mes(\mathfrak{a}^G_M/\tilde{\mathcal{L}}_M)=1$. Supposons toujours $k$ suffisamment divisible et posons\index[iFT2]{$\theta_{P,k}(\lambda)$}
\begin{gather}
  \theta_{P,k}(\lambda) := \mes(\mathfrak{a}^G_M/\mathcal{L}_{M,k})^{-1} \prod_{\alpha \in \Delta_P} \left(1 - e^{-\angles{\lambda, \mu_{\alpha,k}}} \right), \qquad \lambda \in \mathfrak{a}_{M,\C}^* .
\end{gather}

On va transformer $v_M^\dagger(x_1,x_2,T)$ en une expression qui ne dépend pas de $A_M(F)^\dagger$ et $A_G(F)^\dagger$. Les arguments sont identiques à ceux dans \cite[\S 6]{Ar91} sauf que certains facteurs supplémentaires interviennent. Montrons qu'ils disparaissent à la fin. On a
\begin{multline*}
  v_M^\dagger(x_1, x_2, T) = \iota_G \iota_M^{-1} [A_M(F)^1 : A_M(F)^{\dagger, 1}]^{-1} [A_G(F)^1 : A_G(F)^{\dagger, 1}] \cdot \\
  \cdot \sum_{X \in \tilde{\mathfrak{a}}_{M,F}^\dagger / \tilde{\mathfrak{a}}_{G,F}^\dagger} \sigma_M(X, \mathcal{Y}_M(x_1,x_2,T)).
\end{multline*}

On a $\tilde{\mathfrak{a}}_{M,F}^\dagger / \tilde{\mathfrak{a}}_{G,F}^\dagger = \tilde{\mathfrak{a}}_{M,F}^\dagger / (\tilde{\mathfrak{a}}_{M,F}^\dagger \cap \mathfrak{a}_G) = \tilde{\mathcal{L}}_M^\dagger$. Comme dans \cite[\S 6]{Ar91}, la somme sur $\tilde{\mathcal{L}}_M^\dagger$ se transforme en
$$ \sum_{\nu \in \tilde{\mathcal{L}}_M^{\dagger,\vee} /\mathcal{L}_M^\vee} \sum_{P \in \mathcal{P}(M)} [\mathcal{L}_M : \tilde{\mathcal{L}}_M^\dagger]^{-1} \mes (\mathfrak{a}^G_M/\mathcal{L}_{M,k})^{-1} \lim_{\Lambda \to 0} \left( \sum_{X \in \mathcal{L}_M / \mathcal{L}_{M,k}} e^{\angles{\Lambda + \nu, X_P(Y_P)}} \theta_{P,k}(\Lambda + \nu)^{-1} \right) $$
où $\Lambda \in (\mathfrak{a}^G_{M,\C})^*$ est supposé en position générale proche de $0$, et $Y_P = Y_P(x_1, x_2, T)$. Donc $v_M^\dagger(x_1,x_2,T)$ est le produit de
$$ \iota_G \iota_M^{-1} [A_M(F)^1 : A_M(F)^{\dagger, 1}]^{-1} [A_G(F)^1 : A_G(F)^{\dagger, 1}] [\tilde{\mathcal{L}}_M : \tilde{\mathcal{L}}_M^\dagger]^{-1} $$
avec
$$ \sum_{\nu \in \tilde{\mathcal{L}}_M^{\dagger,\vee} /\mathcal{L}_M^\vee} \lim_{\Lambda \to 0} \left( \sum_{P \in \mathcal{P}(M)} [\mathcal{L}_M : \mathcal{L}_{M,k}]^{-1} \sum_{X \in \mathcal{L}_M / \mathcal{L}_{M,k}} e^{\angles{\Lambda + \nu, X_P(Y_P)}} \theta_{P,k}(\Lambda + \nu)^{-1} \right). $$

La première expression vaut $1$ d'après les Lemmes \ref{prop:facteur-a-1}, \ref{prop:facteur-a-2}, tandis que la limite dans la deuxième expression est exactement celle dans le cas des groupes réductifs. C'est justifié de reprendre tous les arguments d'Arthur pour obtenir l'expression
\begin{gather}
  v_M^\dagger(x_1,x_2,T) = \sum_{\xi \in \frac{1}{N} \mathcal{L}_0^\vee / \mathcal{L}_0^\vee} q_\xi(T) e^{\angles{\xi, T}},
\end{gather}
où
\begin{itemize}
  \item $T \in \mathcal{L}_0 \cap \mathfrak{a}_0^+$ est suffisamment régulier en un sens indépendant de $M$, et $\mathfrak{a}_0^+$ est une chambre quelconque dans $\mathfrak{a}_0$;
  \item $N \in \Z_{>0}$ suffisamment divisible, indépendamment de $M$;
  \item pour tout $\xi$, $q_\xi$ est un polynôme.
\end{itemize}

C'est loisible de parler du terme constant $\tilde{v}_M(x_1,x_2) := q_0(0)$. La somme sur $\xi$ provient de la somme précédente sur $\nu \in \tilde{\mathcal{L}}_M^{\dagger,\vee} /\mathcal{L}_M^\vee$, donc elle dépend du choix de $A_M(F)^\dagger$; cependant le terme $q_0$ correspondant à $\xi=0$ n'en dépend pas. En adaptant \cite[(6.6)]{Ar91}, on obtient
\begin{gather}\label{eqn:v_M-tilde}
  \tilde{v}_M(x_1, x_2) = \lim_{\Lambda \to 0} \sum_{P \in \mathcal{P}(M)} |\mathcal{L}_M/\mathcal{L}_{M,k}|^{-1} \sum_{X \in \mathcal{L}_M/\mathcal{L}_{M,k}} e^{\angles{\Lambda, X_P + H_P(x_1) - H_{\bar{P}}(x_2)}} \theta_{P,k}(\Lambda)^{-1}
\end{gather}
où $k \in \Z_{>0}$ est suffisamment divisible en un sens indépendant de $M$. Cela permet de définir
\begin{gather}\label{eqn:int-orb-tilde}
  \tilde{J}_{\tilde{M}}(\gamma, f) := |D(\gamma)| \iota_M^2 \int_{(A_M(F)^\dagger \backslash G(F))^2} f_1(x_1^{-1} \gamma x_1) f_2(x_2^{-1} \gamma x_2) \tilde{v}_M(x_1, x_2) \dd x_1 \dd x_2 .
\end{gather}
pour tout $\gamma \in \Gamma_\text{ell}(M(F))$.

\begin{proposition}[Cf. {\cite[Proposition 6.1]{Ar91}}]\label{prop:J-geo-tilde}
  Supposons $F$ non archimédien. Alors il existe une décomposition
  $$ J^T(f) = \sum_{\xi \in \frac{1}{N} \mathcal{L}_0^\vee / \mathcal{L}_0^\vee} p_\xi(T,f) e^{\angles{\xi, T}}, \qquad T \in \mathcal{L}_0 \cap \mathfrak{a}_0^+ , $$
  où $T$, $\mathfrak{a}_0^+$, $N$ sont comme précédemment et $p_\xi(\cdot,f)$ sont des polynômes. Son terme constant $\tilde{J}(f) := p_0(0,f)$ admet une décomposition
  $$ \tilde{J}(f) = \sum_{M \in \mathcal{L}(M_0)} |W^M_0| |W^G_0|^{-1} \int_{\Gamma_{\mathrm{ell}}(M(F))} \tilde{J}_{\tilde{M}}(\gamma, f) \dd \gamma . $$
\end{proposition}
\begin{remark}
  Les distributions $K^T$, $J^T$ sont $W^G_0$-invariantes, donc $\tilde{J}(f)$ ne dépend pas du choix de $\mathfrak{a}_0^+$. En particulier, le terme constant $\tilde{J}(f)$ n'en dépend pas.
\end{remark}

Le cas archimédien est plus simple du point de vue combinatoire.
\begin{proposition}\label{prop:J-geo-tilde-arch}
  Supposons $F$ archimédien. Alors la Proposition \ref{prop:J-geo-tilde} demeure valable si l'on remplace les réseaux $\mathcal{L}_M$, $\tilde{\mathcal{L}}_M$, $\tilde{\mathcal{L}}_M^\dagger$, $\mathcal{L}_{M,k}$ par $\mathfrak{a}^G_M$, et si l'on remplace $\theta_{P,k}$ par la fonction $\theta_P$ définie dans \cite[\S 4.1]{Li10a} dans la définition \eqref{eqn:int-orb-tilde} de $\tilde{J}_{\tilde{M}}(\gamma,f)$.
\end{proposition}

\subsection{Le côté spectral}\label{sec:LTF-spec}
La référence pour les discussions suivantes sont \cite[\S\S 8-11]{Ar91}. Fixons $P_0 \in \mathcal{P}(M_0)$. On a
\begin{align*}
  K^T(f) & = \sum_{M \in \mathcal{L}(M_0)} |W^M_0| |W^G_0|^{-1} \int_{\Pi_{2,-}(\tilde{M})} K^T(\sigma,f) \dd \sigma, \\
  K^T(\sigma,f) & := \mu(\sigma) \iota_G \int_{A_G(F)^\dagger \backslash G(F)} \sum_{S \in \mathcal{B}_{\tilde{P}}(\sigma)} \Tr(\mathcal{I}_{\tilde{P}}(\sigma, \tilde{x})S(f)) \overline{\Tr(\mathcal{I}_{\tilde{P}}(\sigma, \tilde{x})S)} u(x,T) \dd x,
\end{align*}
où $P \in \mathcal{P}(M_0)$ est standard.

Soient $(\sigma, V) \in \Pi_{2,-}(\tilde{M})$, $\chi \in \Im X(\tilde{M})$. Rappelons que l'on peut réaliser $\mathcal{I}_{\tilde{P}}(\sigma \otimes \chi)$ tel que l'espace $\mathcal{I}_{\tilde{P}}(V)$ muni de l'action de $\tilde{K}$ ne dépend pas de $\chi$. Dans ce qui suit, on fixe un point base $\sigma$ dans chaque $\Im X(\tilde{M})$-orbite de $\Pi_{2,-}(\tilde{M})$.

Donc dans l'expression de $K^T(f)$ on peut regrouper les termes selon $M$ et les $\Im X(\tilde{M})$-orbites de $\Pi_{2,-}(\tilde{M})$. On se ramène à l'étude de l'intégrale
$$ \iota_G \int_{A_G(F)^\dagger \backslash G(F)} \Tr(\mathcal{I}_{\tilde{P}}(\sigma \otimes \chi, \tilde{x})S(f)) \overline{\Tr(\mathcal{I}_{\tilde{P}}(\sigma \otimes \chi, \tilde{x})S)} u(x,T) \dd x $$
où $S \in \mathcal{B}_{\tilde{P}}(\sigma)$ est fixé, et $\chi$ varie dans $\Im X(\tilde{M})$.

On note $\mathcal{O}$ la $\Im X(\tilde{M})$-orbite contenant $\sigma$. Pour simplifier la vie, supposons momentanément $F$ non archimédien. En examinant les définitions dans \S\ref{sec:Plancherel}, on voit que $S$ et $S(f)$ appartiennent à $C^\infty(\mathcal{O},\tilde{P})$, regardées comme fonctions $\chi \mapsto S_\chi$ et $\chi \mapsto S(f)_\chi$, et que
\begin{align*}
  \Tr(\mathcal{I}_{\tilde{P}}(\sigma \otimes \chi, \tilde{x})S(f)) & = E^{\tilde{G}}_{\tilde{P}}(S(f)_\chi)(\tilde{x}), \\
  \Tr(\mathcal{I}_{\tilde{P}}(\sigma \otimes \chi, \tilde{x})S) & = E^{\tilde{G}}_{\tilde{P}}(S_\chi)(\tilde{x}).
\end{align*}

Avec la notation usuelle $(a|b) = a \bar{b}$ pour $a, b \in \C$, définissons le produit scalaire tronqué des coefficients
$$
  \Omega^T_{\tilde{P}}(\sigma \otimes \chi, S(f), S) := \iota_G \int_{A_G(F)^\dagger \backslash G(F)} \left( E^{\tilde{G}}_{\tilde{P}}(S(f)_\chi)(\tilde{x}) \,|\, E^{\tilde{G}}_{\tilde{P}}(S_\chi)(\tilde{x}) \right) u(x,T) \dd x
$$

On obtient donc l'expression suivante pour $K^T(f)$:
\begin{gather}\label{eqn:K^T-Omega}
  \sum_{M \in \mathcal{L}(M_0)} |W^M_0| |W^G_0|^{-1} \sum_{\sigma} \sum_S |\text{Stab}_{\Im X(\tilde{M})}(\sigma)|^{-1} \int_{\Im X(\tilde{M})} \mu(\sigma \otimes \chi) \Omega^T_{\tilde{P}}(\sigma \otimes \chi, S(f), S) \dd\sigma,
\end{gather}
où
\begin{itemize}
  \item $\sigma$ parcourt un système de représentants de $\Pi_{2,-}(\tilde{M}) / \Im X(\tilde{M})$;
  \item $S$ parcourt $\mathcal{B}_{\tilde{P}}(\sigma)$;
  \item $P \in \mathcal{P}(M)$ est standard.
\end{itemize}

\begin{remark}
  Dans \cite[\S\S 7-8]{Ar91}, Arthur choisit le langage des intégrales d'Eisenstein au lieu des coefficients d'induites. Nous laissons le soin au lecteur de réconcilier les définitions.
\end{remark}

Comme dans \cite{Ar91}, le côté spectral nécessite des majorations en trois étapes. Donnons-en une esquisse.

\paragraph{De $K^T$ à $k^T$}
Pour $M \in \mathcal{P}(M_0)$, $P \in \mathcal{P}(M_0)$ standard et $\Psi_1, \Psi_2 \in \mathcal{A}_{2,-}(\tilde{M})$, on note $\varphi_P: \mathfrak{a}_0 \to \R$ la fonction caractéristique de l'ensemble
$$ \{H \in \mathfrak{a}_0 : \angles{\varpi, H} \leq 0, \forall \varpi \in \hat{\Delta}_P \}$$
et on pose
$$ r^T_{\tilde{P}}(\Psi_1 | \Psi_2) := \iota_G \int_{A_G(F)^\dagger \backslash M(F)} (\Psi_1(\tilde{m})|\Psi_2(\tilde{m})) \varphi_P(H_M(m)-T_P) \dd m . $$

Pour tout $\sigma \in \Pi_{2,-}(\tilde{M})$ et $S_1, S_2 \in L(\sigma, \tilde{P})$, on pose
$$
  \omega^T_{\tilde{P}}(\sigma, S_1, S_2) := \sum_{P_1 \supset P_0} r^T_{\tilde{P}_1}(E^{\tilde{G}}_{\tilde{P}}(S_1)^w_{\tilde{P}_1} \,|\, E^{\tilde{G}}_{\tilde{P}}(S_2)^w_{\tilde{P}_1} ) .
$$

Avec la même convention de \eqref{eqn:K^T-Omega}, posons $k^T(f)$ égale à
\begin{gather}\label{eqn:k^T-omega}
  \sum_{M \in \mathcal{L}(M_0)} |W^M_0| |W^G_0|^{-1} \sum_{\sigma} \sum_S |\text{Stab}_{\Im X(\tilde{M})}(\sigma)|^{-1} \int_{\Im X(\tilde{M})} \mu(\sigma \otimes \chi) \omega^T_{\tilde{P}}(\sigma \otimes \chi, S(f), S) \dd\sigma .
\end{gather}

\begin{proposition}
  Soit $\delta > 0$, il existe des constantes $C, \epsilon > 0$ telles que
  $$ |K^T(f) - k^T(f)| \leq C e^{-\epsilon \|T\|} $$
  pourvu que $d(T) \geq \delta\|T\|$.
\end{proposition}
\begin{proof}
  C'est l'analogue de \cite[Lemma 10.1]{Ar91}. Le fait crucial est la majoration \cite[Theorem 8.1]{Ar91} reliant $\Omega^T_{\tilde{P}}$ et $\omega^T_{\tilde{P}}$, dont les ingrédients de la preuve sont établis dans \S\ref{sec:Plancherel}.
\end{proof}

\paragraph{De $k^T$ à $J^T_\text{spec}$}
On définit une distribution $J^T_\text{spec}(f)$ comme dans \cite[Lemma 11.1]{Ar91}, qui consiste en des fonctions $c$ de Harish-Chandra et des fonctions combinatoires. Sa définition précise est malheureusement trop compliquée à rapporter ici. On montre que pour tout $n \in \Z_{\geq 1}$ et tout $\delta > 0$, il existe une constante $c_n$ telle que $|k^T(f)-J^T_\text{spec}(f)| \leq c_n \|T\|^{-n}$ lorsque $d(T) \geq \delta\|T\|$. Vu l'étape précédente et la majoration pour $|K^T(f)-J^T(f)|$, on obtient $|J^T(f)-J^T_\text{spec}(f)| \leq c_n \|T\|^{-n}$ sous les mêmes conditions, quitte à agrandir $c_n$.

D'autre part, on montre, comme dans la preuve de \cite[Lemma 11.1]{Ar91}, que
$$ J^T_\text{spec}(f) = \sum_{\xi \in \frac{1}{N}\mathcal{L}_0^\vee / \mathcal{L}_0^\vee} q_\xi(T,f) e^{\angles{\xi, T}}, $$
pourvu que $T \in \mathcal{L}_0 \cap \mathfrak{a}_0^+$ et $N \in \Z_{\geq 1}$ suffisamment divisible, où $\mathfrak{a}_0^+$ est une chambre fixée. Les fonctions $q_\xi(\cdot,T)$ sont des polynômes. Or $J^T(f)$ admet une décomposition du même genre pour $T \in \mathfrak{a}_0^+$. Comme $\mathfrak{a}_0^+$ est quelconque, la majoration précédente pour $|J^T(f)-J^T_\text{spec}(f)|$ affirme que $J^T(f)=J^T_\text{spec}(f)$ pour de tels $T$.

Notons $\tilde{J}_\text{spec}(f) := q_0(0,f)$ le terme constant. Il en résulte que $\tilde{J}_\text{spec}(f)=\tilde{J}(f)$.

\paragraph{Description de $\tilde{J}_\text{spec}$}
Donnons finalement une expression explicite pour $\tilde{J}_\text{spec}$ comme dans \cite[Corollary 11.2, Proposition 11.3]{Ar91}.

Soient $L, M \in \mathcal{L}(M_0)$, $L \supset M$. Nous posons\index[iFT2]{$W^L(M)_\text{reg}$}
$$ W^L(M)_\text{reg} := \{t \in W^L(M) : \det(t-1|\mathfrak{a}^L_M) \neq 0 \}. $$

Soit $t \in W^G(M)$, on note $\Pi_{2,-}(\tilde{M})^t := \{\sigma \in \Pi_{2,-}(\tilde{M}) : t\sigma \simeq \sigma \}$. Soient $L, M \in \mathcal{L}(M_0)$, $L \supset M$, $R \in \mathcal{P}(L)$, $P \in \mathcal{P}(M)$, $\zeta \in i\mathfrak{a}_L^*$, $t \in W^L(M)_\text{reg}$, $\sigma \in \Pi_{2,-}(\tilde{M})^t$. Introduisons les objets suivants.
\begin{itemize}
  \item $\Pi := P \cap L$, $R(\Pi)$ est l'unique élément de $\mathcal{P}(M)$ tel que $R(\Pi) \cap L = \Pi$ et $R(\Pi) \subset R$.
  \item Fixons des facteurs normalisants faibles $r_{\tilde{Q}'|\tilde{Q}}(\sigma)$ (rappel: la Définition \ref{def:normalisant-faible}), d'où les opérateurs d'entrelacement normalisés $R_{\tilde{Q}'|\tilde{Q}}(\sigma)$ pour $Q, Q' \in \mathcal{P}(M)$.
  \item $R_{\tilde{P}}(t, \sigma) := \sigma(\tilde{t}) \circ A(\tilde{t}) \circ R_{t^{-1}\tilde{P}|\tilde{P}}(\sigma)$, où $\tilde{t} \in \tilde{L} \cap \tilde{K}$ est un représentant de $t$, et $\sigma(\tilde{t})$ est l'isomorphisme $\mathcal{I}_{\tilde{P}}(\tilde{t}\sigma) \rightiso \mathcal{I}_{\tilde{P}}(\sigma)$ induit par un isomorphisme fixé $\tilde{t}\sigma \rightiso \sigma$. C'est bien défini à une constante multiplicative près.
  \item $\tau_{1,\tilde{R}}(\zeta)$ est
    $$ \Tr\left( \mathcal{I}_{\tilde{P}}(\sigma, \check{f}_1) R_{\tilde{P}}(t,\sigma)^{-1} J_{\widetilde{P}|\widetilde{R(\Pi)}}(\sigma_\zeta) J_{\widetilde{P}|\widetilde{R(\Pi)}}(\sigma)^{-1} \right). $$
  \item $\tau_{2,\tilde{R}}(\zeta)$ est
    $$ \Tr\left( J_{\widetilde{R(\Pi)}|\widetilde{P}}(\sigma)^{-1} J_{\widetilde{R(\Pi)}|\widetilde{P}}(\sigma_\zeta) R_{\tilde{P}}(t,\sigma) \mathcal{I}_{\tilde{P}}(\sigma, f_2) \right). $$
  \item $\tilde{J}_{\tilde{L}}(\sigma,t,f)$ est
    \begin{gather}\label{eqn:J-tilde-L-t}
      \lim_{\zeta \to 0} \sum_{R \in \mathcal{P}(L)} \tau_{1,\tilde{R}}(\zeta) \tau_{2,\widetilde{\bar{R}}}(\zeta) |\mathcal{L}_L/\mathcal{L}_{L,k}|^{-1} \sum_{X \in \mathcal{L}_L/\mathcal{L}_{L,k}} e^{\angles{\zeta, X_R}} \theta_{R,k}(\zeta)^{-1}
    \end{gather}
    avec $k \in \Z_{\geq 1}$ suffisamment divisible.
  \item On note
    \begin{align*}
      \Sigma_P^\text{red}(\sigma) & := \{\beta \in \Sigma_P^\text{red} : r_\beta \text{ a un pôle en } \sigma \} \\
      & = \{\beta \in \Sigma_P^\text{red} : \mu_\beta \text{ a un zéro en } \sigma \},
    \end{align*}
    et on pose
    \begin{gather}\label{eqn:epsilon-sigma}
      \epsilon_\sigma(t) := (-1)^{|t \Sigma_P^\text{red}(\sigma) \cap \Sigma_{\bar{P}}^\text{red}(\sigma)|}.
    \end{gather}
\end{itemize}

Pour $L, \sigma$ comme ci-dessus, on munit $\Im X(\tilde{L}) \cdot \sigma$ de la mesure quotient déduite de la mesure de $i\mathfrak{a}_L^*$.

\begin{proposition}\label{prop:J-spec-tilde}
  On a
  $$ \tilde{J}_\mathrm{spec}(f) = \sum_{L,M,t} |W^M_0| |W^G_0|^{-1} |\det(t-1|\mathfrak{a}^L_M)|^{-1} \sum_{\sigma} \int_{\Im X(\tilde{L}) \cdot \sigma} \epsilon_{\sigma'}(t) \tilde{J}_{\tilde{L}}(\sigma',t,f) \dd\sigma' $$
  où $L,M,t$ sont comme précédemment, et $\sigma$ parcourt les $\Im X(\tilde{L})$-orbites de $\Pi_{2,t}(\tilde{M})^t$.
\end{proposition}
\begin{proof}
  Puisque la théorie des opérateurs d'entrelacement, des fonctions $c$ de Harish-Chandra et la formule de Plancherel pour revêtements sont établies dans \S\ref{sec:Plancherel}, il suffit d'adapter \cite[\S 11]{Ar91}. Comme pour le côté géométrique, des facteurs de la forme $\iota_M$, $[A_M(F)^1 : A_M(F)^{\dagger, 1}]$ ou $[\tilde{\mathfrak{a}}_{M,F} : \tilde{\mathfrak{a}}_{M,F}^\dagger]$ (où $M \in \mathcal{L}(M_0)$) peuvent intervenir dans $J_\text{spec}^T(f)$, donc dans la description de $\tilde{J}_{\tilde{L}}(\sigma',t,f)$. Toutefois on peut inspecter les arguments d'Arthur et montrer, comme dans \S\ref{sec:LTF-geo}, que ces facteurs se compensent à l'aide des Lemmes \ref{prop:facteur-a-1}, \ref{prop:facteur-a-2}. Ces compensations ne sont pas surprenantes car le formalisme est choisi de sorte que si $\tilde{G} = \bmu_m \times G(F)$ et $\rev = (1, \identity)$, alors on revient à la situation considérée par Arthur.
\end{proof}

\subsection{Interlude: \texorpdfstring{$R$}{R}-groupes}\label{sec:R-groupe}
Avant de procéder à la formule des traces locale, rappelons brièvement le formalisme de $R$-groupes de \cite[\S 2]{Ar93}. Des sources possibles des démonstrations sont \cite{Si78, Si79}. Les preuves reposent sur
\begin{itemize}
  \item la formule de Plancherel,
  \item la normalisation des opérateurs d'entrelacement,
  \item une formule de Casselman \cite[I.4.1]{Wa03} des coefficients matriciels du module de Jacquet, pour le cas archimédien.
\end{itemize}
On a tous ces ingrédients pour les revêtements.

Soient $M \in \mathcal{L}(M_0)$, $P \in \mathcal{P}(M)$, $(\sigma,V) \in \Pi_{2,-}(\tilde{M})$. Notons $\Pi_\sigma(\tilde{G})$ l'ensemble des classes de constituants irréductibles de $\mathcal{I}_{\tilde{P}}(\sigma)$, qui ne dépend pas du choix de $P$. Notons\index[iFT2]{$W_\sigma$} $W_\sigma := \{w \in W^G(M) : w\sigma \simeq \sigma \}$. Fixons des facteurs normalisants faibles $r_{\tilde{Q}'|\tilde{Q}}(\sigma)$ et les opérateurs d'entrelacement normalisés $R_{\tilde{Q}'|\tilde{Q}}(\sigma)$.

Soient $w \in W_\sigma$ et $\tilde{w} \in \tilde{K}$ un représentant, alors $\sigma$ se prolonge en une représentation du groupe $\widetilde{M}_w^+$ engendré par $\tilde{M}$ et $\tilde{w}$; désignons cette représentation par $\sigma_w$. Ce prolongement est  unique à une constante près, nous y imposerons des conditions plus tard.

Définissons l'opérateur d'entrelacement
\begin{align*}
  A(\sigma_w): \mathcal{I}_{w^{-1} \tilde{P}}(\sigma) & \rightiso \mathcal{I}_{\tilde{P}}(\sigma), \\
  \phi(\cdot) & \mapsto \sigma_w(\tilde{w}) \phi(\tilde{w}^{-1} \cdot ).
\end{align*}

Posons $R_{\tilde{P}}(w,\sigma) := A(\sigma_w) R_{w^{-1}\tilde{P}|\tilde{P}}(\sigma)$ et\index[iFT2]{$W_\sigma^0$}\index[iFT2]{$R_\sigma$}
\begin{align*}
  W_\sigma^0  &:= \{w \in W_\sigma : R_{\tilde{P}}(w,\sigma) \in \C^\times \identity \}, \\
  R_\sigma & := W_\sigma/W_\sigma^0 .
\end{align*}
Notons que $R_\sigma$ ne dépend pas de $P$.

L'un des faits de la théorie de $R$-groupes est que $W_\sigma^0$ est le groupe de Weyl associé au système de racines engendré par $\{\beta \in \Sigma_P^\text{red} : \mu_\beta(\sigma)=0\}$. Si l'on fixe une chambre $\mathfrak{a}_\sigma^+$ de ce système, alors on peut identifier $R_\sigma$ et $\{w \in W_\sigma : w\mathfrak{a}_\sigma^+ = \mathfrak{a}_\sigma^+ \}$; avec ce choix, on peut écrire
$$ W_\sigma = W_\sigma^0 \rtimes R_\sigma . $$

Supposons maintenant que pour tout $w \in W_\sigma^0$, $\sigma_w$ est choisi de sorte que $R_{\tilde{P}}(w, \sigma)=\identity$. L'application $r \mapsto R_{\tilde{P}}(r,\sigma)$ pour $r \in R_\sigma$ n'est pas forcément un homomorphisme: il existe un $2$-cocycle $\eta_\sigma: R_\sigma^2 \to \C^\times$ tel que $R_{\tilde{P}}(r_1 r_2, \sigma) = \eta_\sigma(r_1,r_2) R_{\tilde{P}}(r_1,\sigma) R_{\tilde{P}}(r_2,\sigma)$. Plus précisément, on a $\eta_\sigma(r_1, r_2) = A(\sigma_{r_1 r_2}) A(\sigma_{r_1})^{-1} A(\sigma_{r_2})^{-1}$. Fixons désormais une extension centrale\index[iFT2]{$R_\sigma$}\index[iFT2]{$Z_\sigma$}\index[iFT2]{$\chi_\sigma$}
\begin{gather}\label{eqn:R-ext}
  1 \to Z_\sigma \to \tilde{R}_\sigma \to R_\sigma \to 1
\end{gather}
de sorte que $Z_\sigma$ est fini et l'image de $\eta_\sigma$ dans $H^2(\tilde{R}_\sigma, \C^\times)$ est triviale. Autrement dit, il existe $\xi_\sigma: \tilde{R}_\sigma \to \C^\times$ tel que $\eta_\sigma(r_1, r_2) = \xi_\sigma(r_1 r_2) \xi_\sigma(r_1)^{-1} \xi_\sigma(r_2)^{-1}$ pour tous $r_1, r_2 \in \tilde{R}_\sigma$. On en déduit le caractère $\chi_\sigma: Z_\sigma \to \C^\times$ tel que $\xi_\sigma(zr)=\chi_\sigma(z)\xi_\sigma(r)$ pour tout $r \in \tilde{R}_\sigma$ et tout $z \in Z_\sigma$. Si l'on pose\index[iFT2]{$\tilde{R}_{\tilde{P}}(r,\sigma)$}
$$ \tilde{R}_{\tilde{P}}(r,\sigma) := \xi_\sigma(r)^{-1} R_{\tilde{P}}(r,\sigma), \qquad r \in \tilde{R}_\sigma $$
alors $r \mapsto \tilde{R}_{\tilde{P}}(r,\sigma)$ est un homomorphisme avec $\tilde{R}_{\tilde{P}}(zr,\sigma)=\chi_\sigma(z)^{-1} \tilde{R}_{\tilde{P}}(r,\sigma)$.

Notons $\Pi(\tilde{R}_\sigma, \chi_\sigma)$ l'ensemble des classes de représentation irréductibles de $\tilde{R}_\sigma$ dont la restriction sur $Z_\sigma$ est $\chi_\sigma$. Le théorème principal des $R$-groupes (d'après Harish-Chandra, Knapp, Stein, Silberger) s'exprime comme suit, ce que nous admettre comme une hypothèse dans le cas non archimédien.

\begin{theorem}\label{prop:R}
  La représentation $\mathcal{R}_{\tilde{P}}(r,\tilde{x}) = \tilde{R}_{\tilde{P}}(r,\sigma) \mathcal{I}_{\tilde{P}}(\sigma, \tilde{x})$ de $\tilde{R}_\sigma \times \tilde{G}$ sur $\mathcal{I}_{\tilde{P}}(V)$ admet une décomposition
  $$ \mathcal{R}_{\tilde{P}} = \bigoplus_{\rho \in \Pi(\tilde{R}_\sigma, \chi_\sigma)} (\rho^\vee \boxtimes \pi_\rho) $$
  où $\rho \mapsto \pi_\rho$ est une bijection de $\Pi(\tilde{R}_\sigma, \chi_\sigma)$ sur $\Pi_\sigma(\tilde{G})$. En particulier, on a   $$ \Tr(\tilde{R}_{\tilde{P}}(r,\sigma) \mathcal{I}_{\tilde{P}}(\sigma, f)) = \sum_{\rho \in \Pi(\tilde{R}_\sigma, \chi_\sigma)} \Tr\rho^\vee (r) \cdot \angles{\Theta_{\pi_\rho}, f}, \qquad r \in \tilde{R}_\sigma, \; f \in C_{c,\asp}^\infty(\tilde{G}). $$
\end{theorem}
\begin{remark}
  Selon notre normalisation $R_{\tilde{P}}(w, \sigma)=\identity$ pour tout $w \in W_\sigma^0$, on voit que l'opérateur $\tilde{R}_{\tilde{P}}(r, \sigma)$ est indépendant du plongement $R_\sigma \hookrightarrow W_\sigma$, autrement dit indépendant du choix de la chambre $\mathfrak{a}_\sigma^+$.
\end{remark}

Enregistrons des propriétés de $R$-groupes, dont les preuves se trouvent dans \cite{Ar93}.
\begin{enumerate}
  \item Fixons $\mathfrak{a}_\sigma^+$. Soit $L \in \mathcal{L}(M)$ tel que $\mathfrak{a}_\sigma^+$ contient un ouvert de $\mathfrak{a}_L$. Notons $R^L_\sigma$ le $R$-groupe relativement à $\tilde{L}$ au lieu de $\tilde{G}$, alors on a une identification $R^L_\sigma = R_\sigma \cap W^L(M)$. Vu les propriétés des opérateurs d'entrelacement normalisés, on peut tirer l'extension centrale \eqref{eqn:R-ext} par $R^L_\sigma \hookrightarrow R_\sigma$ et obtient
  $$ 1 \to Z_\sigma \to \tilde{R}^L_\sigma \to R^L_\sigma \to 1 $$
  qui trivialise encore les $2$-cocycles $\xi^L_\sigma$ associés à $\tilde{L}$.

  \item Soit $L$ comme ci-dessus. On note $\mathbf{K}_0(\tilde{R}_\sigma, \chi_\sigma)$ l'espace des caractères virtuels engendré par $\Pi(\tilde{R}_\sigma, \chi_\sigma)$. Idem pour $\mathbf{K}_0(\tilde{R}^L_\sigma, \chi_\sigma)$. Alors un élément de $\mathbf{K}_0(\tilde{R}_\sigma, \chi_\sigma)$ est induit de $\mathbf{K}_0(\tilde{R}^L_\sigma, \chi_\sigma)$ si et seulement si le caractère virtuel associé de $\tilde{G}$ est induit d'un caractère virtuel engendré par $\Pi_\sigma(\tilde{L})$.

  \item Le signe $\epsilon_\sigma(t)$ défini dans \eqref{eqn:epsilon-sigma} est égal à $(-1)^{\ell(w^0)}$ si $t = w^0 r$ avec $w^0 \in W^0_\sigma$, $r \in R_\sigma$.
\end{enumerate}

Supposons de plus que de tels choix (eg. facteurs normalisants, $\mathfrak{a}_\sigma^+$, $\sigma_w$, l'extension centrale \eqref{eqn:R-ext}, etc.) sont faits pour tout $w\sigma$, $w \in W^G_0$, de façon compatible. Par exemple, on exige que l'action de $w \in W^G_0$ induit un isomorphisme $\tilde{R}_\sigma \rightiso \tilde{R}_{w\sigma}$. L'étape suivante est d'introduire des espaces de paramètres convenables pour le côté spectral.

Soit $(M,\sigma,r)$ un triplet avec $M \in \mathcal{L}(M_0)$, $\sigma \in \Pi_{2,-}(\tilde{M})$, $r \in \tilde{R}_\sigma$. On dit qu'il est essentiel si pour tout $z \in Z_\sigma$ tel que $zr$ est conjugué à $r$, on a $\chi_\sigma(z)=1$. Le groupe de Weyl $W^G_0$ opère sur de tels triplets par $w(M,\sigma,r) = (wM, w\sigma, wrw^{-1})$. Posons\index[iFT2]{$R_{\sigma,\text{reg}}$} $R_{\sigma, \text{reg}} := R_\sigma \cap W^G(M)_\text{reg}$, notons $\tilde{R}_{\sigma, \text{reg}}$ son image réciproque dans $\tilde{R}_\sigma$ et\index[iFT2]{$T(\tilde{G}), T_\text{disc}(\tilde{G}), T_\text{ell}(\tilde{G})$}
\begin{align*}
  \tilde{T}(\tilde{G}) & := \{ (M,\sigma,r) : \text{un triplet essentiel}. \}, \\
  \tilde{T}_{\text{disc}}(\tilde{G}) & := \{ (M,\sigma,r) \in T(\tilde{G}) : W_\sigma^0 \cdot r \cap W^G(M)_\text{reg} \neq \emptyset \}, \\
  \tilde{T}_\text{ell}(\tilde{G}) & := \{ (M,\sigma,r) \in T(\tilde{G}) : r \in \tilde{R}_{\sigma,\text{reg}} \}, \\
\end{align*}

On a $\tilde{T}(\tilde{G}) \supset \tilde{T}_{\text{disc}}(\tilde{G}) \supset \tilde{T}_\text{ell}(\tilde{G})$. Le groupe $\Im X(\tilde{G})$ opère sur $\tilde{T}_\text{ell}(\tilde{G})$ par $\chi \cdot (M, \sigma, r) = (M, \sigma \otimes \chi, r)$, ce qui munit $\tilde{T}_\text{ell}(\tilde{G})$ d'une structure de variété analytique connexe. On vérifie aussi que $\tilde{T}(\tilde{G}) = \bigsqcup_{L \in \mathcal{L}(M_0)} \tilde{T}_\text{ell}(\tilde{L})$. Donc $\tilde{T}(\tilde{G})$ et $\tilde{T}_{\text{disc}}(\tilde{G})$ sont aussi munis de structures de variété analytique. Posons

\begin{align*}
  T(\tilde{G}) & := \tilde{T}(\tilde{G})/ W^G_0, \\
  T_\text{disc}(\tilde{G}) & := \tilde{T}_\text{disc}(\tilde{G})/W^G_0 , \\
  T_\text{ell}(\tilde{G}) & := \tilde{T}_\text{ell}(\tilde{G})/W^G_0  .
\end{align*}

Certes, on peut introduire l'équivariance sous $\bmu_m$ et définir les espaces $T_-(\tilde{G})$, $T_{\asp}(\tilde{G})$, etc. Munissons $T_{\text{disc},-}(\tilde{G})$ de la mesure de Radon telle que
$$ \int_{T_{\text{disc},-}(\tilde{G})} \theta(\tau) \dd\tau = \sum_{\tau \in T_{\text{disc},-}(\tilde{G})/\Im X(\tilde{G})} |\tilde{R}_{\sigma,r}|^{-1} \int_{\Im X(\tilde{G}) \cdot \tau} \theta(\tau') \dd\tau' $$
pour toute $\theta \in C_c(T_{\text{disc},-}(\tilde{G}))$, où $\tilde{R}_{\sigma,r}$ est le stabilisateur de $r$ dans $\tilde{R}_\sigma$, et $\Im X(\tilde{G}) \cdot \tau$ est muni de la mesure quotient déduite de la mesure de $i\mathfrak{a}_G^*$.

\begin{definition}\index[iFT2]{représentation elliptique}
  Une représentation $\pi \in \Pi_{\text{temp}}(\tilde{G})$ est dite elliptique si la restriction de son caractère sur l'ensemble des éléments réguliers $F$-elliptiques n'est pas nulle. On définit les représentations virtuelles tempérées elliptiques de la même manière.
\end{definition}

La proposition suivante décrit les représentations tempérées (resp. virtuelles tempérées) en termes du $R$-groupe; elle explique aussi la notation $T_{\text{ell},-}(\tilde{G})$. Sa démonstration dans le cas non archimédien dans \cite[\S 2]{Ar93} nécessite un résultat de Kazhdan qui sera prouvé dans le Théorème \ref{prop:KZ0}. Nous n'utiliserons pas cette proposition dans cet article.

\begin{proposition}
  Soient $M \in \mathcal{L}(M_0)$, $\sigma \in \Pi_{2,-}(\tilde{M})$ et $\rho \in \Pi(\tilde{R}_\sigma, \chi_\sigma)$. La représentation $\pi_\rho \in \Pi_\sigma(\tilde{G})$ est elliptique si et seulement si $\Tr\rho$ ne s'annule pas sur $\tilde{R}_{\sigma,\mathrm{reg}}$. L'ensemble $T_{\mathrm{ell},-}(\tilde{G})$ paramètre une base de l'espace engendré par les représentations virtuelles tempérées elliptiques spécifiques de $\tilde{G}$.
\end{proposition}

Pour tout $g \in \mathcal{C}_{\asp}(\tilde{G})$ et $\tau = (M,\sigma,r) \in T(\tilde{G})$, définissons\index[iFT2]{$\Theta(\tau,\cdot)$}\index[iFT2]{$i(\tau)$}
\begin{align}
  \label{eqn:Theta(tau)} \Theta(\tau,g) & := \Tr(\tilde{R}_{\tilde{P}}(r,\sigma) \mathcal{I}_{\tilde{P}}(\sigma,g)), \\
  \label{eqn:i(tau)} i(\tau) & := |W_\sigma^0|^{-1} \sum_{t \in W_\sigma^0 \cdot r \cap W^G(M)_\text{reg}} \epsilon_\sigma(t) |\det(1-t|\mathfrak{a}^G_M)|^{-1}, \quad \text{lorsque } \tau \in T_{\text{disc}}(\tilde{G}).
\end{align}
On vérifie que $\Theta$ est bien défini, c'est-à-dire qu'il ne dépend que de la $W^G_0$-orbite de $(M,\sigma,r)$. D'autre part, on vérifie que $\Theta((M, \sigma, zr), g)=\chi_\sigma(z)^{-1} \Theta((M,\sigma,r), g)$ pour tout $z \in Z_\sigma$. Idem si l'on considère $g \in \mathcal{C}_-(\tilde{G})$ et $\tau \in T_{\text{disc},\asp}(\tilde{G})$.

Soit $f=f_1 f_2$ avec $f_1 \in \mathcal{H}_-(\tilde{G})$, $f_2 \in \mathcal{H}_{\asp}(\tilde{G})$. On définit\index[iFT2]{$I_\text{disc}$}
\begin{gather}\label{eqn:I_disc}
  I_\text{disc}(f) = \int_{T_{\text{disc},-}(\tilde{G})} i(\tau) \Theta(\tau^\vee, f_1) \Theta(\tau, f_2) \dd\tau.
\end{gather}

Cette distribution est reliée à la formule des traces locale grâce au fait suivant.

\begin{proposition}
  Soit $f=f_1 f_2$ avec $f_1 \in \mathcal{H}_-(\tilde{G})$, $f_2 \in \mathcal{H}_{\asp}(\tilde{G})$. Alors $I_\mathrm{disc}(f)$ est la somme des termes correspondant à $L=G$ dans l'expression de $\tilde{J}_\mathrm{spec}(f)$ dans la Proposition \ref{prop:J-spec-tilde}.
\end{proposition}
\begin{proof}
  Il suffit de reprendre \cite[pp.95-96]{Ar93}.
\end{proof}

\subsection{La formule des traces locale}
\paragraph{Le côté géométrique}
Soient $x = (x_1,x_2) \in G(F)^2$, $M \in \mathcal{L}(M_0)$. Rappelons que l'on choisit la mesure de Haar sur $\mathfrak{a}^G_M$ de sorte que $\mes(\mathfrak{a}^G_M/\tilde{\mathcal{L}}_M)=1$, où $\tilde{\mathcal{L}}_M := (\tilde{\mathfrak{a}}_{M,F}+\mathfrak{a}_G)/\mathfrak{a}_G$. Les fonctions
$$ v_P(\Lambda, x) := e^{\angles{-H_P(x_2)+H_{\bar{P}}(x_1), \Lambda}}, \qquad P \in \mathcal{P}(M) $$
forment une $(G,M)$-famille, d'où est défini le terme $v_M(x) = v_M(x_1,x_2)$. C'est le volume de l'enveloppe convexe de $\{-H_P(x_2) + H_{\bar{P}}(x_1) : P \in \mathcal{P}(M)\}$ dans $\mathfrak{a}^G_M$, donc est une fonction sur $(M(F) \backslash G(F) /K)^2$.

Soit $\gamma \in M_{G-\text{reg}}(F)$, posons\index[iFT2]{$J_{\tilde{M}}(\gamma, \cdot)$}
\begin{gather}
  J_{\tilde{M}}(\gamma, f) = |D(\gamma)| \iota_M^2 \int_{(A_M(F)^\dagger \backslash G(F))^2} f_1(x_1^{-1} \tilde{\gamma} x_1) f_2(x_2^{-1} \tilde{\gamma} x_2) v_M(x_1, x_2) \dd x_1 \dd x_2 .
\end{gather}

Ce seront les ingrédients dans le côté géométrique de la formule des traces locale.

Un élément $\tilde{\gamma} \in \tilde{M}$ est dit bon si son commutant est l'image réciproque du commutant de $\gamma := \rev(\tilde{\gamma}) \in M(F)$. Cela ne dépend que de la classe de conjugaison de $\gamma$. On définit ainsi le sous-ensemble $\Gamma(M(F))^\text{bon} \subset \Gamma(M(F))$, etc. Observons que $J_{\tilde{M}}(\gamma,f)=0$ sauf si $\gamma \in \Gamma(M(F))^\text{bon}$.

\paragraph{Le côté spectral}
On note $\Pi_{\text{disc},-}(\tilde{G})$ la réunion des $\Pi_\sigma(\tilde{G})$ où $M \in \mathcal{L}(M_0)$, $\sigma \in \Pi_{2,-}(\tilde{M})^t$ pour un $t \in W^G(M)_\text{reg}$. Idem pour tout Lévi de $G$. Soit $\pi = \pi_1^\vee \boxtimes \pi_2$ tel que la $X(\tilde{M})$-orbite de $\pi_i$ rencontre de $\Pi_{\text{disc},-}(\tilde{M})$, pour $i=1,2$. Soient $P,Q \in \mathcal{P}(M)$, on définit
\begin{align*}
  \mathcal{I}_{\tilde{P}}(\pi, f) & := \mathcal{I}_{\tilde{P}}(\pi_1^\vee, f_1) \boxtimes \mathcal{I}_{\tilde{P}}(\pi_2, f_2), \\
  J_{\tilde{Q}|\tilde{P}}(\pi) & := J_{\widetilde{\bar{Q}}|\widetilde{P}}(\pi_1^\vee) \boxtimes J_{\tilde{Q}|\tilde{P}}(\pi_2), \\
  \pi_\Lambda & := (\pi_{1,\Lambda})^\vee \boxtimes \pi_{2, \Lambda}, \quad \Lambda \in \mathfrak{a}_{M,\C}^*, \\
  \mathcal{J}_{\tilde{Q}}(\Lambda, \pi, \tilde{P}) & := J_{\tilde{Q}|\tilde{P}}(\pi)^{-1} J_{\tilde{Q}|\tilde{P}}(\pi_\Lambda).
\end{align*}
En variant $Q$, on montre que les $\mathcal{J}_{\tilde{Q}}(\Lambda, \pi, \tilde{P})$ forment une $(G,M)$-famille à valeurs dans les endomorphismes de $\mathcal{I}_{\tilde{P}}(\pi_1^\vee) \boxtimes \mathcal{I}_{\tilde{P}}(\pi_2)$, méromorphes en $\pi$. D'où les opérateurs $\mathcal{J}_{\tilde{M}}(\pi, \tilde{P})$ méromorphes en $\pi$.

Désormais, nous supposons que $\pi = \pi_1^\vee \boxtimes \pi_2$ avec $\pi_1, \pi_2 \in \Pi_{\text{disc},-}(\tilde{M})$.

\begin{lemma}
  Les coefficients matriciels de $\mathcal{J}_{\tilde{M}}(\pi, \tilde{P})$ sont analytiques. Leurs dérivées sont à croissance modérée pour de telles $\pi_1, \pi_2$.
\end{lemma}
\begin{proof}
  On reprend \cite[Lemma 12.1]{Ar91}.
\end{proof}

Soient $\pi=\pi_1^\vee \boxtimes \pi_2$ comme ci-dessus et $f=f_1 f_2$ avec $f_1 \in \mathcal{H}_-(\tilde{G})$, $f_2 \in \mathcal{H}_{\asp}(\tilde{G})$, on pose\index[iFT2]{$J_{\tilde{M}}(\pi, \cdot)$}
\begin{gather}
  J_{\tilde{M}}(\pi, f) := \Tr(\mathcal{J}_{\tilde{M}}(\pi, \tilde{P}) \mathcal{I}_{\tilde{P}}(\pi, f)) .
\end{gather}
Nous démontrerons (la Proposition \ref{prop:caractere-pondere-indep-P}) que $J_{\tilde{M}}(\pi, \cdot)$ ne dépend pas du choix $P \in \mathcal{P}(M)$, ce qui justifie la notation.

Si $\tau=(M_1,\sigma,r) \in T_{\text{disc},-}(\tilde{M})$, alors pour tout $\rho \in \Pi(\tilde{R}^M_\sigma, \chi_\sigma)$, la représentation $\pi_\rho$ de $\tilde{M}$ appartient à $\Pi_{\text{disc},-}(\tilde{M})$, donc c'est loisible de poser
\begin{gather}\label{eqn:caractere-pondere-T}
  J_{\tilde{M}}(\tau, f) := \sum_{\rho_1, \rho_2 \in \Pi(\tilde{R}^M_\sigma, \chi_\sigma)} \Tr\rho_1(r) \cdot \Tr\rho_2(r) \cdot J_{\tilde{M}}(\pi_{\rho_1}^\vee \boxtimes \pi_{\rho_2}, f).
\end{gather}

\begin{theorem}\label{prop:formule-traces-locale}\index[iFT2]{$J_\text{geom}, J_\text{spec}$}
  Soit $f=f_1 f_2$ avec $f_1 \in \mathcal{H}_-(\tilde{G})$, $f_2 \in \mathcal{H}_{\asp}(\tilde{G})$. Posons
  \begin{align*}
    J_\mathrm{geom}(f) & := \sum_{M \in \mathcal{L}(M_0)} |W^M_0| |W^G_0|^{-1} (-1)^{\dim A_M/A_G} \int_{\Gamma_{\mathrm{ell}}(M(F))^{\mathrm{bon}}} J_{\tilde{M}}(\gamma, f) \dd\gamma, \\
    J_\mathrm{spec}(f) & := \sum_{M \in \mathcal{L}(M_0)} |W^M_0| |W^G_0|^{-1} (-1)^{\dim A_M/A_G} \int_{T_{\mathrm{disc},-}(\tilde{M})} i^{\tilde{M}}(\tau) J_{\tilde{M}}(\tau, f) \dd\tau.
  \end{align*}
  Alors on a $J_\mathrm{geom}(f) = J_\mathrm{spec}(f)$.
\end{theorem}

On observe aussi que le terme correspondant à $M=G$ dans $J_\text{spec}(f)$ est exactement $I_\text{disc}(f)$. Désormais, on note $J(f) := J_\mathrm{geom}(f) = J_\mathrm{spec}(f)$.

\begin{proof}
  On itère l'argument pour \cite[Proposition 4.1]{Ar91} avec récurrence sur $\dim G$. Donnons-en une esquisse très brève. On part des deux développements de $\tilde{J}(f)$ fournis par la Proposition \ref{prop:J-geo-tilde} et la Proposition \ref{prop:J-spec-tilde}. Observons d'abord que les termes $\tilde{J}_{\tilde{M}}(\gamma, f)$ du côté géométrique sont similaires à $J_{\tilde{M}}(\gamma, f)$ sauf que la fonction poids est $\tilde{v}_M(x_1, x_2)$ définie dans \eqref{eqn:v_M-tilde} au lieu de $v_M(x_1,x_2)$ lorsque $F$ est non archimédien. Pour passer de l'un à l'autre, on introduit les fonctions
  $$ c_Q(\Lambda) := |\mathcal{L}_{M_Q}/\mathcal{L}_{M_Q,k}|^{-1} \sum_{X \in \mathcal{L}_{M_Q}/\mathcal{L}_{M_Q,k}} e^{\angles{\Lambda, X_P}} \theta_{\bar{Q},k}(\Lambda)^{-1} \theta_{\bar{Q}}(\Lambda) $$
  pour $Q = M_Q U_Q \in \mathcal{F}(M_0)$, où $k \in \Z_{\geq 1}$ est suffisamment divisible.

  Avec les notations dans \cite[\S 4]{Li10a}, on montre que
  $$ \tilde{v}_M(x_1,x_2) = (-1)^{\dim A_M/A_G} \sum_{Q \in \mathcal{F}(M)} v^Q_M(x_1, x_2) c'_Q. $$
  Comme dans \cite{Ar91}, cela entraîne que
  $$ \tilde{J}(f) = \sum_{Q \in \mathcal{F}(M_0)} |W^{M_Q}_0| |W^G_0|^{-1} (-1)^{\dim A_M/A_G} J^{\widetilde{M_Q}}_\text{geom}(f_{\tilde{Q}}) \cdot c'_Q. $$

  Lorsque $F$ est archimédien, nous utilisons la convention dans la Proposition \ref{prop:J-geo-tilde-arch} de sorte que les résultats précédents demeurent valables. En fait, on aura $c_Q(\Lambda)=1$ dans ce cas-là.

  Les fonctions $c_Q$ interviennent aussi dans le côté spectral via l'expression \eqref{eqn:J-tilde-L-t}: on a
  $$ \tilde{J}_{\tilde{L}}(\sigma,t,f) = (-1)^{\dim A_M/A_L} \lim_{\zeta \to 0} \sum_{R \in \mathcal{P}(L)} \tau_{1,\widetilde{\bar{R}}}(\zeta) \tau_{2,\tilde{R}}(\zeta) c_R(\zeta) \theta_R(\zeta)^{-1}. $$
  pour tout $L \in \mathcal{L}(M_0)$. Un raisonnement comme dans \cite[pp.92-94]{Ar91} donne
  $$ \tilde{J}(f) = \sum_{Q \in \mathcal{F}(M_0)} |W^{M_Q}_0| |W^G_0|^{-1} (-1)^{\dim A_M/A_G} J^{\widetilde{M_Q}}_\text{spec}(f_{\tilde{Q}}) \cdot c'_Q. $$

  Comme $c'_G \neq 0$, l'hypothèse de récurrence entraîne que $J_\text{geom}(f) = J_\text{spec}(f)$.
\end{proof}

\subsection{Théorème de Paley-Wiener invariant tempéré}
\paragraph{Transformation de Fourier invariante}
Soit $f \in \mathcal{C}(\tilde{G})$, on note $f_{\tilde{G}}$ la fonction sur $\Pi_{\text{temp}}(\tilde{G})$ donnée par\index[iFT2]{$f_{\tilde{G}}$}
$$ f_{\tilde{G}}(\pi) = \Theta_\pi(f). $$

On note l'image de $f \mapsto f_{\tilde{G}}$ par $I\mathcal{C}(\tilde{G})$\index[iFT2]{$I\mathcal{C}(\tilde{G})$} (étymologie: $I$ = l'adjectif ``invariant''), qui est un sous-espace de l'espace vectoriel de fonctions sur $\Pi_\text{temp}(\tilde{G})$. En se restreignant à $\mathcal{C}_{\asp}(\tilde{G})$ (resp. $\mathcal{C}_-(\tilde{G})$), on obtient l'espace $I\mathcal{C}_{\asp}(\tilde{G})$ (resp. $I\mathcal{C}_-(\tilde{G})$) formé de certaines fonctions sur $\Pi_{\text{temp},-}(\tilde{G})$ (resp. $\Pi_{\text{temp},\asp}(\tilde{G})$). Le théorème de Paley-Wiener invariant tempéré donnera une description de ces espaces comme des espaces vectoriels topologiques pour lesquels $f \mapsto f_{\tilde{G}}$ est ouverte et continue. Les résultats ci-dessous concernant $\mathcal{C}(\tilde{G})$ et $I\mathcal{C}(\tilde{G})$ seront valables si l'on rajoute l'indice $\asp$ (resp. $-$), pour des raisons évidentes.

Expliquons brièvement la méthode de notre démonstration. On factorisera $f \mapsto f_{\tilde{G}}$ en $\mathcal{F}_{\tilde{G}}: \mathcal{C}(\tilde{G}) \to \widehat{\mathcal{C}}(\tilde{G})$ (la ``transformation de Fourier non invariante'') et $\Tr : \widehat{\mathcal{C}}(\tilde{G}) \to I\mathcal{C}(\tilde{G})$. On sait caractériser l'image de $\mathcal{F}_{\tilde{G}}$ (Théorème \ref{prop:PW-op}). Ensuite, on construira une section de $\Tr$ à l'aide de la théorie de $R$-groupes. Il faut éviter l'usage dans \cite{Ar94-PW} de ``multiplicité un'' des $\tilde{K}$-types minimaux dans le cas archimédien; pour cela nous aurons besoin d'une partition de l'unité quelque peu technique (Corollaire \ref{prop:comb-M}).

Soient $M \in \mathcal{L}(M_0)$, $P=MU \in \mathcal{P}(M)$. Comme dans le cas de fonctions $C_c^\infty$, on définit la descente parabolique
\begin{align*}
  \mathcal{C}(\tilde{G}) & \longrightarrow \mathcal{C}(\tilde{M}) \\
  f & \longmapsto f_{\tilde{P}}(\tilde{m}) = \delta_P(m)^{\frac{1}{2}} \int_K \int_{U(F)} f(k^{-1} \tilde{m}u k) \dd u \dd k, \quad \tilde{m} \in \tilde{M} .
\end{align*}
C'est une application linéaire continue bien définie: le cas archimédien est \cite[Lemma 22]{HC66}, pour le cas non archimédien, on utilise la majoration \cite[Proposition II.4.5]{Wa03}. 

On étend la définition de $f_{\tilde{G}}$ par linéarité à des combinaisons linéaires formelles des éléments de $\Pi_\text{temp}(\tilde{G})$. On montre que pour tout $\sigma \in \Pi_\text{temp}(\tilde{M})$, on a
$$ (f_{\tilde{P}})_{\tilde{M}}(\sigma) = f_{\tilde{G}}(\sigma^G) $$
où $\sigma^G$ désigne la somme directe des constituants irréductibles de $\mathcal{I}_{\tilde{P}}(\sigma)$. Il en résulte que $f \mapsto f_{\tilde{P}}$ induit une application linéaire
\begin{align*}
  I\mathcal{C}(\tilde{G}) & \longrightarrow I\mathcal{C}(\tilde{M}), \\
  \varphi & \longmapsto \varphi_{\tilde{M}}
\end{align*}
qui ne dépend que de $\tilde{G}$ et $\tilde{M}$.

\begin{definition}\label{prop:cuspidal}
  On dit que $\varphi \in I\mathcal{C}(\tilde{G})$ est cuspidal si pour tout $M \in \mathcal{L}(M_0)$, on a $\varphi_{\tilde{M}}=0$ si $M \neq G$. On dit que $f \in \mathcal{C}(\tilde{G})$ est cuspidal si son image dans $I\mathcal{C}(\tilde{G})$ l'est.
\end{definition}
Cela étend la définition d'Arthur de cuspidalité pour les fonctions $C_c^\infty$.

\begin{remark}\label{rem:changement-parametres}
  On peut aussi regarder $I\mathcal{C}(\tilde{G})$ comme un espace de fonctions sur $T(\tilde{G})$ grâce au Théorème \ref{prop:R}. Ainsi, on regarde $f_{\tilde{G}}$ comme une fonction sur $T(\tilde{G})$ pour tout $f \in \mathcal{C}(\tilde{G})$.
\end{remark}

\paragraph{Des espaces de Fréchet}
Soient $M, M' \in \mathcal{L}(M_0)$, on note
$$ W(M'|M) := \{w \in W^G_0 : wM = M' \}. $$

Soient $(\sigma, V) \in \Pi(\tilde{M})$, $\tilde{w} \in \tilde{K}$ un représentant de $w \in W(M|M')$ et $P \in \mathcal{P}(M)$, $P' \in \mathcal{P}(M')$. On pose
$$ R_{\tilde{P}'|\tilde{P}}(\tilde{w},\sigma) := A(\tilde{w}) R_{w^{-1}\tilde{P}'|\tilde{P}}(\sigma): \mathcal{I}_{\tilde{P}}(V) \to \mathcal{I}_{\tilde{P}'}(\tilde{w}V). $$

On utilisera l'égalité suivante à plusieurs reprises. Soient $P'' \in \mathcal{P}(M'')$, $P' \in \mathcal{P}(M')$, $P \in \mathcal{P}(M)$, $w_1 \in W(M'|M)$, $w_2 \in W(M''|M')$ et $(\sigma, V) \in \Pi(\tilde{M})$. Alors les propriétés dans la Définition \ref{def:normalisant} entraînent que
\begin{gather}\label{eqn:R-transitivite}
  R_{\tilde{P}''|\tilde{P}}(\tilde{w}_2 \tilde{w}_1, \sigma) = R_{\tilde{P}''|\tilde{P}'}(\tilde{w}_2, \tilde{w}_1 \sigma) R_{\tilde{P}'|\tilde{P}}(\tilde{w}_1, \sigma)
\end{gather}
où les opérateurs $A(\cdot)$ sont ceux défini dans la Proposition \ref{prop:prop-entrelacement}, et les données définissant le $R$-groupe de $\tilde{w}_1 \sigma$ sont obtenues par transport de structure.

L'application $f \mapsto f_{\tilde{G}}$ se factorise comme\index[iFT2]{$\widehat{\mathcal{C}}(\tilde{G})$}
$$\xymatrix{
  \mathcal{C}(\tilde{G}) \ar@{->>}[r]^{\mathcal{F}_{\tilde{G}}} & \widehat{\mathcal{C}}(\tilde{G}) \ar@{->>}[r]^{\Tr} & I\mathcal{C}(\tilde{G}) \\
  f \ar@{|->}[r] & [\pi \mapsto f(\pi)] \ar@{|->}[r] & [\pi \mapsto \Theta_\pi(f)]
}$$
où $\pi \in \Pi_\text{temp}(\tilde{G})$. Le résultat suivant décrivant $\widehat{\mathcal{C}}(\tilde{G})$ est essentiellement une partie de la formule de Plancherel lorsque $F$ est non archimédien; c'est dû à Arthur \cite{Ar75} pour $F$ archimédien (voir aussi \cite[\S 4]{Ar70}). 

\begin{theorem}\label{prop:PW-op}
  L'espace $\widehat{\mathcal{C}}(\tilde{G})$ est formé des opérateurs $\Phi: (P,\sigma) \mapsto \Phi_{\tilde{P}}(\sigma) \in \End_\C(\mathcal{I}_{\tilde{P}}(V))$, où $P=MU \in \mathcal{F}(M_0)$, $(\sigma,V) \in \Pi_2(\tilde{M})$, vérifiant les conditions suivantes.

  \begin{description}
    \item[\textbf{Lissité}] Pour $P, M, \sigma$ comme ci-dessus, $\lambda \mapsto \Phi_{\tilde{P}}(\sigma_\lambda)$ est lisse, où $\lambda \in i\mathfrak{a}_M^*$.
    \item[\textbf{Symétrie}] Soient $M, M' \in \mathcal{L}(M_0)$, $P \in \mathcal{P}(M)$, $P' \in \mathcal{P}(M')$, $w \in W(M'|M)$ et $\tilde{w} \in \tilde{K}$ un représentant de $w$. Alors
      $$ \Phi_{\tilde{P}'}(\tilde{w}\sigma) = R_{\tilde{P}'|\tilde{P}}(\tilde{w}, \sigma) \Phi_{\tilde{P}}(\sigma) R_{\tilde{P}'|\tilde{P}}(\tilde{w}, \sigma)^{-1}. $$
    \item[\textbf{Croissance (non archimédien)}] Supposons $F$ non archimédien, alors $\Phi$ est à support compact.
    \item[\textbf{Croissance (archimédien)}] Supposons $F$ archimédien. Soient $n, m_1, m_2 \in \Z_{>0}$, $M \in \mathcal{L}(M_0)$, $\sigma \in \Pi_2(\tilde{M})$ et $D$ un opérateur différentiel invariant sur $i\mathfrak{a}_M^*$, alors
      $$ \sup_{P,\sigma, \delta_1, \delta_2} \| D (\Gamma_{\delta_2} \Phi_{\tilde{P}}(\sigma) \Gamma_{\delta_1}) \| (1+\|\mu_{\sigma}\|)^n (1+\|\mu_{\delta_1}\|)^{m_1} (1+\|\mu_{\delta_2}\|)^{m_2} < +\infty $$
      où
      \begin{itemize}
        \item $D(S(\sigma)) := \frac{\dd}{\dd \lambda}|_{\lambda=0} D(S(\sigma_\lambda))$ pour $S(\sigma)$ une famille lisse d'opérateurs paramétrée par $\sigma \in \Pi_2(\tilde{M})$;
        \item $\delta_1, \delta_2$ parcourent les $\tilde{K}$-types et $\Gamma_{\delta_1}$, $\Gamma_{\delta_2}$ sont les projecteurs correspondants;
        \item on définit la sous-algèbre de Cartan $\mathfrak{h}$, le caractère infinitesimal $\mu_{\sigma}$ et la norme hermitienne $\|\cdot\|$ comme dans la Définition \ref{def:normalisant} (\textbf{R8}), et $\mu_{\delta_1}, \mu_{\delta_2} \in \mathfrak{h}^*_\C$ sont les plus hauts poids de $\delta_1, \delta_2$ respectivement.
      \end{itemize}
  \end{description}

  Dans le cas non archimédien $\widehat{\mathcal{C}}(\tilde{G})$ est l'espace $C^\infty(\Theta)^{\mathrm{inv}}$ défini dans \S\ref{sec:Plancherel-enonce}, donc est muni d'une topologie. Dans le cas archimédien, les expressions dans la dernière condition définissent des semi-normes pour $\widehat{\mathcal{C}}(\tilde{G})$. En tout cas $\widehat{\mathcal{C}}(\tilde{G})$ est un espace de Fréchet pour lequel $\mathcal{F}_{\tilde{G}}$ est un isomorphisme topologique.
\end{theorem}

Cette description est indépendante de facteurs normalisants. En effet, la condition de symétrie peut aussi s'exprimer en termes des fonctions ${}^\circ c_{\tilde{P}'|\tilde{P}}(w, \sigma)$ dans \S\ref{sec:Plancherel-enonce}.

\begin{definition}\index[iFT2]{$\text{PW}(\tilde{G})$}
  Posons $\text{PW}(\tilde{G})$ l'espace des fonctions
  $$ \varphi: \tilde{T}(\tilde{G}) = \bigsqcup_{L \in \mathcal{L}(M_0)} \tilde{T}_\text{ell}(\tilde{L}) \to \C $$
  vérifiant les conditions suivantes.

  \begin{description}
    \item[\textbf{Lissité}] La fonction $\varphi$ est lisse.
    \item[\textbf{Équivariance}] Pour tout $\tau = (M,\sigma,r) \in \tilde{T}(\tilde{G})$ et $z \in Z_\sigma$, on a $\varphi(z\tau)=\chi_\sigma(z)^{-1} \varphi(\tau)$.
    \item[\textbf{Symétrie}] $\varphi$ se factorise par $T(\tilde{G})$.
    \item[\textbf{Croissance (non archimédien)}] $\varphi$ est à support compact.
    \item[\textbf{Croissance (archimédien)}] Soient $L \in \mathcal{L}(M_0)$, $n \in \Z_{>0}$, $D$ un opérateur différentiel invariant sur $i\mathfrak{a}_L^*$, alors
      $$ \|\varphi\|_{L,D,n} = \sup_{\tau=(M,\sigma,r) \in \tilde{T}_\text{ell}(\tilde{L})} |D \varphi(\tau)| (1+\|\mu_\sigma\|)^n < +\infty $$
    avec les conventions dans le Théorème \ref{prop:PW-op}.
  \end{description}

  Ainsi, l'espace $\text{PW}(\tilde{G})$ devient un espace de Fréchet.
\end{definition}

Dorénavant, on adopte le point de vue de la Remarque \ref{rem:changement-parametres} afin de comparer $I\mathcal{C}(\tilde{G})$ et $\text{PW}(\tilde{G})$. En contemplant le Théorème \ref{prop:R}, on voit que l'application $\Tr: \widehat{\mathcal{C}}(\tilde{G}) \to I\mathcal{C}(\tilde{G})$ devient
$$ T_{\tilde{G}}: \Phi \longmapsto \left[ \varphi: (M,\sigma,r) \mapsto \Tr(\tilde{R}_{\tilde{P}}(r,\sigma) \Phi_{\tilde{P}}(\sigma)) \right] $$
où $P \in \mathcal{P}(M)$ est arbitraire.

\begin{proposition}\label{prop:PW-provisoire}
  On a $I\mathcal{C}(\tilde{G}) \subset \mathrm{PW}(\tilde{G})$. L'application $f \mapsto f_{\tilde{G}}$ se factorise par $\mathrm{PW}(\tilde{G})$ et induit une application linéaire continue $\mathcal{C}(\tilde{G}) \to \mathrm{PW}(\tilde{G})$, ou bien $\widehat{\mathcal{C}}(\tilde{G}) \to \mathrm{PW}(\tilde{G})$ d'après le Théorème \ref{prop:PW-op}.
\end{proposition}
\begin{proof}
  Standard modulo le Théorème \ref{prop:PW-op}.
\end{proof}

Cette notation est provisoire: nous allons bientôt démontrer que $\text{PW}(\tilde{G})=I\mathcal{C}(\tilde{G})$

\paragraph{Combinatoire}
Nous allons montrer une variante de la partition de l'unité, qui interviendra dans la preuve du théorème de Paley-Wiener.

On note $\mathfrak{a}_0 := \mathfrak{a}_{M_0}$. On fixe une norme euclidienne $W^G_0$-invariante sur $\mathfrak{a}_0$, la distance associée est notée $d(\cdot,\cdot)$; on note aussi $\|\cdot\| := d(\cdot, 0)$. Soient $L \in \mathcal{L}(M_0)$, $\lambda \in \mathfrak{a}_0$, on écrit $\lambda = \lambda^L + \lambda_L$ selon la décomposition orthogonale $\mathfrak{a}_0 = \mathfrak{a}^L_0 \oplus \mathfrak{a}_L$. Pour tout $Q \in \mathcal{P}(L)$, on a la chambre positive associée $\mathfrak{a}_Q^+ \subset \mathfrak{a}_L$ qui vérifie $\overline{\mathfrak{a}_Q^+} = \bigsqcup_{Q' \supset Q} \mathfrak{a}_{Q'}^+$ (elle est de la forme $\mathfrak{a}_Q^+ = \mathfrak{a}^{G,+}_Q \times \mathfrak{a}_G$, et on pose $\mathfrak{a}_G^{G,+} := \{0\}$). Il y a une décomposition en facettes
$$ \mathfrak{a}_0 = \bigsqcup_{Q \in \mathcal{F}(M_0)} \mathfrak{a}_Q^+ . $$

\begin{lemma}\label{prop:comb}
  Supposons fixé un voisinage ouvert $\mathcal{V}^L$ de $0$ dans $\mathfrak{a}^L_0$ pour tout $L \in \mathcal{L}(M_0)$. Alors il existe une famille de fonctions $(\beta^Q)_{Q \in \mathcal{F}(M_0)}$ telle que
  \begin{enumerate}\renewcommand{\labelenumi}{(\arabic{enumi})}
    \item soit $Q=LU \in \mathcal{F}(M_0)$, alors $\beta^Q(\lambda) = \beta^Q(\lambda^G)$ pour tout $\lambda \in \mathfrak{a}_0$ et $\beta^Q|_{\mathfrak{a}^G_0} \in C_c^\infty(\mathfrak{a}^G_0)$, toute dérivée de $\beta^Q$ est bornée;
    \item $\beta^Q(\lambda) \neq 0$ entraîne que $\lambda^L \in \mathcal{V}^L$;
    \item si $\beta^Q(\lambda) \neq 0$, alors $\lambda \in \bigsqcup_{Q' \subset Q} \mathfrak{a}_{Q'}^+$;
    \item $\sum_{Q \in \mathcal{F}(M_0)} \beta^Q = 1$.
  \end{enumerate}\renewcommand{\labelenumi}{\arabic{enumi}}
\end{lemma}
\begin{proof}
  Pour simplifier la vie, on suppose $G$ semi-simple dans cette démonstration. Pour $Q=LU \in \mathcal{F}(M_0)$, on pose $N(Q) := \dim \mathfrak{a}_L$. On prouve par récurrence sur $N$ que l'on peut trouver des $\beta^Q$ pour $N(Q) \leq N$ vérifiant (1), (2), (3) et vérifiant: il existe $\epsilon_N > 0$ tel que $\sum_{N(Q) \leq N} \beta^Q(\lambda) = 1$ pour tout $\lambda$ tel que $d(\lambda, \bigcup_{N(Q) \leq N} \mathfrak{a}_Q^+) < \epsilon_N$. On suppose $N$ fixé et les fonctions $(\beta^Q)_{N(Q) < N}$ construites.

  On considère un paramètre auxiliaire $\epsilon > 0$ et on suppose que pour tout $L \in \mathcal{L}(M_0)$ avec $\dim \mathfrak{a}_L = N$, on a $\{\lambda^L \in \mathfrak{a}^L_0 : \|\lambda^L\| \leq \epsilon \} \subset \mathcal{V}^L$. Soit $Q = LU \in \mathcal{F}(M_0)$ avec $N(Q)=N$. On choisit $\alpha^L \in C_c^\infty(\mathfrak{a}^L_0)$ telle que
  $$ \alpha^L(\lambda^L) = \begin{cases}1, & \text{si } \|\lambda^L\| < \epsilon/2, \\ 0, & \text{si } \|\lambda^L\| > \epsilon . \end{cases} $$

  On définit la fonction sur $\mathfrak{a}_0$
  $$ \gamma^Q(\lambda) = \alpha^L(\lambda^L) \left( 1 - \sum_{N(Q')<N} \beta^{Q'}(\lambda) \right). $$

  On suppose $\epsilon < \epsilon_{N-1}$. Montrons que
  \begin{trivlist}
  \item[(i) \quad] $d(\lambda, \mathfrak{a}_Q^+) \in ]\epsilon, \epsilon_{N-1}[ \;  \Longrightarrow \; \gamma^Q(\lambda)=0.$

  En effet, soit $\lambda = \lambda^L + \lambda_L$ tel que $d(\lambda, \mathfrak{a}_Q^+) \in ]\epsilon, \epsilon_{N-1}[$. Si $\lambda_L \in \mathfrak{a}_Q^+$, alors $d(\lambda, \mathfrak{a}_Q^+) \leq \|\lambda^L\|$ donc $\|\lambda^L\| > \epsilon$ et $\alpha^L(\lambda^L)=0$. Si $\lambda_L \notin \mathfrak{a}_Q^+$, on fixe $\mu \in \mathfrak{a}_Q^+$ tel que $d(\lambda, \mu) < \epsilon_{N-1}$. Le segment $[\lambda_L, \mu]$ coupe $\partial \overline{\mathfrak{a}_Q^+}$ en un point $\mu'$ et on a $d(\lambda, \mu') \leq d(\lambda,\mu)$. Donc $d(\lambda, \bigcup_{N(Q') < N} \mathfrak{a}_{Q'}^+) < \epsilon_{N-1}$, d'où $1 - \sum_{N(Q') < N} \beta^{Q'}(\lambda) = 0$ d'après l'hypothèse de récurrence. Cela prouve (i).

  On suppose $2\epsilon < \epsilon_{N-1}$. On définit $\beta^Q$ par
  $$ \beta^Q(\lambda) = \begin{cases} \gamma^Q(\lambda), & \text{si } d(\lambda, \mathfrak{a}_Q^+) < 2\epsilon, \\ 0, & \text{sinon.} \end{cases} $$

  Vu (i), $\beta^Q$ vérifie (1). Par définition de $\alpha^L$, (2) est aussi vérifié. Montrons que
  \item[(ii) \quad] pour tout $\epsilon' > 0$, il existe $\epsilon'' > 0$ tels que les conditions
  \begin{gather*}
    d(\lambda, \mathfrak{a}_Q^+) < \epsilon'', \\
    \|\lambda^L\| < \epsilon'', \\
    d\left(\lambda, \bigcup_{N(Q')<N} \mathfrak{a}_{Q'}^+\right) > \epsilon'
  \end{gather*}
  entraînent $\lambda \in \bigcup_{Q' \subset Q} \mathfrak{a}_{Q'}^+$.

  On peut supposer $\lambda \in \mathfrak{a}^G_0$ et on écrit $\lambda = \lambda^L + \sum_{\alpha \in \Delta_Q} x_\alpha \varpi_\alpha$, où $\{\varpi_\alpha : \alpha \in \Delta_Q\} \subset \mathfrak{a}^G_L$ est la base duale de $\Delta_Q$. Soit $\alpha \in \Delta_Q$, l'élément $\mu := \sum_{\alpha' \neq \alpha} x_{\alpha'} \varpi_{\alpha'}$ appartient à $\bigcup_{N(Q') < N} \mathfrak{a}_{Q'}^+$ et $d(\lambda, \mu) = \sqrt{\|\lambda^L\|^2 + x_\alpha^2 \|\varpi_\alpha\|^2}$, d'où
  $$ \sqrt{\epsilon''^2 + x_\alpha^2 \|\varpi_\alpha\|^2} > d(\lambda, \mu) \geq d\left(\lambda, \bigcup_{N(Q') < N} \alpha_{Q'}^+\right) > \epsilon' . $$

  En supposant $\epsilon'' < \epsilon'$, cela entraîne que $|x_\alpha| > \|\varpi_\alpha\|^{-1} \sqrt{\epsilon'^2 - \epsilon''^2}$. Si $x_\alpha < 0$, on a $d(\lambda, \mathfrak{a}_Q^+) > c|x_\alpha|$ où $c$ est une constante positive, d'où $\epsilon'' > c \|\varpi_\alpha\|^{-1} \sqrt{\epsilon'^2 - \epsilon''^2}$. C'est impossible si $\epsilon''$ est assez petit. Donc on a $x_\alpha > \|\varpi_\alpha\|^{-1} \sqrt{\epsilon'^2 - \epsilon''^2}$. On note $\Sigma_0 := \Sigma_{M_0}$ l'ensemble des racines de $A_0$. Il existe $c' > 0$ tel que $|\angles{\alpha, \mu}| < c' \|\mu\|$ pour tout $\mu$ et tout $\alpha \in \Sigma_0$. On note $\Sigma^U_0$ le sous-ensemble de $\Sigma_0$ des racines dans $U$. Pour $\alpha \in \Sigma^U_0$, sa projection sur $\mathfrak{a}_L^*$ est de la forme $\sum_{\alpha' \in \Delta_Q} n_{\alpha'} \alpha'$ avec $n_{\alpha'} \in \Z_{\geq 0}$ et $n_{\alpha'} \geq 1$ pour au moins un $\alpha'$. Alors
  $$ \angles{\alpha, \lambda} = \angles{\alpha, \lambda^L} + \sum_{\alpha' \in \Delta_Q} n_{\alpha'} x_{\alpha'} > -c' \epsilon'' + \|\varpi_\alpha\|^{-1} \sqrt{\epsilon'^2 - \epsilon''^2} . $$
  C'est positif si $\epsilon''$ est suffisamment petit, donc $\angles{\alpha, \lambda} > 0$. Soit $Q' \in \mathcal{F}(M_0)$ tel que $\lambda \in \mathfrak{a}_{Q'}^+$, alors $\{\alpha \in \Sigma_0 : \angles{\alpha, \lambda} > 0\} = \Sigma^{U'}_0$. Donc $\Sigma^U_0 \subset \Sigma^{U'}_0$, ce qui entraîne $Q' \subset Q$. D'où (ii).

  Prouvons que $\beta^Q$ vérifie (3). Si $\beta^Q(\lambda)=0$, alors
  \begin{gather*}
    d(\lambda, \mathfrak{a}_Q^+) < 2\epsilon, \\
    \|\lambda^L\| < \epsilon, \\
    d\left(\lambda, \bigcup_{N(Q')<N} \mathfrak{a}_{Q'}^+\right) > \epsilon_{N-1}
  \end{gather*}
  par notre construction de $\beta^Q$. L'assertion (ii) avec $\epsilon' = \epsilon_{N-1}$ fournit une constante $\epsilon''$. Si l'on suppose que $2\epsilon < \epsilon''$, alors la condition (3) est satisfaite.

  Il reste à établir l'hypothèse de récurrence, ce qui impliquera (4) si elle est satisfaite pour tout $N$:
  \item[(iii) \quad] il existe $\epsilon_N$ tel que $d(\lambda, \bigcup_{N(Q) \leq N} \mathfrak{a}_Q^+) < \epsilon_N$ entraîne $\sum_{N(Q) \leq N} \beta^Q(\lambda)=1$.

  Fixons $\lambda \in \mathfrak{a}_0$ tel que $d(\lambda, \bigcup_{N(Q) \leq N} \mathfrak{a}_Q^+) < \epsilon_N$. Si $\sum_{N(Q') < N} \beta^{Q'}(\lambda)=1$, on a par construction $\beta^Q(\lambda)=0$ pour tout $Q$ avec $N(Q)=N$. Donc $\sum_{N(Q) \leq N} \beta^Q(\lambda)=1$. On peut donc supposer que $\sum_{N(Q') < N} \beta^{Q'}(\lambda) \neq 1$, d'où $d(\lambda, \bigcup_{N(Q') < N} \mathfrak{a}_{Q'}^+) > \epsilon_{N-1}$ par l'hypothèse de récurrence. Montrons d'abord que
  \item[(iv)\quad] il existe au plus un $Q \in \mathcal{F}(M_0)$ avec $N(Q)=N$ et $\beta^Q(\lambda) \neq 0$.

  En effet, supposons qu'il en existe deux, disons $Q_1, Q_2$ avec $Q_i = L_i U_i$, $i=1,2$. Comme ci-dessus, on a $d(\lambda, \mathfrak{a}_{Q_1}^+) < 2\epsilon$, $\|\lambda^{L_1}\| < \epsilon$, $d(\lambda, \bigcup_{N(Q')<N} \mathfrak{a}_{Q'}^+) > \epsilon_{N-1}$. D'autre part on a $d(\lambda, \mathfrak{a}_{Q_2}^+) < 2\epsilon$. Soit $\mu \in \mathfrak{a}_{Q_2}^+$ tel que $d(\lambda,\mu) < 2\epsilon$, alors on a aussi $d(\lambda^{L_1}, \mu^{L_1}) < 2\epsilon$. D'où
  \begin{gather*}
    d(\mu, \mathfrak{a}_{Q_1}^+) \leq d(\lambda, \mu) + d(\lambda, \mathfrak{a}_{Q_1}^+) < 4\epsilon, \\
    \|\mu^{L_1}\| \leq d(\mu^{L_1}, \lambda^{L_1}) + \|\lambda^{L_1}\| < 3\epsilon, \\
    d\left(\mu, \bigcup_{N(Q') < N} \mathfrak{a}_{Q'}^+\right) \geq d\left(\lambda, \bigcup_{N(Q') < N} \mathfrak{a}_{Q'}^+\right) - d(\lambda,\mu) > \epsilon_{N-1} - 2\epsilon.
  \end{gather*}

  De (ii) avec $\epsilon' = \epsilon_{N-1}/2$ se déduit $\epsilon''$. On suppose $\epsilon_{N-1} - 2\epsilon > \frac{1}{2} \epsilon_{N-1}$ et $4\epsilon < \epsilon''$. Alors $\mu \in \bigcup_{Q' \subset Q_1} \mathfrak{a}_{Q'}^+$. Cela contredit l'hypothèse que $\mu \in \mathfrak{a}_{Q_2}^+$. Cela prouve (iv).

  Montrons que
  \item[(v) \quad] pour $\epsilon_N$ suffisamment petit, il existe un $Q \in \mathcal{F}(M_0)$ tel que $N(Q)=N$ et $\beta^Q(\lambda) = 1 - \sum_{N(Q') < N} \beta^{Q'}(\lambda) \neq 0$.

  On suppose $\epsilon_N < \epsilon_{N-1}$. Les conditions
  \begin{gather*}
    d\left(\lambda, \bigcup_{N(Q') \leq N} \mathfrak{a}_Q^+\right) < \epsilon_N, \\
    d\left(\lambda, \bigcup_{N(Q') < N} \mathfrak{a}_{Q'}^+\right) > \epsilon_{N-1}
  \end{gather*}
  entraînent qu'il existe $Q \in \mathcal{F}(M_0)$ avec $N(Q)=N$ et $d(\lambda, \mathfrak{a}_Q^+) < \epsilon_N$. En particulier, $\|\lambda^L\| \leq \epsilon_N$. En prenant $\epsilon_N < \epsilon/2$, on a $\alpha^L(\lambda^L)=1$ et $\beta^Q(\lambda) = \gamma^Q(\lambda) = 1 - \sum_{N(Q') < N} \beta^{Q'}(\lambda) \neq 0$. D'où (v).
  \end{trivlist}

  D'après (iv) et (v), $\sum_{N(Q) \leq N} \beta^Q(\lambda)$ est la somme de $\sum_{N(Q') \leq N} \beta^{Q'}(\lambda)$ et d'un $\beta^Q(\lambda)$ pour un unique $Q \in \mathcal{F}(M_0)$ avec $N(Q)=N$, pour lequel $\beta^Q(\lambda)= 1 - \sum_{N(Q') < N} \beta^{Q'}(\lambda)$. Cela entraîne (iii) et achève la preuve. 
\end{proof}

Donnons-en une généralisation simple.

\begin{corollary}\label{prop:comb-M}
  Soit $M \in \mathcal{L}(M_0)$.  Supposons fixé un voisinage ouvert $\mathcal{V}^L$ de $0$ dans $\mathfrak{a}^L_M$ pour tout $L \in \mathcal{L}(M)$. Alors il existe une famille de fonctions $(\beta^Q)_{Q \in \mathcal{F}(M)}$ telle que
  \begin{enumerate}\renewcommand{\labelenumi}{(\arabic{enumi})}
    \item soit $Q=LU \in \mathcal{F}(M)$, alors $\beta^Q(\lambda) = \beta^Q(\lambda^G)$ pour tout $\lambda \in \mathfrak{a}_M$ et $\beta^Q|_{\mathfrak{a}^G_M} \in C_c^\infty(\mathfrak{a}^G_M)$, toute dérivée de $\beta^Q$ est bornée;
    \item $\beta^Q(\lambda) \neq 0$ entraîne que $\lambda^L \in \mathcal{V}^L$;
    \item si $\beta^Q(\lambda) \neq 0$, alors $\lambda \in \bigsqcup_{M \subset Q' \subset Q} \mathfrak{a}_{Q'}^+$;
    \item $\sum_{Q \in \mathcal{F}(M)} \beta^Q = 1$.
  \end{enumerate}\renewcommand{\labelenumi}{\arabic{enumi}}
\end{corollary}
\begin{proof}
  On a la décomposition $\mathfrak{a}_0 = \mathfrak{a}^M_0 \oplus \mathfrak{a}_M$. On choisit un voisinage ouvert $\mathcal{U}^M$ de $0$ dans $\mathfrak{a}^M_0$. Soit $L \in \mathcal{L}(M_0)$. Si $L \notin \mathcal{L}(M)$, on choisit un voisinage quelconque $\bar{\mathcal{V}}^L$ de $0$ dans $\mathfrak{a}^L_0$; si $L \in \mathcal{L}(M)$, alors $\mathfrak{a}^L_0 = \mathfrak{a}^M_0 \oplus \mathfrak{a}^L_M$ et on prend le voisinage $\bar{\mathcal{V}}^L = \mathcal{U}^M \times \mathcal{V}^L$ de $0$ dans $\mathfrak{a}^L_0$. Alors le Lemme \ref{prop:comb} fournit des fonctions $\bar{\beta}^Q$ adaptées aux voisinages $\bar{\mathcal{V}}^Q$. On note $\beta^Q := \bar{\beta}^Q|_{\mathfrak{a}_M}$ pour tout $Q \in \mathcal{F}(M)$, alors les propriétés (1) et (2) sont automatiquement satisfaites.

  Supposons que $\lambda \in \mathfrak{a}_M$, $Q \in \mathcal{F}(M_0)$ et $\bar{\beta}^Q(\lambda) \neq 0$, alors il existe $Q' \in \mathcal{F}(M_0)$, $Q' \subset Q$ tel que $\lambda \in \mathfrak{a}_{Q'}^+$. Or la décomposition $\mathfrak{a}_M = \bigsqcup_{Q'' \in \mathcal{F}(M)} \mathfrak{a}_{Q''}^+$ entraîne que $Q' \in \mathcal{F}(M)$, donc $Q \in \mathcal{F}(M)$. Les propriétés (3) et (4) en résultent.
\end{proof}

\begin{remark}\label{rem:comb}
  Le Corollaire \ref{prop:comb-M} est aussi valable pour la décomposition dans l'espace dual
  $$ \mathfrak{a}_0^* = \bigsqcup_{Q \in \mathcal{F}(M_0)} \mathfrak{a}_Q^{*,+} $$
  ou pour
  $$ i\mathfrak{a}_0^* = \bigsqcup_{Q \in \mathcal{F}(M_0)} i\mathfrak{a}_Q^{*,+}. $$
\end{remark}

\paragraph{Démonstration du Théorème de Paley-Wiener invariant}
Soient $M \in \mathcal{L}(M_0)$ et $\sigma \in \Pi_2(\tilde{M})$. Pour $P \in \mathcal{P}(M)$, $w \in W(M|M)$ et $\tilde{w} \in \tilde{K}$ un représentant, on note
$$ R_{\tilde{P}}(\tilde{w}, \sigma) := R_{\tilde{P}|\tilde{P}}(\tilde{w}, \sigma). $$

Le résultat suivant jouera un rôle crucial dans notre preuve.

\begin{lemma}\label{prop:R-indep-lambda}
  Soient $L \in \mathcal{L}(M)$, $P \in \mathcal{P}(M)$ et $Q \in \mathcal{P}(L)$ tels que $Q \supset P$, alors pour tout $\lambda \in i\mathfrak{a}_L^*$ et tout $w \in W^L_0 \cap W(M|M)$ avec représentant $\tilde{w} \in \tilde{K} \cap \tilde{L}$, on a
  $$ R_{\tilde{P}}(\tilde{w}, \sigma_\lambda) = R_{\tilde{P}}(\tilde{w}, \sigma). $$

  Par conséquent, on a $R^L_\sigma = R^L_{\sigma_\lambda}$.
\end{lemma}
\begin{proof}
  D'après (\textbf{R5}) dans la Définition \ref{def:normalisant}, on a
  $$ R_{\tilde{P}}(\tilde{w}, \sigma_\lambda) = A(\tilde{w}) \circ \mathcal{I}_{\tilde{Q}}(R^{\tilde{L}}_{w^{-1}\tilde{P} \cap \tilde{L}|\tilde{P} \cap \tilde{L}}(\sigma_\lambda)). $$
  Cela est indépendant de $\lambda \in i\mathfrak{a}_L^*$ d'après (\textbf{R6}). L'assertion concernant $R^L_\sigma$ en découle d'après la définition du $R$-groupe.
\end{proof}

\begin{theorem}\label{prop:PW-invariant}
  On a $\mathrm{PW}(\tilde{G})=I\mathcal{C}(\tilde{G})$. L'application linéaire continue surjective $\mathcal{C}(\tilde{G}) \to I\mathcal{C}(\tilde{G})$ admet une section continue; en particulier, c'est une application ouverte.
\end{theorem}
\begin{proof}
  Vu la Proposition \ref{prop:PW-provisoire}, on se ramène à montrer que $T_{\tilde{G}}: \widehat{\mathcal{C}}(\tilde{G}) \to \mathrm{PW}(\tilde{G})$ admet une section continue $\varphi \mapsto \Phi$. La preuve du cas non archimédien est identique à celle de \cite{Ar94-PW}. Supposons donc $F=\R$ dans ce qui suit. La preuve suivante est une variante de celle d'Arthur qui évite l'usage de multiplicité $1$ de $\tilde{K}$-types minimaux.

  Soient $M$, $\sigma$ comme ci-dessus. Comme $F=\R$, on peut écrire de façon unique $\sigma = (\sigma^1)_{\lambda(\sigma)}$ où $\sigma^1$ se factorise par $\tilde{M}^1$ et $\lambda(\sigma) \in i\mathfrak{a}_M^*$. Soit $k \in \tilde{K}$ un représentant d'un élément dans $W^G_0$. La condition $k\sigma \simeq \sigma$ équivaut à $k\sigma^1 \simeq \sigma^1$ et $k\lambda(\sigma)=\lambda(\sigma)$. Si $\lambda(\sigma) \in i\mathfrak{a}_Q^{*,+}$ avec $Q = LU \in \mathcal{F}(M)$, la condition $k\lambda(\sigma)=\lambda(\sigma)$ équivaut à $k \in \tilde{L}$. Donc $W_\sigma = W^L_{\sigma^1}$ et, pour $P$, $Q$, $r, \tilde{r}$ comme dans le Lemme \ref{prop:R-indep-lambda} (on remplace le symbole $w$ par $r$ pour souligner le rôle du $R$-groupe), on a $R_{\tilde{P}}(\tilde{r},\sigma)=R_{\tilde{P}}(\tilde{r}, \sigma^1)$ et $R_\sigma = R^L_{\sigma^1}$.

  Dorénavant, on conserve la notation $\sigma$ pour un élément de $\Pi_2(\tilde{M})$ tel que $\lambda(\sigma)=0$, autrement dit $\sigma$ se factorise par $\tilde{M}^1$. On écrit tout élément de $\Pi_2(\tilde{M})$ sous la forme $\sigma_\lambda$ avec $\lambda \in i\mathfrak{a}_M^*$.

  On fixe un ensemble fini de $\tilde{K}$-types, noté $K(M,\sigma)$, tel que
  \begin{itemize}
    \item chaque élément de $K(M,\sigma)$ est un $\tilde{K}$-type minimal (voir \cite[Chapter X]{KV95})\index[iFT2]{$\tilde{K}$-type minimal} d'un $\pi_\rho$, où $\rho \in \Pi(\tilde{R}_\sigma, \chi_\sigma)$;
    \item pour tout $\rho \in \Pi(\tilde{R}_\sigma, \chi_\sigma)$, $\pi_\rho$ contient au moins un $\tilde{K}$-type minimal appartenant à $K(M,\sigma)$;
    \item pour tout $\tilde{w} \in N_{\tilde{K}}(\widetilde{M_0})$, $K(wM, \tilde{w}\sigma)=K(M,\sigma)$.
  \end{itemize}

  Alors il existe des constantes $A, B > 0$, qui ne dépendent que de $\tilde{G}$, $\tilde{M}$ et de la norme $\|\cdot\|$ sur $\mathfrak{h}_\C$, telles que pour tout $\delta \in K(M,\sigma)$, on a
  \begin{gather}\label{eqn:K-type-majoration}
    \|\mu_\delta \| \leq A \|\mu_\sigma\| + B.
  \end{gather}
  En effet, le cas $\tilde{G}$ connexe est \cite[Theorem 10.26]{KV95}; le cas général en résulte sans peine, cf. la remarque dans \cite[p.642]{KV95}.

  On fixe $f_{M,\sigma} \in C^\infty(\tilde{K})$ qui agit comme la projection sur les $\tilde{K}$-types appartenant à $K(M,\sigma)$. Soient $P \in \mathcal{P}(M)$, $Q=LU \in \mathcal{F}(M)$ avec $Q \supset P$. Le Théorème \ref{prop:R} donne
  $$ \mathcal{I}_{\tilde{P}}(\sigma) = \bigoplus_{\rho \in \Pi(\tilde{R}_\sigma, \chi_\sigma)} \rho^\vee \boxtimes \pi_\rho = \bigoplus_{\rho^L \in \Pi(\tilde{R}^L_\sigma, \chi_\sigma)} E_P(\sigma, \rho^L) $$
  où $E_P(\sigma, \rho^L) = (\rho^L)^\vee \boxtimes \pi(\sigma, \rho^L)$ avec
  \begin{align*}
    \pi(\sigma, \rho^L) & = \bigoplus_{\rho \in \Pi(\tilde{R}_\sigma, \chi_\sigma)} \text{mult}((\rho^L)^\vee, \rho^\vee) \pi_\rho, \\
    \text{mult}((\rho^L)^\vee, \rho^\vee) & := \dim_\C \Hom_{\tilde{R}^L_\sigma}((\rho^L)^\vee, \rho^\vee).
  \end{align*}
  Autrement dit, on regroupe les termes selon les facteurs irréductibles de $\rho|_{\tilde{R}^L_\sigma}$.

  Soit $r \in \tilde{R}^L_\sigma$. On définit l'opérateur de $\mathcal{I}_{\tilde{P}}(\sigma)$:
  \begin{align*}
    D^Q_P(r, \sigma) & := |R^L_\sigma|^{-1} \sum_{\rho^L \in \Pi(\tilde{R}^L_\sigma, \chi_\sigma)} \deg(\rho^L) (\dim \pi(\sigma, \rho^L)[K(M,\sigma)])^{-1} \cdot \\
    & \cdot (\tilde{R}_{\tilde{P}}(r, \sigma) \mathcal{I}_{\tilde{P}}(\sigma, f_{M,\sigma}))|_{E_P(\sigma, \rho^L)}
  \end{align*}
  où $\pi(\sigma, \rho^L)[K(M,\sigma)]$ signifie le sous-espace de $\pi(\sigma, \rho^L)$ se transformant par des $\tilde{K}$-types dans $K(M,\sigma)$. Montrons que pour $r, r' \in \tilde{R}^L_\sigma$, on a
  \begin{gather}\label{eqn:D^Q_P}
    \Tr \left(\tilde{R}_{\tilde{P}}(r', \sigma) D^Q_P(r^{-1}, \sigma)\right) =
    \begin{cases}
      \chi_\sigma(z), & \text{si } r=r'z, z \in Z_\sigma, \\
      0, & \text{sinon}.
    \end{cases}
  \end{gather}
  En effet, $\Tr(\tilde{R}_{\tilde{P}}(r', \sigma) D^Q_P(r^{-1}, \sigma))$ est égal à $|R^L_\sigma|^{-1} \sum_{\rho^L} \deg(\rho^L) \Tr(\rho^L)^\vee(r'r^{-1})$ d'après notre choix de $f_{M,\sigma}$.

  Pour $\lambda \in i\mathfrak{a}_M^*$, on définit l'opérateur
  $$ E^Q_P(\sigma, \lambda) := \sum_{w \in W^L_0/W^M_0} \sum_{\substack{P' \in \mathcal{P}(wM) \\ P' \subset Q}} R_{\tilde{P}'|\tilde{P}}(\tilde{w}, \sigma_\lambda)^{-1} R_{\tilde{P}'|\tilde{P}}(\tilde{w}, \sigma) $$
  où $\tilde{w} \in \tilde{K} \cap \tilde{L}$  est un représentant quelconque de $w$. Cela ne dépend que de $\lambda^L$ d'après le Lemme \ref{prop:R-indep-lambda} et est une homothétie non nulle pour $\lambda^L=0$. Donc $E^Q_P(\sigma, \lambda)$ est inversible et lisse pour $\lambda^L$ dans un voisinage de $0$ dans $i\mathfrak{a}_M^{L,*}$.

  Soient $w \in W^L_0$ avec représentant $\tilde{w} \in \tilde{K} \cap \tilde{L}$, et $P' \in \mathcal{P}(w M)$ tel que $P' \subset  Q$. Montrons que
  \begin{gather}\label{eqn:E^Q_P}
    R_{\tilde{P}'|\tilde{P}}(\tilde{w}, \sigma_\lambda) E^Q_P(\sigma, \lambda) = E^Q_{P'}(\tilde{w} \sigma, w\lambda) R_{\tilde{P}'|\tilde{P}}(\tilde{w}, \sigma), \quad \lambda \in i\mathfrak{a}_M^*.
  \end{gather}
  En effet,
  \begin{align*}
    R_{\tilde{P}'|\tilde{P}}(\tilde{w}, \sigma_\lambda) E^Q_P(\sigma, \lambda) R_{\tilde{P}'|\tilde{P}}(\tilde{w}, \sigma)^{-1} & = \sum_{w_1, P_1} R_{\tilde{P}'|\tilde{P}}(\tilde{w}, \sigma_\lambda) R_{\tilde{P}_1|\tilde{P}}(\tilde{w}_1, \sigma_\lambda)^{-1} R_{\tilde{P}_1|\tilde{P}}(\tilde{w}_1, \sigma) R_{\tilde{P}'|\tilde{P}}(\tilde{w}, \sigma)^{-1} \\
    & = \sum_{w_1, P_1} R_{\tilde{P}_1|\tilde{P}'}(\tilde{w}_1 \tilde{w}^{-1}, \tilde{w} \sigma_{w\lambda})^{-1} R_{\tilde{P}_1|\tilde{P}'}(\tilde{w}_1 \tilde{w}^{-1}, \tilde{w} \sigma_\lambda)
  \end{align*}
  où on utilise la propriété $R_{\tilde{P}_1|\tilde{P}'}(\tilde{w}_1 \tilde{w}^{-1}, \tilde{w} \sigma_{w\lambda}) R_{\tilde{P}'|\tilde{P}}(\tilde{w}, \sigma_\lambda) = R_{\tilde{P}_1|\tilde{P}}(\tilde{w}_1, \sigma_\lambda)$, qui résulte de \eqref{eqn:R-transitivite}. On en déduit \eqref{eqn:E^Q_P} en remplaçant $\tilde{w}_1$ par $\tilde{w}_1 \tilde{w}$ et $M$ par $wM$.

  Pour tout $\lambda \in i\mathfrak{a}_M^*$ avec $\lambda^L$ proche de $0$, on pose
  \begin{gather}
    D^Q_P(r, \sigma, \lambda) := E^Q_P(\sigma, \lambda) D^Q_P(r,\sigma) E^Q_P(\sigma, \lambda)^{-1}, \quad r \in \tilde{R}^L_\sigma.
  \end{gather}

  Pour $P \in \mathcal{P}(M)$, $Q = LU \in \mathcal{F}(M)$, $r \in \tilde{R}^L_\sigma$, posons
  \begin{multline*}
    S^Q_P(r, \sigma, \lambda) = |W^G_0|^{-1} \sum_{w \in W^G_0} |\{P' \in \mathcal{P}(wM) : P' \subset wQ\}|^{-1}  \\
    \sum_{\substack{P' \in \mathcal{P}(wM) \\ P' \subset wQ}} R_{\tilde{P}'|\tilde{P}}(\tilde{w}, \sigma_\lambda)^{-1} D^{wQ}_{P'}(wrw^{-1}, \tilde{w}\sigma, w\lambda) R_{\tilde{P}'|\tilde{P}}(\tilde{w},\sigma_\lambda)
  \end{multline*}
  où $\tilde{w} \in \tilde{K}$ est un représentant de $w$. Cet opérateur est bien défini et lisse en $\lambda$ pour $\lambda^L$ proche de $0$. Pour $w \in W^G_0$, $\tilde{w} \in \tilde{K}$ un représentant et $P' \in \mathcal{P}(wM)$, l'argument pour \eqref{eqn:E^Q_P} donne
  \begin{gather}\label{eqn:S^Q_P-symetrie}
    S^{wQ}_{P'}(wrw^{-1}, \tilde{w}\sigma, w\lambda) = R_{\tilde{P}'|\tilde{P}}(\tilde{w}, \sigma_\lambda) S^Q_P(r,\sigma,\lambda)  R_{\tilde{P}'|\tilde{P}}(\tilde{w}, \sigma_\lambda)^{-1}.
  \end{gather}

  On choisit des fonctions $\beta^Q$ sur $i\mathfrak{a}_M^*$ fournies par le Lemme \ref{prop:comb-M} (et aussi par la Remarque \ref{rem:comb}) adaptées aux voisinages où les opérateurs $S^Q_P(r,\sigma, \lambda)$ sont bien définis. On peut aussi supposer que $\beta^{wQ}(w\lambda)=\beta^Q(\lambda)$ pour tout $Q$ et tout $w \in W^G_0$.

  Soit $(M,\sigma_\lambda, r) \mapsto \varphi(M,\sigma, \lambda,r) = \varphi(M,\sigma_\lambda,r)$ une fonction dans $\text{PW}(\tilde{G})$. Pour $r \in \tilde{R}_\sigma$, on note $L(r)$ le Lévi contenant $M$ tel que $r \in \tilde{R}^{L(r)}_{\sigma,\text{reg}}$. Étant donnés $M$, $\sigma_\lambda$ et $P \in \mathcal{P}(M)$, on définit l'opérateur
  $$ \Phi_{\tilde{P}}(\sigma_\lambda) = \sum_{Q=LU \in \mathcal{F}(M)} \beta^Q(\lambda) \sum_{r \in R^L_\sigma} S^Q_P(r^{-1}, \sigma, \lambda) \varphi(M, \sigma, \lambda_{L(r)}, r). $$

  En fait, il faut prendre un relèvement de $r$ dans $\tilde{R}^L_\sigma$ dans la somme. Vu l'équivariance de $S^Q_P$ et $\varphi$ sous $Z_\sigma$, cela n'affecte pas la définition.

  Montrons que $\Phi := (\Phi_{\tilde{P}})_{P \in \mathcal{F}(M_0)}$ appartient à $\mathcal{C}(\tilde{G})$. La lissité en $\lambda$ est évidente. Pour $w \in W^G_0$ avec représentant $\tilde{w} \in \tilde{K}$ et $P' \in \mathcal{P}(wM)$, on vérifie
  $$ \Phi_{\tilde{P}'}(\tilde{w}\sigma_{w\lambda}) = R_{\tilde{P}'|\tilde{P}}(\tilde{w},\sigma_\lambda) \Phi_{\tilde{P}}(\sigma_\lambda)  R_{\tilde{P}'|\tilde{P}}(\tilde{w},\sigma_\lambda)^{-1} $$
  à l'aide de \eqref{eqn:S^Q_P-symetrie}. Cela permet de vérifier la condition de symétrie car $\beta^{wQ}(w\lambda)=\beta^Q(\lambda)$ pour tout $Q$.

  La condition de croissance résulte des faits suivants.
  \begin{itemize}
    \item Pour tout $Q \in \mathcal{F}(M)$, toute dérivée de $\beta^Q$ est bornée.
    \item La condition de croissance satisfaite par $\varphi$.
    \item Les coefficients matriciels $\tilde{K}$-finis des opérateurs d'entrelacement normalisés sont à croissance modérée (cf. (\textbf{R6}) dans la Définition \ref{def:normalisant}).
    \item La majoration \eqref{eqn:K-type-majoration}.
  \end{itemize}
  Vu les définitions des semi-normes définissant les espaces de Fréchet en question, ces arguments impliquent aussi la continuité de l'application $\varphi \mapsto \Phi$.

  Il reste à montrer que $\varphi \mapsto \Phi$ est une section de $T_{\tilde{G}}$. Soient $P \in \mathcal{P}(M)$ et $r' \in \tilde{R}_{\sigma_\lambda}$ Calculons d'abord $\Tr(\tilde{R}_{\tilde{P}}(r', \sigma_\lambda) \Phi_{\tilde{P}}(\sigma_\lambda))$. L'hypothèse sur $r'$ implique $\lambda \in i\mathfrak{a}_{L(r')}^*$, donc $\lambda \in i\mathfrak{a}_{Q(\lambda)}^{*,+}$ avec $Q(\lambda)=L(\lambda)U(\lambda)$ et $L(r') \subset L(\lambda)$. Pour $Q=LU \in \mathcal{F}(M)$, si $\beta^Q(\lambda) \neq 0$ alors $Q(\lambda) \subset Q$, donc $L(r') \subset L(\lambda) \subset L$. Quitte à changer $P$, ce qui est loisible d'après la symétrie de $\Phi$, le Lemme \ref{prop:R-indep-lambda} affirme que $r' \in \tilde{R}^{L(r')}_{\sigma_\lambda} = \tilde{R}^{L(r')}_\sigma \subset \tilde{R}^L_\sigma$.

  Les égalités \eqref{eqn:D^Q_P}, \eqref{eqn:E^Q_P} et la définition de $S^Q_P(r^{-1}, \sigma, \lambda)$ entraînent que pour tout $r \in \tilde{R}^L_{\sigma}$, on a
  \begin{gather}\label{eqn:S^Q_P-tr}
    \Tr \left(\tilde{R}_{\tilde{P}}(r', \sigma_\lambda) S^Q_P(r^{-1}, \sigma, \lambda)\right) =
    \begin{cases}
       \chi_\sigma(z), & \text{si } r=r'z, z \in Z_\sigma, \\
       0, & \text{sinon}.
    \end{cases}
  \end{gather}
  Donc
  \begin{align*}
    \sum_{r \in R^L_\sigma} \Tr \left( \tilde{R}_{\tilde{P}}(r', \sigma_\lambda) S^Q_P(r^{-1}, \sigma, \lambda) \right) \varphi(M,\sigma,\lambda_{L(r')}, r) = \varphi(M, \sigma, \lambda_{L(r)}, r') = \varphi(M, \sigma, \lambda, r')
  \end{align*}
  car $\lambda \in i\mathfrak{a}_{L(r')}^*$. D'où

  $$ \Tr\left(\tilde{R}_{\tilde{P}}(r', \sigma_\lambda) \Phi_{\tilde{P}}(\sigma_\lambda)\right) = \sum_{Q \in \mathcal{F}(M)} \beta^Q(\lambda) \varphi(M, \sigma, \lambda, r') = \varphi(M,\sigma,\lambda,r'). $$

 Ce qu'il fallait démontrer.
\end{proof}

\subsection{Caractères pondérés tempérés}
\paragraph{Caractères pondérés tempérés de $\tilde{M}$}
Fixons des familles de facteurs normalisants faibles. Soient $M \in \mathcal{L}(M_0)$ et $\pi$ une représentation admissible irréductible de $\tilde{M}$ dont la $X(\tilde{M})$-orbite contient une représentation unitaire spécifique. Soit $P \in \mathcal{P}(M)$, introduisons les $(G,M)$-familles suivantes des opérateurs méromorphes en $\pi$ (sur la $X(\tilde{M})$-orbite en question)
\begin{align*}
  \mathcal{J}_{\tilde{Q}}(\Lambda, \pi, \tilde{P}) & := J_{\tilde{Q}|\tilde{P}}(\pi)^{-1} J_{\tilde{Q}|\tilde{P}}(\pi_\Lambda), \\
  \mathcal{R}_{\tilde{Q}}(\Lambda, \pi, \tilde{P}) & := R_{\tilde{Q}|\tilde{P}}(\pi)^{-1} R_{\tilde{Q}|\tilde{P}}(\pi_\Lambda), \quad Q \in \mathcal{P}(M), \; \Lambda \in \mathfrak{a}_M^* .
\end{align*}

Alors $\mathcal{R}_{\tilde{M}}(\pi, \tilde{P})$ est régulier si $\pi$ est unitaire, d'après (\textbf{R2}). Définissons le caractère pondéré
$$ J_{\tilde{M}}^r(\pi, f) := \Tr(\mathcal{R}_{\tilde{M}}(\pi, \tilde{P}) \mathcal{I}_{\tilde{P}}(\pi, f)), \quad f \in \mathcal{H}_{\asp}(\tilde{G}). $$
C'est facile de voir qu'elle ne dépend pas du choix de $P$ à l'aide de (\textbf{R1}). C'est une distribution tempérée en $f$ si $\pi$ est tempérée (voir \cite[p.175]{Ar94}).

Posons aussi $\mu_{\tilde{P}|\tilde{Q}}(\pi) := \prod_\alpha \mu_\alpha(\pi)$ où $\alpha$ parcourt $\Sigma_Q^\text{red} \cap \Sigma_{\bar{P}}^\text{red}$, et les $(G,M)$-familles
\begin{align*}
  \mu_{\tilde{Q}}(\Lambda, \pi, \tilde{P}) & := \mu_{\tilde{Q}|\tilde{P}}(\pi)^{-1} \mu_{\tilde{Q}|\tilde{P}}(\pi_{\frac{1}{2}\Lambda}), \\
  \mathcal{M}_{\tilde{Q}}(\Lambda, \pi, \tilde{P}) & := \mu_{\tilde{Q}}(\Lambda, \pi, \tilde{P}) \mathcal{J}_{\tilde{Q}}(\Lambda, \pi, \tilde{P}).
\end{align*}

Le caractère pondéré canoniquement normalisé est défini comme\index[iFT2]{$J_{\tilde{M}}(\pi, \cdot)$}
$$ J_{\tilde{M}}(\pi, f) := \Tr(\mathcal{M}_{\tilde{M}}(\pi, \tilde{P}) \mathcal{I}_{\tilde{P}}(\pi, f)), \quad f \in \mathcal{H}_{\asp}(\tilde{G}). $$

Comme dans le cas de groupes réductifs \cite{Ar98}, on utilise les propriétés des fonctions $\mu$ et des facteurs normalisants pour démontrer que
\begin{itemize}
  \item $J_{\tilde{M}}(\pi, \cdot)$ ne dépend pas du choix de $P$, cf. ci-dessous;
  \item $J_{\tilde{M}}(\pi, f)$ est régulier si $\pi$ est unitaire;
  \item définissons la $(G,M)$-famille
    $$ r_{\tilde{Q}}(\Lambda, \pi) := r_{\widetilde{\bar{Q}}|\widetilde{Q}}(\pi)^{-1} r_{\widetilde{\bar{Q}}|\widetilde{Q}}(\pi_{\frac{1}{2}\Lambda}), \quad Q \in \mathcal{P}(M), $$
    d'où les termes $r^{\tilde{R}}_{\tilde{M}}(\pi)$ définis comme dans \cite[\S 4]{Li10a}, où $R \in \mathcal{F}(M)$; on montre que $r^{\tilde{R}}_{\tilde{M}}(\pi)$ ne dépend que de $\pi$ et de la composante de Lévi de $R$, donc c'est loisible d'écrire $r^{\tilde{L}}_{\tilde{M}}(\pi)$ où $L \in \mathcal{L}(M)$;
  \item on a l'identité
    $$ J_{\tilde{M}}(\pi, f) = \sum_{L \in \mathcal{L}(M)} r^{\tilde{L}}_{\tilde{M}}(\pi) J_{\tilde{L}}^r(\pi^L, f) $$
    où $\pi$ est en position générale de sorte que $\pi^L := \mathcal{I}^{\tilde{L}}_{\tilde{M}}(\pi)$ est irréductible;
  \item $r^{\tilde{L}}_{\tilde{M}}(\pi)$ est régulier si $\pi$ est unitaire, pour tout $L \in \mathcal{L}(M)$.
\end{itemize}

Par exemple, prouvons l'indépendance de $P$ de $J_{\tilde{M}}(\pi, \cdot)$. On peut supposer $\pi$ unitaire. Vu la propriété
$$ J_{\tilde{M}}(\pi, f) = \sum_{L \in \mathcal{L}(M)} r^{\tilde{L}}_{\tilde{M}}(\pi) J_{\tilde{L}}^r(\pi^L, f),$$
on conclut en utilisant le résultat suivant, ce qui sera utile plus tard.

\begin{lemma}\label{prop:indep-P-1copie}
  Soient $P, P' \in \mathcal{P}(M)$. On a
  $$ \mathcal{R}_{\tilde{M}}(\pi, \tilde{P}) = R_{\tilde{P}'|\tilde{P}}(\pi)^{-1} \mathcal{R}_{\tilde{M}}(\pi, \tilde{P}') R_{\tilde{P}'|\tilde{P}}(\pi).$$

  Par conséquent, la distribution $J^r_{\tilde{M}}(\pi, \cdot)$ est indépendante du choix de $P$.
\end{lemma}
\begin{proof}
  Fixons $P \in \mathcal{P}(M)$. Soit $f \in \mathcal{H}_{\asp}(\tilde{G})$. Pour $P' \in \mathcal{P}(M)$, on a
  \begin{align*}
    \mathcal{R}_{\tilde{Q}}(\Lambda, \pi, \tilde{P}) &= (R_{\tilde{Q}|\tilde{P}'}(\pi) R_{\tilde{P}'|\tilde{P}}(\pi))^{-1} R_{\tilde{Q}|\tilde{P}'}(\pi_\Lambda) R_{\tilde{P}'|\tilde{P}}(\pi_\Lambda) \\
    & = R_{\tilde{P}'|\tilde{P}}(\pi)^{-1} \mathcal{R}_{\tilde{Q}}(\Lambda, \pi, \tilde{P}') R_{\tilde{P}'|\tilde{P}}(\pi_\Lambda).
  \end{align*}
  Comme
  $$ \mathcal{R}_{\tilde{M}}(\pi, \tilde{P}) = \lim_{\Lambda \to 0} \sum_{Q \in \mathcal{P}(M)} \mathcal{R}_{\tilde{Q}}(\Lambda, \pi, \tilde{P}) \theta_Q(\Lambda)^{-1} $$
  (resp. $P'$ au lieu de $P$), cela prouve que $\mathcal{R}_{\tilde{M}}(\pi, \tilde{P}) = R_{\tilde{P}'|\tilde{P}}(\pi)^{-1} \mathcal{R}_{\tilde{M}}(\pi, \tilde{P}') R_{\tilde{P}'|\tilde{P}}(\pi)$. Il en résulte que
  \begin{align*}
    \mathcal{R}_{\tilde{M}}(\pi, \tilde{P}) \mathcal{I}_{\tilde{P}}(\pi, f) & = R_{\tilde{P}'|\tilde{P}}(\pi)^{-1} \mathcal{R}_{\tilde{M}}(\pi, \tilde{P}') R_{\tilde{P}'|\tilde{P}}(\pi) \mathcal{I}_{\tilde{P}}(\pi, f) \\
    & = R_{\tilde{P}'|\tilde{P}}(\pi)^{-1} \mathcal{R}_{\tilde{M}}(\pi, \tilde{P}') \mathcal{I}_{\tilde{P}'}(\pi, f) R_{\tilde{P}'|\tilde{P}}(\pi).
  \end{align*}
  Donc $J^r_{\tilde{M}}(\pi, f) = \Tr(\mathcal{R}_{\tilde{M}}(\pi, \tilde{P}) \mathcal{I}_{\tilde{P}}(\pi, f)) = \Tr( \mathcal{R}_{\tilde{M}}(\pi, \tilde{P}') \mathcal{I}_{\tilde{P}'}(\pi, f))$.
\end{proof}

Soient $M \in \mathcal{L}(M_0)$ et $f \in \mathcal{C}_{\asp}(\tilde{G})$, définissons les fonctions sur $\Pi_{\text{temp},-}(\tilde{G})$
\begin{align*}
  \phi^r_{\tilde{M}}(f, \pi) & := J_{\tilde{M}}^r(\pi, f), \quad \pi \in \Pi_{\text{temp},-}(\tilde{G}), \\
  \phi_{\tilde{M}}(f, \pi) & := J_{\tilde{M}}(\pi, f), \quad \pi \in \Pi_{\text{temp},-}(\tilde{G}).
\end{align*}
Alors on a
\begin{gather}\label{eqn:phi^r-phi}
  \phi_{\tilde{M}}(f, \pi) = \sum_{L \in \mathcal{L}(M)} r^{\tilde{L}}_{\tilde{M}}(\pi) \phi^r_{\tilde{L}}(f, \pi^L).
\end{gather}
Ici, il se peut que $\pi^L$ soit réductible; dans ce cas-là on définit $\phi^r_{\tilde{L}}(f, \pi^L)$ comme une somme de $\phi^r_{\tilde{L}}(f, \cdot)$ évalué en les constituants irréductibles (cf. \cite[\S 3]{Ar98}). Remarquons aussi que $\phi_{\tilde{G}}(f) = f_{\tilde{G}}$. Le fait suivant est crucial.

\begin{theorem}[Cf. {\cite[p.179]{Ar94}} et {\cite[Corollary 9.2]{Ar81}}]\label{prop:phi-local}\index[iFT2]{$\phi_{\tilde{M}}$}
  L'application $\phi_{\tilde{M}}^r$ (resp. $\phi_{\tilde{M}}$) induit une application linéaire continue $\mathcal{C}_{\asp}(\tilde{G}) \to I\mathcal{C}_{\asp}(\tilde{M})$.
\end{theorem}
\begin{proof}
  Pour l'application $\phi_{\tilde{M}}^r$, l'outil essentiel dans l'argument d'Arthur est le théorème de Paley-Wiener invariant caractérisant $I\mathcal{C}(\tilde{M})$, qui est déjà établi pour les revêtements. Pour $\phi_{\tilde{M}}$, vu \eqref{eqn:phi^r-phi}, il suffit de majorer $r^{\tilde{L}}_{\tilde{M}}(\pi_\lambda)$ et ses dérivées en $\lambda$. Cela découle de (\textbf{R8}), cf. \cite[Lemma 3.1]{Ar98}.
\end{proof}

\paragraph{Les distributions dans la formule des traces locale}
Maintenant, plaçons-nous dans le formalisme de la formule des traces locale. Les distributions en question vivent sur $\tilde{G} \times \tilde{G}$. Nous fixons une famille de facteurs normalisants faibles $r^\vee$ (resp. $r$) dans la première (resp. la deuxième) copie de $\tilde{G}$, telles que $r^\vee$ et $r$ sont complémentaires (cf. la Définition \ref{def:normalisant-faible}). Néanmoins on utilise la même lettre $R$ (resp. $J_{\tilde{M}}^r$) pour désigner les opérateurs d'entrelacement normalisés (resp. les caractères pondérés normalisés) dans chaque copie, puisqu'il n'y aura aucune confusion à craindre.

Prenons $\pi = \pi_1^\vee \boxtimes \pi_2 \in \Pi_{\text{unit},\asp}(\tilde{M}) \times \Pi_{\text{unit},-}(\tilde{M})$. Rappelons que le caractère pondéré non normalisé $J_{\tilde{M}}(\pi, \cdot)$ est défini à l'aide de la $(G,M)$-famille $\{\mathcal{J}_{\tilde{Q}}(\Lambda, \pi, \tilde{P}) : Q \in \mathcal{P}(M)\}$. Définissons ses avatars en posant
\begin{align*}
  R_{\tilde{Q}|\tilde{P}}(\pi) & := R_{\widetilde{\bar{Q}}|\widetilde{P}}(\pi_1^\vee) \boxtimes R_{\tilde{Q}|\tilde{P}}(\pi_2), \\
  \mathcal{R}_{\tilde{Q}}(\Lambda, \pi, \tilde{P}) & := R_{\tilde{Q}|\tilde{P}}(\pi)^{-1} R_{\tilde{Q}|\tilde{P}}(\pi_\Lambda), \\
  J_{\tilde{M}}^r(\pi, f) & := \Tr(\mathcal{R}_{\tilde{M}}(\pi, \tilde{P}) \mathcal{I}_{\tilde{P}}(\pi, f));
\end{align*}
et
\begin{align*}
  \mu_{\tilde{Q}}(\Lambda, \pi, \tilde{P}) & := \mu_{\widetilde{\bar{Q}}}(\Lambda, \pi_1^\vee, \tilde{P}) \mu_{\widetilde{Q}}(\Lambda, \pi_2, \tilde{P}), \\
  \mathcal{M}_{\tilde{Q}}(\Lambda, \pi, \tilde{P}) & := \mu_{\tilde{Q}}(\Lambda, \pi, \tilde{P}) \mathcal{J}_{\tilde{Q}}(\Lambda, \pi, \tilde{P}), \\
  J_{\tilde{M}}^\mu(\pi, f) & := \Tr(\mathcal{M}_{\tilde{M}}(\pi, \tilde{P}) \mathcal{I}_{\tilde{P}}(\pi, f));
\end{align*}
pour $f = f_1 f_2$ avec $f_1 \in \mathcal{H}_-(\tilde{G})$, $f_2 \in \mathcal{H}_{\asp}(\tilde{G})$.

Supposons désormais qu'il existe $\Lambda \in \mathfrak{a}_{M,\C}^*$, $M_1 \in \mathcal{L}^M(M_0)$ et $\sigma \in \Pi_{2,-}(\widetilde{M_1})$ avec $\pi_{1,\Lambda}, \pi_{2,\Lambda} \in \Pi_\sigma(\tilde{M})$. Cela inclut les représentations qui interviennent dans la formule des traces locale.

\begin{lemma}\label{prop:R=J=M}
  On a $\mathcal{R}_{\tilde{M}}(\pi, \tilde{P}) = \mathcal{J}_{\tilde{M}}(\pi, \tilde{P}) = \mathcal{M}_{\tilde{M}}(\pi, \tilde{P})$, d'où $J_{\tilde{M}}^r(\pi, \cdot) = J_{\tilde{M}}(\pi, \cdot) = J_{\tilde{M}}^\mu(\pi, \cdot)$.
\end{lemma}
\begin{proof}
  Montrons la première égalité. Considérons d'abord le cas $\pi_1, \pi_2 \in \Pi_\sigma(\tilde{M})$, alors
  \begin{align*}
    r_{\tilde{\bar{Q}}|\tilde{P}}^\vee(\pi_1^\vee) & = r_{\tilde{\bar{Q}}|\tilde{P}}^\vee(\sigma^{M,\vee}) & \\
    & = r_{\tilde{P}|\tilde{\bar{Q}}}^\vee(\sigma^M) & \text{par (\textbf{R2})} \\
    & = r_{\tilde{\bar{Q}}|\tilde{P}}(\sigma^M) & \text{car $r^\vee$ et $r$ sont complémentaires} \\
    & = r_{\tilde{\bar{P}}|\tilde{Q}}(\sigma^M) & \text{par (\textbf{R4})} .
  \end{align*}
  où $\sigma^M = \mathcal{I}^{\tilde{M}}_{\tilde{M}_1}(\sigma)$. D'autre part $r_{\tilde{Q}|\tilde{P}}(\pi_2) = r_{\tilde{Q}|\tilde{P}}(\sigma^M)$. Leur produit est donc égal à $r_{\tilde{\bar{P}}|\tilde{P}}(\sigma^M)$, qui est indépendant de $Q$. Idem si $\pi$ est remplacé par $\pi_\Lambda$ et $\sigma$ est remplacé par $\sigma_\Lambda$.

  Donc $\mathcal{R}_{\tilde{M}}(\pi, \tilde{P})$ et $\mathcal{J}_{\tilde{M}}(\pi, \tilde{P})$ ne différent qu'à la constante
  $$ \lim_{\Lambda \to 0} r_{\tilde{\bar{P}}|\tilde{P}}(\sigma^M)^{-1} r_{\tilde{\bar{P}}|\tilde{P}}(\sigma^M_\Lambda) = 1 .$$

  Le même argument permet de montrer la deuxième égalité; il suffit d'observer que les propriétés des fonctions $\mu$ entraînent
  \begin{align*}
    \mu_{\tilde{\bar{Q}}|\tilde{P}}(\pi_1^\vee) & = \mu_{\tilde{\bar{Q}}|\tilde{P}}(\sigma^{M,\vee}) = \mu_{\tilde{\bar{Q}}|\tilde{P}}(\sigma^M) = \mu_{\tilde{\bar{P}}|\tilde{Q}}(\sigma^M), \\
    \mu_{\tilde{Q}|\tilde{P}}(\pi_2) & = \mu_{\tilde{Q}|\tilde{P}}(\sigma^M).
  \end{align*}
\end{proof}

\begin{lemma}\label{prop:caractere-pondere-indep-P}
  La distribution $J_{\tilde{M}}(\pi, \cdot)$ ne dépend pas de $P \in \mathcal{P}(M)$.
\end{lemma}
\begin{proof}
  Grâce au Lemme \ref{prop:R=J=M}, il suffit de montrer que $J_{\tilde{M}}^r(\pi, \cdot)$ ne dépend pas de $P$. La preuve est similaire au cas d'une seule copie de $\tilde{G}$, cf. la démonstration du Lemme \ref{prop:indep-P-1copie}.
\end{proof}

Soient $\gamma \in \Gamma_\text{ell}(M(F))^\text{bon}$, $\tilde{\gamma} \in \rev^{-1}(\gamma)$ et $g \in \mathcal{H}_{\asp}(\tilde{G})$ (resp. $g \in \mathcal{H}_{-}(\tilde{G})$). On sait définir l'intégrale orbitale pondérée $J_{\tilde{M}}(\tilde{\gamma}, g)$. Cf. \cite[\S 6.3]{Li10a}. Pour ce faire, il faut fixer une mesure de Haar sur $M_\gamma(F)$. On choisit la mesure telle que $\mes(M_\gamma(F)/A_M(F))=1$.

\begin{lemma}\label{prop:pondere-scindage}
  Avec le formalisme de \cite[\S 4.2]{Li10a}, on a
  \begin{align*}
    J_{\tilde{M}}(\pi, f) & = \sum_{L_1, L_2 \in \mathcal{L}(M)} d^G_M(L_1, L_2) J^{\tilde{L}_1}_{\tilde{M}}(\pi_1^\vee, f_{1,\widetilde{Q_1}}) J^{\tilde{L}_2}_{\tilde{M}}(\pi_2, f_{2,\widetilde{Q_2}}) \\
    & = \sum_{L_1, L_2 \in \mathcal{L}(M)} d^G_M(L_1, L_2) J^{\tilde{L}_1, r}_{\tilde{M}}(\pi_1^\vee, f_{1,\widetilde{Q_1}}) J^{\tilde{L}_2, r}_{\tilde{M}}(\pi_2, f_{2,\widetilde{Q_2}})
  \end{align*}
  où $L_i \mapsto Q_i \in \mathcal{P}(L_i)$, $i=1,2$, est associé à un choix de $\xi \in \{(H,-H) : H \in \mathfrak{a}_M \}$ en position générale.

  D'autre part, on a
  $$ J_{\tilde{M}}(\gamma, f) = \sum_{L_1, L_2 \in \mathcal{L}(M)} d^G_M(L_1, L_2) J^{\tilde{L}_1}_{\tilde{M}}(\tilde{\gamma}, f_{\widetilde{Q_1}}) J^{\tilde{L}_2}_{\tilde{M}}(\tilde{\gamma}, f_{\widetilde{Q_2}}) $$
  pour tout $\gamma \in \Gamma_{\mathrm{ell}}(M(F))^{\mathrm{bon}}$ qui est $G$-régulier, et $\tilde{\gamma} \mapsto \gamma$ quelconque.
\end{lemma}
\begin{proof}
  Pour l'assertion spectrale, les cas de $J_{\tilde{M}}(\pi, \cdot)$ et de $J^r_{\tilde{M}}(\pi, \cdot)$ sont équivalents d'après le Lemme \ref{prop:R=J=M}. On applique la formule de scindage \cite[\S 4.2]{Li10a} à la $(G,M)$-famille $(\mathcal{R}_{\tilde{Q}}(\Lambda, \pi, \tilde{P}))_{Q \in \mathcal{P}(M)}$:
  $$ \mathcal{R}_{\tilde{M}}(\pi, \tilde{P}) = \sum_{L_1, L_2 \in \mathcal{L}(M)} d^G_M(L_1,L_2) \mathcal{R}^{Q_1}_{\tilde{M}}(\pi_1^\vee, \tilde{P}) \boxtimes \mathcal{R}^{Q_2}_{\tilde{M}}(\pi_2, \tilde{P}). $$
  Maintenant on peut reprendre l'argument standard d'Arthur (cf. la démonstration de \cite[Lemma 7.1]{Ar81}) pour obtenir la formule de $J^r_{\tilde{M}}(\pi,f)$, puisque pour chaque $L_i$ dans la somme ($i=1,2$), on peut remplacer $P$ par un parabolique contenu dans $Q_i$ d'après un argument analogue à celui de la preuve du Lemme \ref{prop:indep-P-1copie}.

  L'assertion géométrique résulte de la formule de scindage appliquée à l'ensemble $(G,M)$-orthogonal définissant la fonction poids $v_M(x_1,x_2)$.
\end{proof}

On définit l'espace de Schwartz-Harish-Chandra $\mathcal{C}(\tilde{G} \times \tilde{G})$ et on note $\mathcal{C}_{\asp}(\tilde{G} \times \tilde{G})$ son sous-espace des fonctions anti-spécifiques sous l'action de $\bmu_m$ via l'immersion anti-diagonale $\noyau \mapsto (\noyau^{-1}, \noyau)$. Les fonctions test $f = f_1 f_2$ considérées jusqu'à présent appartiennent à $\mathcal{C}_{\asp}(\tilde{G} \times \tilde{G})$.\index[iFT2]{$\mathcal{C}_{\asp}(\tilde{G} \times \tilde{G})$}

\begin{proposition}\label{prop:dist-temperee}
  Les distributions $J_{\tilde{M}}(\gamma, \cdot)$ et $J_{\tilde{M}}(\pi, \cdot)$ dans la formule des traces locale (le Théorème \ref{prop:formule-traces-locale}) se prolongent de façon unique en des formes linéaires continues sur $\mathcal{C}_{\asp}(\tilde{G} \times \tilde{G})$. De plus, la formule des traces locale (Théorème \ref{prop:formule-traces-locale}) demeure valable pour les fonctions test dans $\mathcal{C}_-(\tilde{G}) \oplus \mathcal{C}_{\asp}(\tilde{G})$.
\end{proposition}
\begin{proof}
  Pareille que \cite[p.189]{Ar94}. Cependant, pour la part concernant le Théorème \ref{prop:formule-traces-locale} il convient de remarquer que la majoration d'Arthur \cite[(5.7)]{Ar94}
  $$ |J_{\tilde{M}}(\tilde{\gamma},f)| \leq \nu_r(f) (1+|\log |D(\gamma)||)^{\dim \mathfrak{a}^G_M} (1 + \|H_M(\gamma)\|)^{-n}, \quad f \in \mathcal{C}_{\asp}(\tilde{G}) \cup \mathcal{C}_-(\tilde{G}) $$
  où $n \in \Z_{>0}$ est arbitraire et $r = r(n) \in \R$ dépend de $n$, est encore valable pour des raisons triviales. En effet, $|J_{\tilde{M}}(\tilde{\gamma},f)| \leq J_M(\gamma, |f|)$; on applique alors la majoration d'Arthur à $J_M(\gamma, |f|)$ et on note que la borne cherchée ne dépend que de $|f|$ et $\gamma$. Par conséquent, on n'a pas encore besoin du théorème de finitude de Howe sur les revêtements dans le cas non archimédien.
\end{proof}

On est en mesure de définir les applications $\phi_{\tilde{M}}$ pour la formule des traces locale.\index[iFT2]{$\phi_{\tilde{M}}$}

\begin{align*}
  \phi_{\tilde{M}}: & \mathcal{C}_{\asp}(\tilde{G} \times \tilde{G}) \longrightarrow I\mathcal{C}_{\asp}(\tilde{M} \times \tilde{M}) \\
  & f \longmapsto [\pi \mapsto J_{\tilde{M}}^\mu(\pi, f)]
\end{align*}
où $\pi = \pi_1^\vee \boxtimes \pi_2 \in \Pi_{\text{temp},\asp}(\tilde{M}) \times \Pi_{\text{temp},-}(\tilde{M})$. C'est sous-entendu que $\pi \mapsto J_{\tilde{M}}^\mu(\pi, f)$ appartient à $I\mathcal{C}_{\asp}(\tilde{M} \times \tilde{M})$, un fait qui est inclus dans le résultat suivant.

\begin{theorem}
  L'application $\phi_{\tilde{M}}$ est bien définie, linéaire et continue.
\end{theorem}
\begin{proof}
  Pareil que le cas avec une seule copie de $\tilde{G}$.
\end{proof}

\begin{remark}
  Vu la théorie de $R$-groupe \S\ref{sec:R-groupe}, on peut aussi changer le paramétrage et définir les caractères pondérés $J_{\tilde{M}}(\tau,\cdot)$ ou $J_{\tilde{M}}^\mu(\tau_1^\vee \boxtimes \tau_2)$, où $\tau, \tau_1, \tau_2 \in T_-(\tilde{M})$. Les caractères pondérés ``en $\pi$'' et ``en $\tau$'' s'expriment réciproquement à l'aide du Théorème \ref{prop:R}.
\end{remark}

\subsection{La formule des traces locale invariante}

\begin{definition}\index[iFT2]{forme linéaire supportée par un espace}
  Soient $V, V'$ des espaces vectoriels topologiques et $\phi: V \to V'$ une application linéaire continue. On dit qu'une forme linéaire continue $I: V \to \C$ est supportée par $V'$ si $I|_{\Ker\phi}=0$. Dans ce cas-là, la forme linéaire continue induite $\Im(\phi) \to \C$ est notée $\hat{I}$.
\end{definition}

Dans cet article, cette notion sera appliquée aux cas suivants.
\begin{enumerate}
  \item $V = \mathcal{C}_{\asp}(\tilde{G})$, $V' = I\mathcal{C}_{\asp}(\tilde{G})$ et $\phi = \phi_{\tilde{G}}$;
  \item $V = \mathcal{C}_{\asp}(\tilde{G} \times \tilde{G})$, $V' = I\mathcal{C}_{\asp}(\tilde{G} \times \tilde{G})$ et $\phi = \phi_{\tilde{G}}$.
\end{enumerate}

En tout cas $\phi$ est surjectif. Afin d'établir la formule des traces locale invariante, raisonnons par récurrence sur $\dim G$ avec deux hypothèses qui seront établies dans le Corollaire \ref{prop:dist-support}. Dans ce qui suit, $\gamma$ signifie un élément dans $\Gamma_{G-\text{reg,ell}}(M(F))^\text{bon}$ et $\tilde{\gamma} \in \rev^{-1}(\gamma)$ est quelconque.

\begin{hypothesis}\label{hyp:support-caractere}\index[iFT2]{$I_{\tilde{M}}(\tilde{\gamma}, \cdot)$}
  On a défini les distributions spécifiques $I^{\tilde{L}}_{\tilde{M}}(\tilde{\gamma}, \cdot): \mathcal{C}_{\asp}(\tilde{L}) \to \C$. Elles sont supportées par $I\mathcal{C}_{\asp}(\tilde{L})$ pour tout $L \in \mathcal{L}(M_0)$, $L \neq G$. On définit
  $$ I_{\tilde{M}}(\tilde{\gamma}, f) := J_{\tilde{M}}(\gamma, f) - \sum_{L \in \mathcal{L}(M), L \neq G} \hat{I}^{\tilde{L}}_{\tilde{M}}(\gamma, \phi_{\tilde{L}}(f)). $$
\end{hypothesis}

\begin{hypothesis}\label{hyp:support-caractere-double}
  On a défini les distributions spécifiques $I^{\tilde{L}}_{\tilde{M}}(\gamma, \cdot): \mathcal{C}_{\asp}(\tilde{L} \times \tilde{L}) \to \C$. Elles sont supportées par $I\mathcal{C}_{\asp}(\tilde{L} \times \tilde{L})$ pour tout $L \in \mathcal{L}(M_0)$, $L \neq G$. On définit
  \begin{align*}
    I_{\tilde{M}}(\gamma, f) & := J_{\tilde{M}}(\gamma, f) - \sum_{L \in \mathcal{L}(M), L \neq G} \hat{I}^{\tilde{L}}_{\tilde{M}}(\gamma, \phi_{\tilde{L}}(f)), \\
    I(f) & := J(f) - \sum_{L \in \mathcal{L}(M_0), L \neq G} |W^L_0| |W^G_0|^{-1} (-1)^{\dim A_L/A_G} \hat{I}^{\tilde{L}}(\phi_{\tilde{L}}(f)).
  \end{align*}
\end{hypothesis}

La définition ci-dessus de $I(f)$ est loisible car on vérifie que $I(f)$ est égal à\index[iFT2]{$I, I_\text{geom}$}
\begin{gather}\label{eqn:I_geom}
  I_\text{geom}(f) := \sum_{M \in \mathcal{L}(M_0)} |W^M_0| |W^G_0|^{-1} (-1)^{\dim A_M/A_G} \int_{\Gamma_{G-\text{reg,ell}}(M(F))^\text{bon}} I_{\tilde{M}}(\gamma, f) \dd\gamma
\end{gather}
à l'aide du Théorème \ref{prop:formule-traces-locale} et de la définition de $I_{\tilde{M}}(\gamma, \cdot)$; donc $I(\cdot)$ est spécifique supportée par $I\mathcal{C}_{\asp}(\tilde{G} \times \tilde{G})$ si chaque $I_{\tilde{M}}(\tilde{\gamma}, \cdot)$ l'est, et c'est satisfait si l'on remplace $G$ par $L \in \mathcal{L}(M_0)$, $L \neq G$.

Tout d'abord, on voit que $I^{\tilde{M}}_{\tilde{M}}(\cdots) = J^{\tilde{M}}_{\tilde{M}}(\cdots)$. Cela permet de déterminer les distributions $I_{\tilde{M}}$ par récurrence modulo les deux hypothèses ci-dessus. Nous pouvons énoncer la formule des traces locale invariante maintenant. Rappelons que nous avons défini $I_\text{geom}$ et $I_\text{disc}$ dans \eqref{eqn:I_geom} et \eqref{eqn:I_disc}, respectivement.

\begin{lemma}\label{prop:I-scindage}
  Soient $M \in \mathcal{L}(M_0)$, $\gamma \in \Gamma_{G-\mathrm{reg,ell}}(M(F))^{\mathrm{bon}}$. Avec le formalisme de \cite[\S 4.2]{Li10a}, on a
  \begin{gather*}
    I_{\tilde{M}}(\gamma, f) = \sum_{L_1, L_2 \in \mathcal{L}(M)} d^G_M(L_1, L_2) I^{\tilde{L}_1}_{\tilde{M}}(\tilde{\gamma}, f_{1,\widetilde{Q_1}}) I^{\tilde{L}_2}_{\tilde{M}}(\tilde{\gamma}, f_{2,\widetilde{Q_2}}).
  \end{gather*}
  où $L_i \mapsto Q_i \in \mathcal{P}(L_i)$, $i=1,2$, est associé à un choix de $\xi \in \{(H,-H) : H \in \mathfrak{a}_M \}$ en position générale.
\end{lemma}
\begin{proof}
  Vu le Lemme \ref{prop:pondere-scindage} et la définition de $I_{\tilde{M}}(\gamma, f)$, cela résulte par récurrence.
\end{proof}

\begin{lemma}
  L'Hypothèse \ref{hyp:support-caractere} entraîne l'Hypothèse \ref{hyp:support-caractere-double}.
\end{lemma}
\begin{proof}
  Les assertions dans l'Hypothèse \ref{hyp:support-caractere} restent valables si l'on remplace l'indice $\asp$ par $-$, car on peut toujours passer de l'un à l'autre en poussant le revêtement par l'automorphisme $\noyau \mapsto \noyau^{-1}$ de $\bmu_m$. L'assertion résulte du Lemme \ref{prop:I-scindage}.
\end{proof}

\begin{theorem}[Cf. {\cite[Proposition 6.1]{Ar03}}]
  Supposons vérifiées l'Hypothèse \ref{hyp:support-caractere}. Pour tout $f =f_1 f_2 \in \mathcal{C}_{\asp}(\tilde{G} \times \tilde{G})$, on a
  $$ I(f) = I_{\mathrm{geom}}(f) = I_{\mathrm{disc}}(f). $$
\end{theorem}
\begin{proof}
  Il suffit de prouver $I_\text{disc}(f) = I(f)$. Soit $M \in \mathcal{L}(M_0)$. On définit par récurrence des distributions
  $$ I^{\tilde{L}}_{\tilde{M}}(\tau, \cdot): \mathcal{C}_{\asp}(\tilde{G} \times \tilde{G}) \to \C, \quad L \in \mathcal{L}(M), \tau \in T_-(\tilde{M}) $$
  telles que $J_{\tilde{M}}(\tau, \cdot) = \sum_{L \in \mathcal{L}(M)} \hat{I}^{\tilde{L}}_{\tilde{M}}(\tau, \phi_{\tilde{L}}(\cdot))$. L'hypothèse de récurrence est que $I^{\tilde{L}}_{\tilde{M}}(\tau, \cdot)$ est supporté par $I\mathcal{C}_{\asp}(\tilde{L} \times \tilde{L})$ pour $L \neq G$. Puisque toute représentation en vue est tempérée, pour tout $\tau \in T_-(\tilde{M})$ on a
  $$ I^{\tilde{G}}_{\tilde{M}}(\tau, f) =
    \begin{cases}
      f_{\tilde{G}}(\tau), & \text{si } M=G, \\
      0, & \text{sinon}.
    \end{cases}
  $$
  D'où l'hypothèse de récurrence. Comme $I(\cdot)$ est défini suivant la même procédure utilisant $J(\cdot) = J_\text{spec}(\cdot)$, il en résulte que $I(\cdot)$ est égal à la somme des termes avec $M=G$ dans $J_\text{spec}(\cdot)$ (voir le Théorème \ref{prop:formule-traces-locale}), ce qui est $I_\text{disc}(\cdot)$.
\end{proof}

\begin{corollary}[Cf. {\cite[Corollary 4.3]{Ar93}}]\label{prop:formule-traces-locale-simple1}
  Pour $\tau = (M, \sigma, r) \in T_{\mathrm{ell},-}(\tilde{G})$, on pose\index[iFT2]{$d(\tau)$}
  $$ d(\tau) := \det(1-r|\mathfrak{a}^G_M).$$

  Soient $f_1 \in \mathcal{C}_-(\tilde{G})$ quelconque et $f_2 \in \mathcal{C}_{\asp}(\tilde{G})$ cuspidale, alors
  \begin{multline*}
    \sum_{M \in \mathcal{L}(M_0)} |W^M_0| |W^G_0|^{-1} (-1)^{\dim A_M/A_G} \int_{\Gamma_{\mathrm{ell}, G-\mathrm{reg}}(M(F))^{\mathrm{bon}}} I_{\tilde{G}}(\tilde{\gamma}, f_1) I_{\tilde{M}}(\tilde{\gamma}, f_2) \dd\gamma  \\
    = \int_{T_{\mathrm{ell},-}(\tilde{G})} |d(\tau)|^{-1} \Theta(\tau^\vee, f_1) \Theta(\tau, f_2) \dd\tau .
  \end{multline*}
\end{corollary}
\begin{proof}
  On utilise l'identité $I_\text{geom}(f) = I_\text{spec}(f)$. Dans le côté géométrique, pour tout $M \in \mathcal{L}(M_0)$ et tout $\gamma \in \Gamma_{\mathrm{ell}, G-\mathrm{reg}}(M(F))^{\mathrm{bon}}$, on a un analogue directe du Lemme \ref{prop:pondere-scindage}:
  $$ I_{\tilde{M}}(\gamma, f) = \sum_{L_1, L_2 \in \mathcal{L}(M)} d^G_M(L_1, L_2) I^{\tilde{L}_1}_{\tilde{M}}(\tilde{\gamma}, f_{1, \widetilde{Q_1}}) I^{\tilde{L}_2}_{\tilde{M}}(\tilde{\gamma}, f_{2, \widetilde{Q_2}}). $$

  Si $L_2 \neq G$, alors $I^{\tilde{L}_2}_{\tilde{M}}(\tilde{\gamma}, \cdot)$ est supportée par $I\mathcal{C}_{\asp}(\tilde{L}_2 \times \tilde{L}_2)$, donc $I^{\tilde{L}_2}_{\tilde{M}}(\tilde{\gamma}, f_{2, \widetilde{Q_2}})=0$ par la cuspidalité de $f_2$. D'autre part,
  $$ d^G_M(L_1, G) =
    \begin{cases}
      1, & \text{si } L_1=M, \\
      0, & \text{sinon};
    \end{cases}
  $$
  et on a $I^{\tilde{L}_1}_{\tilde{M}}(\tilde{\gamma}, f_{1, \tilde{Q}_1}) = I_{\tilde{G}}(\tilde{\gamma}, f_1)$ lorsque $L_1=M$. D'où
  $$ I_\text{geom}(f) = \sum_{M \in \mathcal{L}(M_0)} |W^M_0| |W^G_0|^{-1} (-1)^{\dim A_M/A_G} \int_{\Gamma_{\mathrm{ell}, G-\mathrm{reg}}(M(F))^{\mathrm{bon}}} I_{\tilde{G}}(\tilde{\gamma}, f_1) I_{\tilde{M}}(\tilde{\gamma}, f_2) \dd\gamma .$$

  Par définition
  $$ I_\text{disc}(f) = \int_{T_{\text{disc},-}(\tilde{G})} i(\tau) \Theta(\tau^\vee, f_1) \Theta(\tau, f_2) \dd\tau. $$

  Soit $\tau = (M,\sigma,r) \in T_{\text{disc},-}(\tilde{G})$. On fixe $P \in \mathcal{P}(M)$. Supposons que $\tau \in T_{\text{ell},-}(\tilde{L})$ avec $L \in \mathcal{L}(M)$. D'après la théorie dans \S\ref{sec:R-groupe}, $\Theta(\tau, f_2) = \Tr(\tilde{R}_{\tilde{P}}(r,\sigma)\mathcal{I}_{\tilde{P}}(\sigma, f_2))$ est une combinaison linéaire de $f_{2,\tilde{L}}(\pi_L)$ où $\pi_L$ sont des éléments dans $\Pi_\sigma(\tilde{L})$. Puisque $f_2$ est cuspidale, on a $\Theta(\tau, f_2)=0$ sauf si $L=G$. Montrons que dans le cas $L=G$ on a $W_\sigma^0=\{1\}$. En effet, rappelons que $W_\sigma^0$ est associé au système de racine sur $\mathfrak{a}_M$ engendré par les racines réduites $\alpha$ avec $\mu_\alpha(\sigma)=0$. On choisit une chambre $\mathfrak{a}_\sigma^+$ de ce système et on identifie $R_\sigma$ au sous-groupe de $W_\sigma$ fixant $\mathfrak{a}_\sigma^+$. Si $W_\sigma^0 \neq \{1\}$ alors $\mathfrak{a}_\sigma^+ \neq \emptyset$, et $r$ fixe la demi somme des coracines positives pour
$\mathfrak{a}_\sigma^+$. Cela contredit l'hypothèse que $r \in R_{\sigma,\text{reg}}$.

  Donc pour $\tau \in T_{\text{ell},-}(\tilde{G})$ on a $\epsilon_\sigma(r)=1$ et $i(\tau)=|d(\tau)|^{-1}$ selon la définition de $i(\tau)$ dans \eqref{eqn:i(tau)}. Par conséquent
  $$ I_\text{disc}(f) = \int_{T_{\text{ell},-}(\tilde{G})} |d(\tau)|^{-1} \Theta(\tau^\vee, f_1) \Theta(\tau, f_2) \dd\tau. $$
\end{proof}

\begin{corollary}[Cf. {\cite[Theorem 5.1]{Ar93}}]\index[iFT2]{$\Phi_{\tilde{M}}(\tau^\vee, \tilde{\gamma})$}
  Soit $f_2 \in \mathcal{C}_{\asp}(\tilde{G})$ cuspidale. Pour tout $M \in \mathcal{L}(M_0)$ et tout $\gamma \in \Gamma_{G-\mathrm{reg}}(M(F))^{\mathrm{bon}}$, on a
  $$ I_{\tilde{M}}(\tilde{\gamma}, f_2) = (-1)^{\dim A_M/A_G} \int_{T_{\mathrm{ell},-}(\tilde{G})} |d(\tau)|^{-1} \Phi_{\tilde{M}}(\tau^\vee, \tilde{\gamma}) \Theta(\tau, f) \dd\tau $$
  où
  $$
    \Phi_{\tilde{M}}(\tau^\vee, \tilde{\gamma}) :=
    \begin{cases}
      |D(\gamma)|^{\frac{1}{2}} \Theta(\tau^\vee, \tilde{\gamma}), & \text{si } \gamma \text{ est $F$-elliptique dans $M$}, \\
      0, & \text{sinon}.
    \end{cases}
  $$
\end{corollary}
\begin{proof}
  C'est plus commode d'utiliser les symboles $M',\gamma'$ au lieu de $M,\gamma$ dans cette preuve. Soient $M' \in \mathcal{L}(M_0)$ et $\gamma' \in \Gamma_{G-\mathrm{reg}}(M'(F))^{\mathrm{bon}}$. Supposons d'abord que $\gamma'$ n'est pas $F$-elliptique dans $M'$. Alors une formule de descente (voir \cite[Corollary 8.3]{Ar88-TF1}), le Lemme \ref{prop:pondere-scindage}, l'Hypothèse \ref{hyp:support-caractere} et la cuspidalité de $f_2$ entraînent que $I_{\tilde{M}'}(\gamma', f)=0$. D'autre part $\Phi_{\tilde{M}'}(\tau^\vee, \tilde{\gamma}')=0$ pour tout $\tau$. Cela prouve l'assertion.

  Supposons que $\gamma'$ est $F$-elliptique dans $M'$. Soit $\tau \in T_{\text{ell},-}(\tilde{G})$. Puisque $\Theta(\tau^\vee, \cdot)$ est une combinaison linéaire de caractères admissibles irréductibles, il est représenté par une fonction invariante localement intégrable sur $\tilde{G}$ et lisse sur $\tilde{G}_\text{reg}$ d'après le Corollaire \ref{prop:caractere-admissible}. Soit $f_1 \in \mathcal{C}_-(\tilde{G})$. La formule d'intégrale de Weyl donne
  $$ \Theta(\tau^\vee, f_1) = \sum_{M \in \mathcal{L}(M_0)} |W^M_0| |W^G_0|^{-1} \int_{\Gamma_{\mathrm{ell}, G-\mathrm{reg}}(M(F))^{\mathrm{bon}}} \Phi_{\tilde{M}}(\tau^\vee, \tilde{\gamma}) I_{\tilde{G}}(\tilde{\gamma}, f_1) \dd\tilde{\gamma}. $$

  Pour $\gamma \in \Gamma_{\mathrm{ell}, G-\mathrm{reg}}(M(F))^{\mathrm{bon}}$ avec un relèvement $\tilde{\gamma}$ dans $\tilde{G}$, on définit la distribution spécifique
  $$ \delta(M, \tilde{\gamma}, \cdot) := I_{\tilde{M}}(\tilde{\gamma}, \cdot) - (-1)^{\dim A_M/A_G} \int_{T_{\mathrm{ell},-}(\tilde{G})} |d(\tau)|^{-1} \Phi_{\tilde{M}}(\tau^\vee, \tilde{\gamma}) \Theta(\tau, \cdot) \dd\tau . $$

  Le but est de montrer que $\delta(M', \tilde{\gamma}', f_2)=0$. D'après le Corollaire \ref{prop:formule-traces-locale-simple1} et la formule ci-dessus pour $\Theta(\tau^\vee, \cdot)$, on a
  \begin{gather}\label{eqn:delta_M}
    \sum_{M \in \mathcal{L}(M_0)} |W^M_0| |W^G_0|^{-1} (-1)^{\dim A_M/A_G} \int_{\Gamma_{\mathrm{ell}, G-\mathrm{reg}}(M(F))^{\mathrm{bon}}} I_{\tilde{G}}(\tilde{\gamma}, f_1) \delta(M, \tilde{\gamma}, f_2) \dd\gamma = 0.
  \end{gather}

  On a des isomorphismes
  $$\xymatrix{
    [M, \tilde{\gamma}] \ar@{|->}[d] & \in & \left( \bigsqcup_{M \in \mathcal{L}(M_0)}\limits \Gamma_{\text{ell},G-\text{reg}}(\tilde{M})^{\text{bon}}\right) / W^G_0 \ar[d]^{\simeq} \\
    \tilde{\gamma} \ar@{|->}[d] & \in & \Gamma_\text{reg}(\tilde{G})^{\text{bon}} \ar[d]^{\simeq} \\
    (G_\gamma, \tilde{\gamma}) & \in & \bigsqcup_{\substack{T: F-\text{tore maximal} \\ \text{mod conjugaison}}}\limits (\tilde{T} \cap \tilde{G}_{\text{reg}}) / W(G(F),T(F)).
  }$$

  Ces espaces sont des $\bmu_m$-torseurs au-dessus de leurs avatars associés à $G(F)$. On vérifie que $\delta(\cdot, \cdot, f_2)$ est bien définie comme une fonction sur le première espace et $I_{\tilde{G}}(\cdot, f_1)$ est une fonction sur le deuxième espace. Le côté à gauche de \eqref{eqn:delta_M} peut être regardé comme une intégrale sur l'un des trois espaces.

  On pose $T' := G_{\gamma'}$. On prend $f_1 \in \mathcal{C}_-(\tilde{G})$ à support dans les classes de conjugaison rencontrant $\rev^{-1}(T' \cap G_\text{reg})(F)$. Si l'on regarde $I_{\tilde{G}}(\cdot, f_2)$ comme une fonction sur le troisième espace dans ledit diagramme, alors elle est à support dans la composante indexée par $T'$. De plus, on peut prendre une suite de telles fonctions $f_1$ de sorte que $I_{\tilde{G}}(\cdot, f_1)$ tend vers la mesure de Dirac spécifique sur l'image réciproque de la $W(G(F),T'(F))$-orbite de $\gamma'$. Le côté à gauche de \eqref{eqn:delta_M} tend vers $m (-1)^{\dim A_M/A_G} \cdot |W(G(F),T'(F))| \cdot \delta(M',\tilde{\gamma}', f_2)$. On en déduit que $\delta(M', \tilde{\gamma}', f_2)=0$, ce qu'il fallait démontrer.
\end{proof}

\begin{corollary}\label{prop:dist-support}
  Les distributions $I_{\tilde{M}}(\tilde{\gamma}, \cdot)$ sont supportées par $I\mathcal{C}_{\asp}(\tilde{M})$.
\end{corollary}
\begin{proof}
  Soit $f_2 \in \mathcal{C}_{\asp}(\tilde{G})$ telle que $f_{2,\tilde{G}}=0$, alors $f_2$ est cuspidale. Le corollaire précédent permet de conclure que $I_{\tilde{M}}(\tilde{\gamma}, f_2)=0$ car $\Theta(\tau, \cdot)$ est à support à $I\mathcal{C}_{\asp}(\tilde{M})$.
\end{proof}

Ce résultat justifie les Hypothèses \ref{hyp:support-caractere}, \ref{hyp:support-caractere-double}. On établit ainsi la formule des traces locale invariante.

Donnons une application importante de ce corollaire.
\begin{theorem}[Cf. {\cite[Theorem 0]{Ka86}}]\label{prop:KZ0}\index[iFT2]{Théorème 0 de Kazhdan}
  Supposons $F$ non archimédien. Soit $f \in C_{c,\asp}^\infty(\tilde{G})$, les conditions suivantes sont équivalentes.
  \begin{enumerate}\renewcommand{\labelenumi}{(\alph{enumi})}
    \item $f$ appartient au sous-espace de $C_{c,\asp}^\infty(\tilde{G})$ engendré par fonction de la forme $g - g^y$, où $g \in C_{c,\asp}^\infty(\tilde{G})$ et $y \in G(F)$.
    \item Pour tout $\tilde{\gamma} \in \Gamma_{\mathrm{reg}}(\tilde{G})^\mathrm{bon}$, on a $I_{\tilde{G}}(\tilde{\gamma}, f)=0$.
    \item Pour tout $\pi \in \Pi_-(\tilde{G})$, on a $\angles{\Theta_\pi, f}=0$.
    \item Pour tout $\pi \in \Pi_{\mathrm{temp},-}(\tilde{G})$, on a $\angles{\Theta_\pi, f}=0$.
  \end{enumerate}\renewcommand{\labelenumi}{\arabic{enumi}}
\end{theorem}
Rappelons que $\Theta_\pi$ siginifie le caractère de $\pi$.
\begin{proof}
  C'est clair que (a) entraîne (b),(c),(d) et que (c) entraîne (d). Vu la Proposition \ref{prop:int-orb-dense}, (b) entraîne (a). D'après le Corollaire \ref{prop:dist-support} appliqué au cas $M=G$, on voit que (d) entraîne (b). Cela achève la démonstration.
\end{proof}

\begin{remark}
  L'une des conséquences directes de la formule des traces locale simple est le rapport de l'orthogonalité, cf. \cite[\S 6]{Ar93}; les preuves sont identiques sauf que l'on utilise le théorème de Paley-Wiener invariant tempéré (le Théorème \ref{prop:PW-invariant}) au lieu de sa version $\tilde{K}$-finie à support compact. Cette théorie sera importante dans des applications de la formule des traces aux revêtements.
\end{remark}

\section{Corrections pour \texorpdfstring{\S}{}4.1 (ajouté en novembre 2024)}
Dans le début de \S 4.1 (dans «Un fait important à se rappeler...»), je prétend que tout caractère continu $G_{\mathrm{SC}}(F) \to \mathbb{C}^{\times}$ est trivial, mais ceci n'est pas vrai. En décomposant $G_{\mathrm{SC}}$ en des facteurs $F$-simples, des contre-exemples paraissent exactement en les facteurs de la forme $\mathrm{S}D^{\times}$ où $D$ est une $E$-algèbre à division centrale, le corps $E$ étant une extension finie séparable de $F$ (on omet la restriction de scalaires par $E|F$). On renvoie à l'Appendice A par J.-P.\ Labesse de E.\ Lapid et Z.\ Mao, \emph{A conjecture on Whittaker--Fourier coefficients of cusp forms}, Journal of Number Theory, Vol.~146 (2015), pp.448-505 pour les détails.

Le même problème paraît aussi dans l'article \emph{La formule des traces pour les revêtements de groupes réductifs connexes.\ I.\ Le développement géométrique fin.} Or dans \textit{loc.\ cit.}, le caractère est soit un caractère automorphe de $G_{\mathrm{SC}}(\mathbb{A}_F)$ pour un corps de nombres $F$, soit un produit de composantes locales d'un tel caractère; de tels caractères sont triviaux d'après J.-L.\ Waldspurger, \emph{Caractères automorphes d'un groupe réductif} (\href{https://arxiv.org/abs/1608.07150}{arXiv:1608.07150}).

Ces problèmes ne se posent pas si l'on se limite aux revêtements de $\mathrm{GL}(n, F)$, les groupes classiques ou les groupes de similitudes.

Revenons au cadre local non-archimédien de \S 4.1. L'hypothèse fausse ci-dessus sert à trivialiser le caractère $\bomega' := \bomega \circ (\pi, \iota)$ de $G'(F) = G_{\mathrm{SC}}(F) \times Z_{G^\circ}^{\circ}(F)$, afin de ramener les $\bomega$-intégrales orbitales sur $\mathfrak{g}$ aux intégrales orbitales usuelles sur $\mathfrak{g}' = \mathfrak{g}$. Ce qu'on obtient maintenant sont les $\bomega'$-intégrales orbitales sur $\mathfrak{g}'$. Observons que $\bomega'$ est d'ordre fini, et trivial sur $Z_{G^\circ}^\circ(F)$; l'action adjointe de $G(F)$ sur $G'(F)$ est bien définie, laissant $\bomega'$ invariante car $\bomega'$ provient de $\bomega: G(F) \to \mathbb{C}^{\times}$.

Par conséquent, l'espace $\mathcal{I}(\mathfrak{g}'(F))$ doit être remplacé par $\mathcal{I}(\mathfrak{g}'(F))_{\bomega'}$. On peut définir le groupe fini
\[ \Xi := G(F) \big/ \mathrm{Im} \left[ (\pi, \iota): \mathrm{Ker}(\bomega') \to G(F) \right], \]
et noter $\Pi(\Xi)$ l'ensemble de classes d'isomorphismes de représentations irréductibles de $\Xi$. Posons
\[ \Pi_{\bomega'}(\Xi) = \{\xi \in \Pi(\Xi): \xi|_{G_{\mathrm{SC}}(F)} = \bomega' \cdot \mathrm{id} \} \]
et observons que $\bomega \in \Pi_{\bomega'}(\Xi)$. L'action de $\Xi$ sur $\mathcal{I}(\mathfrak{g}'(F))_{\bomega'}$ est toujours induite par $f \mapsto f^{y^{-1}}$ où $y \in G(F)$. Idem pour $\mathcal{J}(\mathfrak{g}'(F))$. Dans les décompositions $\Xi$-isotypiques de $\mathcal{I}(\mathfrak{g}'(F))_{\bomega'}$ et $\mathcal{J}(\mathfrak{g}'(F))^{\bomega'}$, seulement les $\xi \in \Pi_{\bomega'}(\Xi)$ interviennent.

Avec ces modifications, les lemmes 4.1.1 et 4.1.2 restent valides. L'égalité dans la Remarque 4.1.3 doit être remplacée par
\[ J^{\bomega}_G(\dot{X}, f) = J^{\bomega'}_{G'}(\dot{X}, \mathrm{pr}_{\bomega} f), \quad f \in \mathcal{I}(\mathfrak{g}(F))_{\bomega}. \]

Les démonstrations du Théorème 4.1.4 ($\bomega$-germes de Shalika) et de la Proposition 4.1.5 (densité de $\bomega$-intégrales orbitales régulières) marchent comme suit. On se ramène aux variantes pour $\mathcal{I}(\mathfrak{g}'(F))_{\bomega'}$ et $\mathcal{J}(\mathfrak{g}'(F))^{\bomega'}$; ensuite, on se ramène aux facteurs $F$-simples de $G_{\mathrm{SC}}$ comme la théorie sur $Z_{G^\circ}^\circ$ est triviale. Or $\bomega'$ est non-trivial seulement sur facteurs de la forme $\mathrm{S}D^{\times}$, qui sont $F$-anisotropes, donc les assertions sont triviales sur ces facteurs. Pour les autres facteurs, les caractères sont triviaux et le résultat est connu.

Pour prouver la Proposition 4.1.7, la clef est toujours sa variante pour $\mathfrak{g}'$ avec un caractère $\bomega'$, et on se ramène toujours au cas de $\mathrm{S}D^{\times}$ avec un caractère, noté $\bomega_D$. Fixons $\mathrm{S}D^{\times}$ et un réseau convenable $L_D$ dans son algèbre de Lie. Le voisinage $V_D$ dans $\mathcal{J}(V_D, \infty, L_D)^{\bomega_D}$ est vide, ce qui entraîne que $\mathcal{J}(V_D, \infty, L_D)^{\bomega_D}$ est l'espace des $\theta$ qui sont à support compact quand restreintes aux fonctions $L_D$-invariantes. C'est exactement ce qu'on veut pour justifier la preuve.

Les Théorèmes 4.1.8--4.1.10 s'en découlent.

Il y d'autres erreurs mineurs à corriger, comme ci-dessous.

\begin{itemize}
	\item Avant le Lemme 4.1.2, la définition $D^G(X) = \det\left( \mathrm{ad}(X) \middle| \mathfrak{g}/\mathfrak{g}_X \right)$ convient seulement pour $X$ semi-simple. Il faut remplacer le $|D^G(X)|^{\frac{1}{2}}$ dans la définition de $J^{\bomega}_G(\dot{X}, f)$ plus haut par $|D^G(X_{\mathrm{ss}})|^{\frac{1}{2}}$, où $X_{\mathrm{ss}}$ désigne la partie semi-simple dans la décomposition de Jordan de $X$.
	
	\item Dans le point 4 du Théorème 4.1.4, il faut remplacer l'exposant $-\dim G/G_u$ de $|t|$ par $\dim \mathfrak{g}_{\mathrm{nil}} -\dim G/G_u$.
\end{itemize}